\documentclass[pdftex,a4paper,10pt,leqno]{article}
\usepackage{amsmath}
\usepackage{amssymb}
\usepackage{amsthm}
\usepackage{mathrsfs}
\usepackage{makeidx}
\usepackage{comment}
\usepackage{authblk}
\usepackage[nottoc]{tocbibind}
\usepackage{hyperref}
\hypersetup{colorlinks=false, linkcolor=red, citecolor= red, urlcolor=red}
\allowdisplaybreaks[1]

\newcommand{\mres}{\mathop{\vrule height 1.3ex depth 0pt width
0.12ex\vrule height 0.11ex depth 0pt width 1.1ex}}

\swapnumbers
\newtheorem{thm}{Theorem}[section]
\newtheorem{lem}[thm]{Lemma}
\newtheorem{prop}[thm]{Proposition}
\newtheorem{cor}[thm]{Corollary}
\theoremstyle{definition}
\newtheorem{defn}[thm]{Definition}
\newtheorem{sett}[thm]{Setting}
\newtheorem{exmp}[thm]{Example}
\newtheorem{rem}[thm]{Remark}
\numberwithin{equation}{section}

\begin{document}

\title{A new version of \\
Brakke's local regularity theorem}

\author{Ananda Lahiri\footnote{
Max Planck Institute for Gravitational Physics 
(Albert Einstein Institute)}}

\maketitle

\pagestyle{plain}
\setcounter{page}{1}
%
%
\begin{abstract}
Consider an integral Brakke flow $(\mu_t)$, $t\in [0,T]$, 
inside some ball in Euclidean space. If $\mu_{0}$ has small height, 
its measure does not deviate too much from that of a plane and if $\mu_{T}$ is non-empty, 
then Brakke's local regularity theorem yields that $(\mu_t)$ is actually smooth and graphical inside a smaller ball 
for times $t\in (C,T-C)$ for some constant $C$. 
Here we extend this result to times $t\in (C,T)$. 
The main idea is to prove that a Brakke flow that is initially locally graphical with small gradient will remain graphical for some time.
Moreover we use the new local regularity theorem to generalise White's regularity theorem to Brakke flows.

\end{abstract}
\tableofcontents

\section{Introduction}
\paragraph{Overview.}
Consider a smooth family of embeddings $F_t:M^{\mathbf{n}}\to\mathbb{R}^{\mathbf{n}+\mathbf{k}}$ 
for $t\in I$, where $M^{\mathbf{n}}$ is an $\mathbf{n}$-dimensional manifold and $I$ is an open interval in~$\mathbb{R}$.
The family $M_t=F_t[M^{\mathbf{n}}]$ is called a smooth mean curvature flow if
\begin{align}
\label{smoothmcf}
\left(\partial_t F_t(\hat{p})\right)^{\bot}=\mathbf{H}(M_t,F_t(\hat{p}))
\end{align}
for all $(t,\hat{p})\in I\times M^{\mathbf{n}}$.
Here $(\cdot)^{\bot}$ denoetes the projection onto the normal space
and $\mathbf{H}$ denotes the mean curvature vector.
This evolution equation can be generalised to $\mathbf{n}$-rectifiable Radon measures on $\mathbb{R}^{\mathbf{n}+\mathbf{k}}$
(see Definition \ref{brakkeflow}). 
Such a weak solution will be called a Brakke flow.
Here we want to show that under certain local assumptions
a Brakke flow locally satisfies the smooth characterization above.

The Mean curvature flow was introduced by Brakke in his pioneering work \cite{brakke},
where he described the evolution in the setting of geometric measure theory.
This early work already contains an existence result
as well as a regularity theory.
His local regularity theorem states that a Brakke flow 
that lies in a narrow slab and consists of one sheet 
is locally smooth and graphical for a certain time interval.
Actually the arguments in \cite{brakke} often contain gaps or small errors.
A new rigorous proof of the regularity results was given by Kasai and Tonegawa \cite{kasait}, \cite{tonegawa}.
Also the author's thesis \cite{lahiri} 
offers a complete version of Brakke's regularity theory
following the original approach.

A major breakthrough in the studies of mean curvature flow 
was the monotonicity formula
found by Huisken \cite{huisken2} for smooth flows,
which later was generalised to Brakke flows by Ilmanen \cite{ilmanen2}
and localised by Ecker \cite{eckerb}.
With the help of this monotonicity formula, 
White proved a local regularity theorem \cite{white}
stating that Gaussian density ratios close to one yield curvature estimates.
White's theorem is formulated for smooth mean curvature flow
and can be applied to many singular situations as well,
but not for arbitrary Brakke flows.
Building up on White's curvature estimates,
Ilmanen, Neves and Schulze \cite{ilnesch} showed
that a smooth mean curvature flow
which is locally initially graphical with small gradient
remains graphical for some time.
For related gradient and curvature estimates see
\cite{eckerh2}, \cite{eckerh1}, \cite{coldingm2}, \cite{wang}, \cite{chenyin}, \cite{brendleh},
\cite{lahiri2}.

Existence results for generalized solutions of mean curvature flow can be found in
\cite{brakke}, \cite{chengg}, \cite{evanss1}, \cite{evanss4}, \cite{ilmanen1}, \cite{kimt}.
For an introduction to weak mean curvature flow
we recommend the work of Ilmanen \cite{ilmanen1}
which also points out the similarities between Brakke flow and level set flow.
We also want to mention the book of Ecker \cite{eckerb}
as a good reference for smooth mean curvature flow and regularity up to the first singular time.

\paragraph{Results of the present article.}
We consider Brakke flows of $\mathbf{n}$-rectifiable Radon measures in $\mathbb{R}^{\mathbf{n}+\mathbf{k}}$ 
see Definition \ref{brakkeflow} for the details.
All Brakke flows considered here are assumed to be integral.
The constants below may depend on $\mathbf{n}$ and $\mathbf{k}$.

Our main result is a new version of 
Brakke's local regularity theorem \cite[Thm.\ 6.10, Thm.\ 6.11]{brakke},
see also Kasai and Tonegawa \cite[Thm.\ 8.7]{kasait}.
Our  statement says that a non-vanishing Brakke flow 
which initially locally lies in a narrow slab
and consists of less than two sheets,
becomes graphical in a small neighbourhood.

%
%
%
%
\begin{thm}
\label{locregthm}
There exists an $\alpha_0\in (0,1)$
and for every $\lambda\in (0,1)$ there exists a $\gamma_0\in (0,1)$
such that the following holds:

Let $\gamma\in [0,\gamma_0]$, $\rho\in (0,\infty)$, 
$t_1\in\mathbb{R}$, $t_2\in (t_1+\gamma^{\alpha_0}\rho^2,t_1+\alpha_0\rho^2]$,
$a=(\hat{a},\tilde{a})\in\mathbb{R}^{\mathbf{n}}\times\mathbb{R}^{\mathbf{k}}$
and let $(\mu_t)_{t\in [t_1,t_2]}$ be a Brakke flow in $\mathbf{B}(a,2\rho)$
with $a\in\mathrm{spt}\mu_{t_2}$.
Suppose we have 
\begin{align}
\label{locregthma}
\mathrm{spt}\mu_{t_1}\cap\mathbf{B}(a,2\rho)
&\subset\{(\hat{x},\tilde{x})\in\mathbb{R}^{\mathbf{n}}\times\mathbb{R}^{\mathbf{k}}:|\tilde{x}-\tilde{a}|\leq\gamma\rho\},
\\
\label{locregthmb}
\rho^{-\mathbf{n}}\mu_{t_1}(\mathbf{B}(a,\rho))
&\leq (2-\lambda)\omega_{\mathbf{n}}
\end{align}
and set $I:=(t_1+\gamma^{\alpha_0}\rho^2,t_2)$.

Then there exists an 
$u\in\mathcal{C}^{\infty}\left(I\times\mathbf{B}^{\mathbf{n}}(\hat{a},\gamma_0\rho),\mathbb{R}^{\mathbf{k}}\right)$
such that
\begin{align*}
\mu_{t}\mres\mathbf{C}(a,\gamma_0\rho,\rho)=\mathscr{H}^{\mathbf{n}}\mres\mathrm{graph}(u(t,\cdot\,)) 
\quad \text{for all } t\in I.
\end{align*}
Moreover $\sup |Du(t,\cdot\,)|\leq\alpha_0^{-1}\rho^{-2}(t-t_1)$ for all $t\in I$
and $F_t(\hat{x})=(\hat{x},u(t,\hat{x}))$ satisfies \eqref{smoothmcf}.
\end{thm}
\noindent
The main difference to previous versions is that here we obtain regularity up to the time $t_2$ at which we assumed the non-vanishing, 
were in Brakke's original theorem measure bounds from below have to be assumed for later times. 

We also obtain a local regularity theorem similar to the one of White \cite{white}, 
see also Ecker's version \cite[Thm.\ 5.6]{eckerb}. 
We show that a non-vanishing Brakke flow which locally has Gaussian density ratios close to one will become graphical in a small neighbourhood.
%
%
%
%
\begin{thm}
\label{whitelocregthm}
For every $\beta\in (0,1)$ there exists an $\eta\in (0,1)$
such that the following holds:

Let $\rho\in (0,\infty)$, $\rho_0\in [\rho,\infty)$,
$t_0\in\mathbb{R}$,
$a\in\mathbb{R}^{\mathbf{n}+\mathbf{k}}$
and let $(\mu_t)_{t\in [t_0-\rho^2,t_0]}$ 
be a Brakke flow in $\mathbf{B}(a,4\sqrt{\mathbf{n}}\rho_0)$.
Suppose $a\in\mathrm{spt}\mu_{t_0}$ and
for all $(s,y)\in(t_0-\rho^2,t_0]\times\mathbf{B}(a,\rho)$
we have
\begin{align}
\label{whitelocregthma}
\int_{\mathbb{R}^{\mathbf{n}+\mathbf{k}}}\Phi_{(s,y)}\varphi_{(s,y),\rho_0}
\;\mathrm{d}\mu_{t_0-\rho^2}
\leq 1+\eta,
\end{align}
where $\Phi$ and $\varphi$ are from Definition \ref{sphtestfct}.
Set $I:=(t_0-\eta^2\rho^2,t_0)$.

Then there exist $S\in\mathbf{O}(\mathbf{n}+\mathbf{k})$ 
and $u\in\mathcal{C}^{\infty}(I\times\mathbf{B}^{\mathbf{n}}(0,\eta\rho),\mathbb{R}^{\mathbf{n}+\mathbf{k}})$
such that for $M_t=\mathrm{graph}(u(t,\cdot\,))$ we have
\begin{align*}
\mu_{t}\mres\mathbf{B}(a,\eta\rho)
=\mathscr{H}^{\mathbf{n}}
\mres\left(S[M_t]+a\cap\mathbf{B}(a,\eta\rho)\right)
\;\;\;\text{for all}\;\; t\in I.
\end{align*}
Moreover $\sup |Du|\leq\beta$
and $F_t(\hat{x})=(\hat{x},u(t,\hat{x}))$ satisfies \eqref{smoothmcf}.
\end{thm}

One key ingredient in order to obtain these regularity results is the observation
that a non-vanishing Brakke flow
which is initially graphical with small gradient
will stay graphical for some time.
This is the non-smooth version 
of the corresponding theorem by Ilmanen, Neves and Schulze \cite[Thm.\ 1.5]{ilnesch}.
%
%
%
%
\begin{thm}
\label{staygraphthm}
There exists an $l_0\in (0,1)$
such that the following holds:

Let $l\in [0,l_0]$, 
$\rho\in (0,\infty)$,
$t_1\in\mathbb{R}$, $t_2\in (t_1,t_1+l_0\rho^2]$,
$a=(\hat{a},\tilde{a})\in\mathbb{R}^{\mathbf{n}}\times\mathbb{R}^{\mathbf{k}}$
and let $(\mu_t)_{t\in [t_1,t_2]}$ be a Brakke flow in $\mathbf{C}(a,2\rho,2\rho)$.
Assume  $a\in\mathrm{spt}\mu_{t_1}$ and
\begin{align}
\label{staygraphthma}
\mathrm{spt}\mu_{t_2}\cap\mathbf{C}(a,\rho,\rho)\neq\emptyset.
\end{align}
Suppose there exists an $u_0\in\mathcal{C}^{0,1}\left(\mathbf{B}^{\mathbf{n}}(\hat{a},2\rho),\mathbb{R}^{\mathbf{k}}\right)$
with $\mathrm{lip}(u_0)\leq l$ 
and
\begin{align}
\label{staygraphthmb}
\mu_{t_1}\mres\mathbf{C}(a,2\rho,2\rho)=\mathscr{H}^{\mathbf{n}}\mres\mathrm{graph}(u_0).
\end{align}

Then there exists an 
$u\in\mathcal{C}^{\infty}\left((t_1,t_2)\times\mathbf{B}^{\mathbf{n}}(\hat{a},\rho),\mathbb{R}^{\mathbf{k}}\right)$
such that
\begin{align*}
\mu_{t}\mres\mathbf{C}(a,\rho,\rho)=\mathscr{H}^{\mathbf{n}}\mres\mathrm{graph}(u(t,\cdot\,))
\quad\text{for all}\;\; t\in (t_1,t_2).
\end{align*}
Moreover $\sup |Du(t,\cdot\,)|\leq\sqrt[4]{l}+\rho^{-2}(t-t_1)$ for all $t\in (t_1,t_2)$
and $F_t(\hat{x})=(\hat{x},u(t,\hat{x}))$ satisfies \eqref{smoothmcf}.
\end{thm}
\begin{rem}
In all the above results $F_t(\hat{x})=(\hat{x},u(t,\hat{x}))$ satisfies \eqref{smoothmcf}
and $|Du|$ is small,
thus the results for smooth graphical mean curvature flow can be applied to obtain 
bounds on the curvature of $\mathrm{graph}(u(t,\cdot\,))$ and higher derivatives of it, 
see Appendix \ref{curvature_estimates}, in particular Proposition \ref{higherregprop} and Lemma \ref{wangcurveestlem}.
This also yields estimates on the higher derivatives of $u$.
 
Note that in the above results 
we cannot expect to obtain a graphical representation at the final time,
see Example \ref{masslossrem}.
\end{rem}
%
%
%
%
To deal with the potentail vanishing of Brakke flows we 
also use some sort of continuation result for varifolds.
We show that a unit density varifold with absolutely continuous first variation
that is contained in a Lipshitz graph is either empty or coincides with the graph.
This generalizes Allard's constancy theorem \cite[Thm.\ 4.6.(3)]{allard}
(see also Simon \cite[Thm.\ 8.4.1]{simon})
to Lipschitz graphs, but additionally requires unit density.
\begin{thm}
\label{nobdrythm}
Let $D\subset\mathbb{R}^{\mathbf{n}}$ be open and connected
with $(\mathbf{n}-1)$-rectifiable boundary $\partial D=\bar{D}\setminus D$
and set $U:=D\times\mathbb{R}^{\mathbf{k}}$.
Consider a unit density $\mathbf{n}$-rectifiable Radon measure $\mu$
and a Lipschitz function $f:D\to\mathbb{R}^{\mathbf{k}}$ such that
\begin{align}
\label{nobdrylema}
&\emptyset\neq\mathrm{spt}\mu\cap U\subset\mathrm{graph}f
\\
\label{nobdrylemb}
&\mu(A)=0\;\text{implies}\;\|\delta \mu\|(A)=0\;\text{for all}\;A\subset U.
\end{align}

Then $\mu\mres U=\mathscr{H}^{\mathbf{n}}\mres\mathrm{graph}f$.
\end{thm}
Note that Duggan \cite[Thm.\ 4.5]{duggan} shows constant density
for stationary $\mathbf{n}$-rectifiable Radon measures
which already equal a Lipschitz graph.
\paragraph{Organisation and sketch of the proof}
We start by recalling some definitions and important results in Section \ref{preliminaries}.

Then in Section \ref{graphs_without_holes} we show Theorem \ref{nobdrythm}.
In the proof we employ the Gauss--Green theorem by Federer \cite[\S 4.5.6]{federer}
to see that the projection of $\mathrm{spt}\mu$ onto $\mathbb{R}^{\mathbf{n}}$ is stationary in $D$
and subsequently the result follows from Allard's constancy theorem \cite[Thm.\ 4.6.(3)]{allard}.

Section \ref{height_bound} establishes that a local height bound is maintained under Brakke flow, see Proposition \ref{improvedheightbndprop}.
Here we improve the height bound increases in time from the known linear growth to arbitrary high exponents.

The main part of this work is Section \ref{stay_graph},
where Theorem \ref{staygraphthm} is established.
Essentially we consider a Brakke flow in $\mathbf{C}(0,2,2)$
for times in $[0,\tau]$ such that 
$\mathrm{spt}\mu_{\tau}\cap\mathbf{C}(0,\delta,1)\neq\emptyset$.
First assume as initial condition that
$\mu_{0}\mres\mathbf{C}(0,2,2)$ lies in a slab of height $h$
and satisfies certain density ratio assumptions.
Based on Brakke's local regularity theorem \cite[Thm.\ 6.11]{brakke}
and the height bound from Section \ref{height_bound}
we show 
that the flow is graphical inside $\mathbf{C}(0,h,1)$
for times in $[C h^2,\tau- C h^2]$
if $h$ is small enough, $\delta\leq h$ and $\tau\leq\sqrt[32]{h}$.
Under stronger density assumptions we actually obtain
graphical representability inside $\mathbf{C}(0,1,1)$
for times in $[C h,\tau-C h]$,
see Lemma \ref{improvedgohlem}.

Now Change the initial condition to
$\mu_{0}\mres\mathbf{C}(0,2,2)$ being graphical 
with Lipschitz constant smaller than $l$.
This allows applying Lemma \ref{improvedgohlem} on arbitrary small scales,
which yields that the flow is graphical inside $\mathbf{C}(0,1,1)$
for times in $[0,\tau-C l]$ if $l$ is small enough, $\delta\leq l$ and $\tau\leq\sqrt[32]{l}$.
Iterating this result leads to Lemma \ref{staygraphlem} which says 
that the flow is graphical inside $\mathbf{B}(0,L\delta)$
for times in $[0,\tau-\delta^2]$ if we choose $l$ small enough depending on $L$
and suppose $\delta\leq l$, $\tau\leq l$.
Here the curvature estimates from Appendix \ref{curvature_estimates} 
for smooth mean curvature flow are crucial to control the gradient durng the iteration.
Using Lemma \ref{staygraphlem} with varying center points and arbitrary small $\delta$
we obtain that $\mathrm{spt}\mu_t\cap\mathbf{C}(0,1,1)$
is contained in a Lipschitz graph 
and has unit density for almost all $t\in [0,\tau]$.
In view of Theorem \ref{nobdrythm} 
this lets us conclude Theorem~\ref{staygraphthm}.

Section \ref{local_regularity} contains the proof of Theorem \ref{locregthm}.
First we see that Theorem~\ref{staygraphthm} and Lemma \ref{improvedgohlem}
directly imply a version of Theorem \ref{locregthm},
which assumes stronger density bounds initially, see Lemma \ref{becomegraphlem}.
Then we use Brakke's cylindrical growth theorem \cite[Lem.\ 6.4]{brakke} 
to simplify these assumptions,
which establishes Theorem \ref{locregthm} in the desired form.

In Section \ref{white_regularity} Theorem \ref{whitelocregthm} is proven.
In order to do so we first employ Huisken's monotonicty formula \cite[Thm.\ 3.1]{huisken2}
to show that non-moving planes are the only Brakke flows in $\mathbf{R}^{\mathbf{n}+\mathbf{k}}$
that have Gaussian density ratios bounded everywhere by $1$.
Then under the assumptions of Theorem \ref{whitelocregthm}
a blow up argument and Ilmanen's compactness theorem
yield that in a small neighbourhood the conditions of Theorem \ref{locregthm}
are satisfied, which yields the conclusion of Theorem \ref{whitelocregthm}.

In Appendix \ref{curvature_estimates} we show some curvature estiamtes for smooth mean curvature flow.
Most of the presented results are already known despite small variations.
Finally Appendix \ref{appendix} contains some remarks.

\paragraph{Acknowledgements.}
I want to thank Ulrich Menne for his help and advice,
in particular for the proof of Theorem \ref{nobdrythm}.
Also I want to thank Matthew Langford and Yoshihiro Tonegawa
for some helpfull discussions.

\section{Preliminaries}
\label{preliminaries}
%
%
\paragraph{Notation.}
For an excellent introduction to geometric measure theory we recommend the lecture notes by Simon \cite{simon}.
Here we recall the most important definitions.
\begin{itemize}
%
%
\item 
We set $\mathbb{R}^{+}:=\{x\in \mathbb{R},x\geq 0\}$,
$\mathbb{N}:=\{1,2,3,\ldots\}$
and $(a)_+:=\max\{a,0\}$ for $a\in\mathbb{R}$.

%
%
\item 
We fix $\mathbf{n},\mathbf{k}\in\mathbb{N}$.
Quantities that only depend on $\mathbf{n}$ and/or $\mathbf{k}$ are considered constant.
Such a constant may be denoted by $C$ or $c$,
in particular the value of $C$ and $c$ may change in each line.

%
%
\item
For $a\in\mathbb{R}^{\mathbf{n}+\mathbf{k}}$
the values $\hat{a}\in\mathbb{R}^{\mathbf{n}}$
and $\tilde{a}\in\mathbb{R}^{\mathbf{k}}$ 
are given by $a=(\hat{a},\tilde{a})$.

%
%
\item
We denote the canonical basis of $\mathbb{R}^{\mathbf{n}+\mathbf{k}}$ by $(\mathbf{e}_i)_{1\leq i\leq\mathbf{n}+\mathbf{k}}$.
In particular the canonical basis of $\mathbb{R}^{\mathbf{n}}$ is $(\hat{\mathbf{e}}_i)_{1\leq i\leq\mathbf{n}}$.

\end{itemize}
Consider $n,k\in\mathbb{N}$.
\begin{itemize}
%
%
\item
Let $\mathbf{O}(n)$ denote the space of rotations on $\mathbb{R}^{n}$.
Let $\mathbf{G}(n+k,n)$ denote the space of $n$-dimensional subspaces of $\mathbb{R}^{n+k}$.
For $T\in\mathbf{G}(n+k,n)$
set $T^{\bot}:=\{x\in\mathbb{R}^{n+k}:\;x\cdot v=0\quad\forall v\in T\}$.
By $T_{\natural}:\mathbb{R}^{n+k}\to T$ we denote the projection onto $T$.

%
%
\item
For $R,r,h\in (0,\infty)$ and $a\in\mathbb{R}^n$ we set
\begin{align*}
\mathbf{B}^n(a,R):=\left\{x\in\mathbb{R}^{n}:|x-a|<R\right\},
\;\;\;\mathbf{B}(a,R):=\mathbf{B}^{\mathbf{n}+\mathbf{k}}(a,R),
\\
\mathbf{C}(a,r,h):=\mathbf{B}^{\mathbf{n}}(\hat{a},r)\times\mathbf{B}^{\mathbf{k}}(\tilde{a},h),
\;\;\;
\mathbf{C}(a,r):=\mathbf{B}^{\mathbf{n}}(\hat{a},r)\times\mathbb{R}^{\mathbf{k}}.
\end{align*}

%
%
\item
Consider open sets $I\subset\mathbb{R}$ and $V\subset\mathbb{R}^n$.
For $f\in\mathcal{C}^{1}(I\times V)$ we denote by $\partial_t f$ the derivative of $f$ in $I$,
while $Df$ denotes the derivative of $f$ in~$V$.
If $(\mu_t)_{t\in I}$ is a family of Radon measures on $V$ we often abbreviate
$\int_Vf(t,x)\;\mathrm{d}\mu_t(x)=\int_Vf\;\mathrm{d}\mu_t$.

%
%
\item
Let $\mathscr{L}^n$ denote the $n$-dimensional Lebesque measure
and $\mathscr{H}^n$ denote the $n$-dimensional Hausdorf measure.
Set $\omega_n:=\mathscr{L}^n(B^n(0,1))$.
\end{itemize}
%
%
%
%
Consider an open set $U\subset\mathbb{R}^{n+k}$ be open
and a Radon measure $\mu$ on $U$.
\begin{itemize}
%
%
\item
Set $\mathrm{spt}\mu:=\{x\in U:\;\mu(\mathbf{B}^{n+k}(x,r))>0,\;\text{for all}\;r\in (0,\infty)\}$.
%
%
\item
Consider $x\in U$.
We define the upper and lower density by
\begin{align*}
\Theta^{*n}(\mu,x):=\limsup_{r\searrow 0}\frac{\mu(\mathbf{B}^{n+k}(x,r))}{\omega_n r^{n}},
\;\;\;
\Theta^{n}_*(\mu,x):=\liminf_{r\searrow 0}\frac{\mu(\mathbf{B}^{n+k}(x,r))}{\omega_n r^{n}}
\end{align*}
and if both coincide 
the value is denoted by $\Theta^{n}(\mu,x)$ and called the density of $\mu$ at $x$.

%
%
\item
Consider $y\in U$.
If there exist $\theta(y)\in (0,\infty)$ and $\mathbf{T}(\mu,y)\in\mathbf{G}(n+k,n)$
such that
\begin{align*}
\lim_{\lambda\searrow 0}\lambda^{-n}\int_U\phi(\lambda^{-1}(x-y))\,\mathrm{d}\mu(x)
=\theta(y)\int_{\mathbf{T}(\mu,y)}\phi(x)\,\mathrm{d}\mathscr{H}^n(x)
\end{align*}
for all $\phi\in\mathcal{C}^0_{\mathrm{c}}\left(\mathbb{R}^{n+k}\right)$,
then $\mathbf{T}(\mu,y)$ is called the ($n$-dimensional) approximate tangent space of $\mu$ at $x$
with multiplicity $\theta(y)$.

%
%
\item
We say $\mu$ is $n$-rectifiable,
if the approximate tangent space exists at $\mu$-a.e. $x\in U$.
Note that in this case $\theta(x)=\Theta^{n}(\mu,x)$ for $\mu$-a.e. $x\in U$.
We say $\mu$ is integer $n$-rectifiable,
if $\mu$ is $n$-rectifiable and $\Theta^{n}(\mu,x)\in\mathbb{N}$ for $\mu$-a.e. $x\in U$.
We say $\mu$ has unit density,
if $\mu$ is $n$-rectifiable and $\Theta^{n}(\mu,x)=1$ for $\mu$-a.e. $x\in U$.
\end{itemize}
%
%
%
%
Let $\mu$ be an $n$-rectifiable Radon measure on $U$
\begin{itemize}
%
%
\item
Consider $\phi\in\mathcal{C}^1(U,\mathbb{R}^{n+k})$.
For $x\in U$ such that $\mathbf{T}(\mu,x)$ exists
set $\mathrm{div}_{\mu}\phi(x):=\sum_{i=1}^n (D_{b_i}\phi(x))\cdot b_i$,
where $(b_i)_{1\leq i\leq n}$ is an orthonormal basis of $\mathbf{T}(\mu,x)$.
%
%
\item
Denote the first variation of $\mu$ in $U$ by
$\delta\mu(\phi):=\int_U\mathrm{div}_{\mu}\phi\;\mathrm{d}\mu$
for $\phi\in\mathcal{C}_{\mathrm{c}}^1(U,\mathbb{R}^{n+k})$.
Set $\|\delta\mu\|(A):=\sup\{\partial\mu(\phi), \phi\in\mathcal{C}_{\mathrm{c}}^1(A,\mathbb{R}^{n+k}), |\phi|\leq 1\}$
for $A\subset U$ open.
%
%
\item
If there exists $\mathbf{H}_{\mu}:\mathrm{spt}\mu\to\mathbb{R}^{n+k}$ such that
$\mathbf{H}_{\mu}$ is locally $\mu$-integrable and
$\delta\mu(\phi)=\int_U\mathbf{H}_{\mu}\cdot\phi\;\mathrm{d}\mu$
for all $\phi\in\mathcal{C}_{\mathrm{c}}^1(U,\mathbb{R}^{n+k})$,
then $\mathbf{H}_{\mu}$ is called the generalised mean curvature vector of $\mu$ in $U$.
\end{itemize}

\paragraph{Brakke flow.}
An introduction to the Brakke flow can be found in \cite{brakke}, \cite{ilmanen1}, \cite{kasait} or \cite{lahiri}.
%
%
\begin{defn}
\label{brakkevar}
For a Radon measure $\mu$ on $\mathbb{R}^{\mathbf{n}+\mathbf{k}}$
and a $\phi\in\mathcal{C}^1_{\mathrm{c}}(\mathbb{R}^{\mathbf{n}+\mathbf{k}})$
we define the Brakke variation $\mathscr{B}(\mu,\phi)$ as follows:
If $\mu\mres\{\phi>0\}$ is $\mathbf{n}$-rectifiable,
has generalised mean curvature vector $\mathbf{H}_{\mu}$ in $\{\phi>0\}$
and ${\int_{\{\phi>0\}}|\mathbf{H}_{\mu}|^2\,\mathrm{d}\mu <\infty}$ then set
\begin{align*}
\mathscr{B}(\mu,\phi):=\int_{\mathbb{R}^{\mathbf{n}+\mathbf{k}}}
\left((\mathbf{T}(\mu,x)^{\bot})_{\natural}D\phi(x)\cdot\mathbf{H}_{\mu}(x)-\phi(x)|\mathbf{H}_{\mu}(x)|^2\right)
\,\mathrm{d}\mu(x).
\end{align*}
Else we set $\mathscr{B}(\mu,\phi):=-\infty$.
Note that in case $\mu$ is integer $\mathbf{n}$-rectifiable,
by a deep theorem of Brakke \cite[Thm.\ 5.8]{brakke},
we have $\mathbf{H}_{\mu}(x)\perp\mathbf{T}(\mu,x)$ for $\mu$-a.e. $x\in\mathbb{R}^{\mathbf{n}+\mathbf{k}}$.
Hence in this case the projection can be left out.
\end{defn}
%
%
\begin{rem}[{\cite[Prop.\ 3.4]{brakke}},{\cite[Lem.\ 6.6]{ilmanen1}}]
\label{brakkevarbound}
If $\phi\in\mathcal{C}^2_{\mathrm{c}}(\mathbb{R}^{\mathbf{n}+\mathbf{k}})$ 
and $\mathscr{B}(\mu,\phi)> -\infty$ we can estimate
\begin{align*}
\mathscr{B}(\mu,\phi)
&\leq
\sup|D^2\phi|\;\mu(\{\phi>0\})-\frac{1}{2}\int_{\mathbb{R}^{\mathbf{n}+\mathbf{k}}}|\mathbf{H}_{\mu}|^2\phi\,\mathrm{d}\mu.
\end{align*}
\end{rem}

%
%
\begin{defn}
\label{brakkeflow}
Let $U\subset\mathbb{R}^{\mathbf{n}+\mathbf{k}}$ be open, $t_1\in\mathbb{R}$, $t_2\in (t_1,\infty)$
and $(\mu_t)_{t\in [t_1,t_2]}$ be a family of radon measures on $\mathbb{R}^{\mathbf{n}+\mathbf{k}}$.
We call $(\mu_t)_{t\in [t_1,t_2]}$ a Brakke flow in $U$ 
if $\mu_t\mres U$ is integer $\mathbf{n}$-rectifiable for a.e.\ $t\in (t_1,t_2)$
and for all $t_1\leq s_1<s_2\leq t_2$ we have
\begin{align}
\label{brakkeflowa}
\mu_{s_2}(\phi(s_2,\cdot\,))-\mu_{s_1}(\phi(s_1,\cdot\,))
\leq
\int_{s_1}^{s_2}\left(\mathscr{B}(\mu_t,\phi(t,\cdot\,))+\mu_{t}(\partial_t\phi(t,\cdot\,))\right)\,\mathrm{d}t
\end{align}
for all $\phi\in\mathcal{C}^1((s_1,s_2)\times U)\cap\mathcal{C}^0([s_1,s_2]\times U)$ 
with $\cup_{t\in [s_1,s_2]}\mathrm{spt}\phi(t,\cdot\,)\subset\subset U$.
\end{defn}
%
%
\begin{rem}
Suppose $(\mu_t)_{t\in [t_1,t_2]}$ is a Brakke flow in $U$
then we have:
\begin{itemize}
\item
For a.e.\ $t\in (t_1,t_2)$ we have:
$\mu_t\mres U$ is integer $\mathbf{n}$-rectifiable, 
has generalised mean curvature vector $\mathbf{H}_{\mu_t}$ in $U$
and $\int_{K}|\mathbf{H}_{\mu_t}|^2\,\mathrm{d}\mu_t <\infty$ 
for all $K\subset\subset U$.
\item
For $(s_0,y_0)\in\mathbb{R}\times\mathbb{R}^{\mathbf{n}+\mathbf{k}}$ 
and $r\in (0,\infty)$ 
set $\nu_t(A):=r^{-\mathbf{n}}\mu_{r^2t+s_0}(r^{}A+y_0)$,
then $(\nu_t)_{t\in [r^{-2}(t_1-s_0),r^{-2}(t_2-s_0)]}$ 
is a Brakke flow in $r^{-1}(U-y_0)$.
\end{itemize}
\end{rem}
%
%
The Brakke flow allows the sudden loss of mass.
In particular we have
\begin{exmp}
\label{masslossrem}
For $0<t_0\leq T$ and $0<\epsilon<\rho<\infty$
consider the Brakke flow $(\mu_t)_{t\in [0,T]}$ given by
$\mu_t=\mathscr{H}^{\mathbf{n}}\mres(\mathbb{R}^{\mathbf{n}}\times\{0\}^{\mathbf{k}})$
for $t\in [0,t_0)$,
$\mu_{t_0}=\mathscr{H}^{\mathbf{n}}\mres(\mathbf{B}^{\mathbf{n}}(0,\epsilon)\times\{0\}^{\mathbf{k}})$
and $\mu_t:=\emptyset$ for $t\in (t_0,T]$.
Note that $\mu_t$ is graphical with Lipschitz constant $0$ for $t\in [0,t_0)$
and $0\in\mathrm{spt}\mu_{t_0}$ but $\mu_{t_0}\mres\mathbf{B}(0,\rho)$
is not graphical. 
\end{exmp}

\paragraph{Important results}
Here we recall some important results that are crucial for the proofs in this article.
%
%
\begin{lem}[Measure bound {\cite[Thm.\ 3.7]{brakke}},{\cite[Prop.\ 4.9]{eckerb}}]
\label{barrierlem}
Let $R\in (0,\infty)$, $t_1\in\mathbb{R}$, $t_2\in (t_1,t_1+(2\mathbf{n})^{-1} R^2]$, 
$z_0\in\mathbb{R}^{\mathbf{n}+\mathbf{k}}$
and let $(\mu_t)_{t\in [t_1,t_2]}$ be a Brakke flow in $\mathbf{B}(z_0,2R)$.

Then for all $t\in [t_1,t_2]$ we have
\begin{align*}
\mu_t\left(\mathbf{B}(z_0,R)\right)
\leq 8\mu_{t_1}\left(\mathbf{B}(z_0,2R)\right).
\end{align*}
\end{lem}

%
%
%
%
\begin{defn}
\label{sphtestfct}
Let $(t_0,x_0)\in\mathbb{R}\times\mathbb{R}^{\mathbf{n}+\mathbf{k}}$ and $\rho\in (0,\infty)$ be fixed.
Set
\begin{align*}
\Phi_{(t_0,x_0)}(t,x)&:=\left(4\pi(t_0-t)\right)^{-\frac{\mathbf{n}}{2}}\exp\left(\frac{|x-x_0|^2}{4(t-t_0)}\right).
\\
\varphi_{(t_0,x_0),\rho}(t,x)&:=\left(1-\rho^{-2}\left(|x-x_0|^2 + 2\mathbf{n}(t-t_0)\right)\right)_+^3
\end{align*}
for $(t,x)\in \mathbb{R}\times\mathbb{R}^{\mathbf{n}+\mathbf{k}}$,
where $\Phi_{(t_0,x_0)}(t,x)$ is only defined for $t\leq t_0$.
\end{defn}
%
%
%
%
\begin{thm}[Monotonicity formula {\cite[Thm.\ 3.1]{huisken2},\cite[Thm.\ 7]{ilmanen2}\cite[Thm.\ 4.13]{eckerb}}]
\label{localmonthm}
Consider an open set $U\subset\mathbb{R}^{\mathbf{n}+\mathbf{k}}$,
$\rho,D\in (0,\infty)$, $(t_0,x_0)\in \mathbb{R}\times U$, 
$s_1\in(-\infty,t_0)$ and $s_2\in (s_1,t_0)$
and let $\left(\mu_t\right)_{t\in [s_1,s_2]}$ be a Brakke flow in $U$.
Assume one (or both) of the following holds
\begin{enumerate}
\item 
$\mathrm{spt}\varphi_{(t_0,x_0),\rho}(s_1,\cdot\,)\subset\subset U$.
\item
$U=\mathbb{R}^{\mathbf{n}+\mathbf{k}}$ and 
$\sup_{t\in[s_1,s_2]}\sup_{R\in (0,\infty)}\mu_{t}(\mathbf{B}(x_0,R))
\leq D R^{\mathbf{n}}$.
\end{enumerate}

Then
\begin{align*}
\begin{split}
&\int_{U}\Phi(s_2,x)\,\mathrm{d}\mu_{s_2}(x)
- \int_{U}\Phi(s_1,x)\,\mathrm{d}\mu_{s_1}(x)
\\&\leq 
\int_{s_1}^{s_2}\int_{\mathbb{R}^{\mathbf{n}+\mathbf{k}}}\left(\Phi(t,x)
\left|\mathbf{H}_{\mu_t}(x)+\frac{(\mathbf{T}(\mu_t,x)^{\bot})_{\natural}(x-x_0)}{2(t_0-t)}\right|^2\right)
\,\mathrm{d}\mu_t(x)\,\mathrm{d}t
\end{split}
\end{align*}
for $\Phi=\Phi_{(t_0,x_0)}\varphi_{(t_0,x_0),\rho}$ if assumption $1$ holds
and $\Phi=\Phi_{(t_0,x_0)}$ if assumption $2$ holds.
Here the term under the time integral is interpreted as $-\infty$ 
at times where one of the technical conditions fails, 
as in Definition \ref{brakkevar}.
\end{thm}

%
%
%
%
\begin{lem}[Clearing out {\cite[Lem.\ 6.3]{brakke}}]
\label{clearoutlem}
There exist $C\in (1,\infty)$ and $\alpha_1:=(\mathbf{n}+6)^{-1}$
such that the following holds:

Let $\eta\in [0,\infty)$, $R\in (0,\infty)$, 
$t_1\in\mathbb{R}$, $t_2\in (t_1+C\eta^{2\alpha_1}R^2,t_1+(4\mathbf{n})^{-1}R^2)$,
$x_0\in\mathbb{R}^{\mathbf{n}+\mathbf{k}}$.
Let $U\subset\mathbb{R}^{\mathbf{n}+\mathbf{k}}$ be open with $U\supset\supset\mathbf{B}(x_0,R)$
and let $(\mu_t)_{t\in [t_1,t_2]}$ be a Brakke flow in $\mathbf{B}(x_0,R)$.
Suppose we have
\begin{align*}
R^{-\mathbf{n}}\int_{U}(1-R^{-2}|x-x_0|^2)_{+}^3\;\mathrm{d}\mu_{t_1}\leq\eta.
\end{align*}
Set $R(t):=\sqrt{R^2-4\mathbf{n}(t-t_0)}$.

Then for all $t\in [t_1+C\eta^{2\alpha_1}R^2,t_2]$
we have $\mu_t(\mathbf{B}(x_0,R(t)))=0.$
\end{lem}

%
%
%
%
\begin{thm}[Local regularity {\cite[Thm.\ 6.11]{brakke}, \cite[Thm.\ 8.7]{kasait}, \cite[Thm.\ 8.4]{lahiri}}]
\label{Blocalregthm}
For every $\lambda\in (0,1]$ 
there exist $\Lambda\in (1,\infty)$ and $h_0\in (0,1)$
such that the following holds:

Let $h\in (0,h_0]$, $R\in (0,\infty)$, 
$t_1\in\mathbb{R}$, $t_2\in (t_1+2\Lambda R^2,\infty)$,
$x_0\in\mathbb{R}^{\mathbf{n}+\mathbf{k}}$,
and let $(\mu_t)_{t\in [t_1,t_2]}$ be a Brakke flow in $\mathbf{B}(x_0,4R)$.
Suppose we have
\begin{align}
\label{Blocalregthma}
\mathrm{spt}\mu_t\cap\mathbf{B}(x_0,4R)
&\subset\mathbf{C}(x_0,4R,hR)
\\
\label{Blocalregthmb}
R^{-\mathbf{n}}\mu_t\left(\mathbf{B}(x_0,4R)\right)
&\leq\lambda^{-1}
\end{align}
for all $t\in [t_1,t_2]$ and
\begin{align}
\label{Blocalregthmc}
R^{-\mathbf{n}}\mu_{t_1}\left(\mathbf{B}(x_0,R\right)
&\leq (2-\lambda)\omega_{\mathbf{n}}
\\
\label{Blocalregthmd}
R^{-\mathbf{n}}\mu_{t_2}\left(\mathbf{B}(x_0,R/2)\right)
&\geq \lambda\omega_{\mathbf{n}}.
\end{align}
Set $I:=(t_1+\Lambda R^2,t_2-\Lambda R^2)$.

Then there exists an
$u\in\mathcal{C}^{\infty}(I\times\mathbf{B}^{\mathbf{n}}(\hat{x}_0,h_0 R),\mathbb{R}^{\mathbf{k}})$
such that
\begin{align*}
\mu_t\mres\mathbf{C}(x_0,h_0 R,R)
=\mathscr{H}^{\mathbf{n}}\mres\mathrm{graph}(u(t,\cdot\,))
\quad \text{for all } t\in I.
\end{align*}
Moreover $\sup|Du|\leq\Lambda h$
and $F_t(\hat{x})=(\hat{x},u(t,\hat{x}))$ satisfies \eqref{smoothmcf}.
\end{thm}
\noindent
To deduce this result from \cite[Thm.\ 6.11]{brakke}, \cite[Thm.\ 8.7]{kasait}
you also need to use \cite[Lem.\ 6.6]{brakke}, \cite[Thm.\ 5.7]{kasait}
to see that the density ratio bounds~\eqref{Blocalregthmc} and \eqref{Blocalregthmd}
actually hold at all times in slightly weaker form.
Note that Brakke as well as Kasai and Tonegawa
state this theorem for unit density Brakke flows,
though their proofs only use integer density.
For the smoothness of $u$ we refer to \cite[Thm.\ 3.6]{tonegawa}.
%
%
%
%
\begin{thm}[Compactness {\cite[Thm.\ 7.1]{ilmanen1}}]
\label{compactnessthm}
Let $t_1\in\mathbb{R}$ and $t_2\in (t_1,\infty)$.
For all $i\in\mathbb{N}$ consider an open set
$U_i\subset\mathbb{R}^{\mathbf{n}+\mathbf{k}}$
and a Brakke flow $(\mu_t^i)_{t\in [t_1,t_2]}$ in~$U_i$.
Assume $U_i\subset U_{i+1}$ for all $i\in\mathbb{N}$
and set $U:=\bigcup_{i=1}^{\infty}U_i$.
Suppose for every $K\subset\subset U$ there exists an $C_K\in (1,\infty)$ such that
\begin{align*}
\sup_{i\in\mathbb{N}}\sup_{t\in [t_1,t_2]}\mu_t^i(K\cap U_i)\leq C_K.
\end{align*}

Then there exists a subsequence $\sigma:\mathbb{N}\to\mathbb{N}$
and a Brakke flow $(\mu_t)_{t\in [t_1,t_2]}$ in~$U$ such that
\begin{align*}
\mu_{t}(\phi)=\lim_{j\to\infty}\mu^{\sigma(j)}_{t}(\phi)
\quad\text{for all}\;\phi\in\mathcal{C}^0_{\mathrm{c}}(U_{j_0})
\end{align*}
for all $t\in [t_1,t_2]$ and all $j_0\in\mathbb{N}$.
\end{thm}
%
%
\noindent
Actually in \cite{ilmanen1} Ilmanen assumes $U_i\equiv M$, 
for a complete manifold $M$.
To derive the above result from \cite[Thm.\ 7.1]{ilmanen1} 
use a diagonal subsequence argument, see Remark \ref{compactnessproofrem}
for some more details.

%
%
%
%
\begin{lem}[Tilt-bound {\cite[Lem.\ 5.5]{brakke}}]
\label{tiltboundlem}
There exists a $C\in(0,\infty)$ such that 
the following holds:

Let $U\subset\mathbb{R}^{\mathbf{n}+\mathbf{k}}$ be open
and let $\mu$ be an integer $\mathbf{n}$-rectifiable Radon measure on $U$
with locally $\mathcal{L}^2$-integrable mean curvature vector $\mathbf{H}_{\mu}$.
Consider $g\in\mathcal{C}^1_{\mathrm{c}}\left(U,\mathbb{R}\right)$, 
$f,h\in\mathcal{C}^0_{\mathrm{c}}\left(U,\mathbb{R}\right)$ with $g^2\leq fh$
and set
\begin{align*}
\alpha_f^2&:=\int_{U}|\mathbf{H}_{\mu}(x)|^2 f(x)^2\,\mathrm{d}\mu(x),
\\
\beta_g^2&:=\int_{U}\left\|(\mathbb{R}^{\mathbf{n}}\times\{0\}^{\mathbf{k}})_{\natural}-\mathbf{T}(\mu,x)_{\natural}\right\|^2g(x)^2
\,\mathrm{d}\mu(x),
\\
\gamma_h^2&:=\int_{U}|\tilde{x}|^2h(x)^2\,\mathrm{d}\mu(x),
\\
\xi_g^2&:=\int_{U}|\tilde{x}|^2 |\nabla^{\mu}g(x)|^2\,\mathrm{d}\mu(x).
\end{align*}

Then we have $\beta_g^2\leq C\left(\alpha_f\gamma_h + \xi_g^2\right).$
\end{lem}

%
%
%
%
\begin{thm}[Cylindrical growth {\cite[Thm.\ 6.4]{brakke}}]
\label{cylgrowththm}
Let $U\subset\mathbb{R}^{\mathbf{n}+\mathbf{k}}$ be open,
$R_1\in (0,\infty)$, $R_2\in (R_1,\infty)$, $\alpha,\beta\in [0,\infty)$.
Let $\mu$ be an integer $\mathbf{n}$-rectifiable Radon measure on $U$ 
with $\mathcal{L}^2$-integrable mean curvature vector $\mathbf{H}_{\mu}$
and $\mathrm{spt}\mu\cap\mathbf{C}(x_0,R_2)\subset\subset U$.
Consider $\psi\in\mathcal{C}^{3}_{\mathrm{c}}\left([-1,1],\mathbb{R}^+\right)$.
Suppose for all $r\in [R_1,R_2]$ we can estimate
\begin{align}
\label{cylgrowththma}
r^{-\mathbf{n}}\int_{U}
|\mathbf{H}_{\mu}(x)|^2\psi(r^{-1}|\hat{x}|)\;\mathrm{d}\mu(x)
\leq \alpha^2,
\\
\label{cylgrowththmb}
r^{-\mathbf{n}}\int_{U}
\left\|(\mathbb{R}^{\mathbf{n}}\times\{0\}^{\mathbf{k}})_{\natural}-\mathbf{T}(\mu,x)_{\natural}\right\|^2\psi(r^{-1}|\hat{x}|)\;\mathrm{d}\mu(x)
\leq \beta^2.
\end{align}

Then we have
\begin{align*}
\begin{split}
\left|R_2^{-\mathbf{n}}\int_{U}\psi(R_2^{-1}|\hat{x}|)\,\mathrm{d}\mu(x)
-R_1^{-\mathbf{n}}\int_{U}\psi(R_1^{-1}|\hat{x}|)\,\mathrm{d}\mu(x)\right|
\\\leq
\mathbf{n}\log(R_2/R_1)\beta^2+(R_2-R_1)\alpha\beta+\beta^2.
\end{split}
\end{align*}
\end{thm}

\section{Graphs without holes}
\label{graphs_without_holes}
In this section we prove Theorem \ref{nobdrythm}.
Consider a unit density Radon measure $\mu$ such that the first variation $\delta\mu$
is absolutely continuous with respect to $\mu$.
In some sense this should imply that $\mu$ has no `boundary points'.
Here we show that,
if such a $\mu$ is contained in the graph of some Lipschitz function $f$,
then $\mu$ actually coincides with the measure generated by the graph of $f$.
For $f\in\mathcal{C}^2$ and stationary $\mu$ this is a direct consequence 
of Allard's constancy theorem \cite[Thm.\ 4.6.(3)]{allard} 
(see also Simon's notes \cite[Thm.\ 8.4.1]{simon}).
Here we use the generalized Gauss-Green theorem by Federer \cite[\S 4.5.6]{federer}
to show that the projection of $\mu$ onto 
$\mathbb{R}^{\mathbf{n}}\times\{0\}^{\mathbf{k}}$
is stationary,
which reduces our problem to the $\mathcal{C}^2$-setting,
thus implying the result.

%
%
\begin{defn}
Let $\mu$ be an $n$-rectifiable Radon measure on $\mathbb{R}^{n+k}$.
We denote the associated general varifold by $\mathbf{V}(\mu)$, i.e.\
$\mathbf{V}(\mu)$ is the Radon measure on 
$\mathbb{R}^{n+k}\times\mathbf{G}(n+k,n)$
given by
\begin{align*}
\mathbf{V}(\mu)(A):=\mu(\{x\in\mathbb{R}^{n+k}:(x,\mathbf{T}(\mu,x))\in A\})
\quad\text{for}\; A\subset\mathbb{R}^{n+k}.
\end{align*}
For $y\in\mathbb{R}^{n+k}$ and $\lambda\in (0,\infty)$ 
we define the $\lambda$-blow-up around $y$ by
\begin{align*}
\mu_{y,\lambda}(A)
:=\lambda^{-n}\mu(\lambda A+y)
\quad\text{for}\; A\subset\mathbb{R}^{n+k}.
\end{align*}
\end{defn}
%
%
\begin{proof}[Proof of Theorem {\ref{nobdrythm}}]
This proof is based on ideas by Ulrich Menne.
Define $F:D\to U$ by $F(\hat{x}):=(\hat{x},f(\hat{x}))$.
Set 
\begin{align*}
U_1&:=\{x\in U:\; 
\Theta^{\mathbf{n}-1}(\|\delta\mu\|,x)=0\},
\\
Q_1&:=\{x\in U:\; \Theta_*^{\mathbf{n}}(\mu,x)\geq 1\},
\;\;\;
Q_2:=Q_1\cap U_1,
\\
R_1&:=\{x\in U:\; \Theta^{\mathbf{n}}(\mu,x)=0\},
\;\;\;
R_2:=R_1\cap U_1.
\end{align*}
We claim
\begin{align}
\label{nobdrylem21}
\mathscr{H}^{\mathbf{n}-1}(U\setminus (Q_2\cup R_2))=0.
\end{align}
As we have absolutely continuous first variation
we can use a result by Menne \cite[Rem.\ 2.11]{menne1} 
to see $\mathscr{H}^{\mathbf{n}-1}(U\setminus (Q_1\cup R_1))=0$.
Hence, to establish the claim it remains to show
\begin{align}
\label{nobdrylem22}
\mathscr{H}^{\mathbf{n}-1}(U\setminus U_1)=0.
\end{align}
We proceed as Federer and Ziemer \cite[\S 8]{federerz}.
For $i\in\mathbb{N}$ set
\begin{align*}
B_i=\{x\in U\cap\mathbf{B}(0,i): \Theta^{*\mathbf{n}-1}(\|\delta\mu\|,x)>i^{-1}\}.
\end{align*}
Then by Federer \cite[\S 2.10.19(3)]{federer} 
we have $i\|\partial\mu\|(B_i)\geq\mathscr{H}^{\mathbf{n}-1}(B_i)$
for all $i\in\mathbb{N}$.
This and \eqref{nobdrylemb} yield the following chain of implications:
$B_i$ bounded,
$\|\partial\mu\|(B_i)<\infty$,
$\mathscr{H}^{\mathbf{n}-1}(B_i)<\infty$,
$\mathscr{H}^{\mathbf{n}}(B_i)=0$,
$\mu (B_i)=0$,
$\|\partial\mu\|(B_i)=0$,
$\mathscr{H}^{\mathbf{n}-1}(B_i)=0$.
This shows \eqref{nobdrylem22} which completes the proof of \eqref{nobdrylem21}.

Now set 
\begin{align*}
p&:=(\mathbb{R}^{\mathbf{n}}\times\{0\}^{\mathbf{k}})_{\natural},
\quad
A_0:=p[\mathrm{spt}\mu]\cap D,
\\
Q_0&:=\{\hat{x}\in \mathbb{R}^{\mathbf{n}}:\; \Theta^{\mathbf{n}}(\mathscr{L}^{\mathbf{n}}\mres(\mathbb{R}^{\mathbf{n}}\setminus A_0),\hat{x})=0\},
\\
R_0&:=\{\hat{x}\in \mathbb{R}^{\mathbf{n}}:\; \Theta^{\mathbf{n}}(\mathscr{L}^{\mathbf{n}}\mres A_0,\hat{x})=0\}.
\end{align*}
We want to use
\begin{align}
\label{nobdrylem33}
p[Q_2]\subset Q_0
\;\text{and}\;
p[R_2\cap\mathrm{graph}f]\subset R_0.
\end{align}
We will prove this statement later.
Suppose \eqref{nobdrylem33} holds,
then
\begin{align*}
&\mathscr{H}^{\mathbf{n}-1}(D\setminus (Q_0\cup R_0))
\leq\mathscr{H}^{\mathbf{n}-1}(F[D\setminus (Q_0\cup R_0)])
\\
&\leq\mathscr{H}^{\mathbf{n}-1}(F[D\setminus (p[Q_2]\cup p[R_2\cap\mathrm{graph}f])])
=\mathscr{H}^{\mathbf{n}-1}(\mathrm{graph}f\setminus (Q_2\cup R_2)).
\end{align*}
Hence by \eqref{nobdrylem21} we have
\begin{align}
\label{nobdrylem34}
\mathscr{H}^{\mathbf{n}-1}(D\setminus (Q_0\cup R_0))=0.
\end{align}

We say $\hat{v}\in\partial\mathbf{B}^{\mathbf{n}}(0,1)$ is an external normal of $A_0$ at $\hat{y}\in\mathbb{R}^{\mathbf{n}}$, if
\begin{align*}
\Theta^{\mathbf{n}}(\mathscr{L}^{\mathbf{n}}\mres \{\hat{x}\in\mathbb{R}^{\mathbf{n}}:(\hat{x}-\hat{y})\cdot\hat{v}>0\}\cap A_0,\hat{y})&=0
\\ \text{and}\;
\Theta^{\mathbf{n}}(\mathscr{L}^{\mathbf{n}}\mres \{\hat{x}\in\mathbb{R}^{\mathbf{n}}:(\hat{x}-\hat{y})\cdot\hat{v}<0\}\setminus A_0,\hat{y})&=0,
\end{align*}
Let $B_0$ be the set consisting of all $\hat{y}\in\mathbb{R}^{\mathbf{n}}$
for which there exists an external normal of $A_0$ at $\hat{y}$.
Then we have
\begin{align}
\label{nobdrylem35}
B_0\cap(Q_0\cup R_0)=\emptyset.
\end{align}
To see this consider $\hat{y}\in Q_0$ and $\hat{v}\in\partial\mathbf{B}^{\mathbf{n}}(0,1)$.
We can estimate
\begin{align*}
&\mathscr{L}^{\mathbf{n}}(\{\hat{x}\in\mathbb{R}^{\mathbf{n}}:(\hat{x}-\hat{y})\cdot\hat{v}>0\}\cap A_0\cap\mathbf{B}^{\mathbf{n}}(\hat{y},r))
\\&\geq 
\mathscr{L}^{\mathbf{n}}(\{\hat{x}\in\mathbb{R}^{\mathbf{n}}:(\hat{x}-\hat{y})\cdot\hat{v}>0\}\cap\mathbf{B}^{\mathbf{n}}(\hat{y},r))
-\mathscr{L}^{\mathbf{n}}((\mathbb{R}^{\mathbf{n}}\setminus A_0)\cap\mathbf{B}^{\mathbf{n}}(\hat{y},r))
\\&\geq (2^{-1}\omega_{\mathbf{n}}-\epsilon) r^{\mathbf{n}}
\end{align*}
for $r$ small enough depending on $\epsilon$.
This yields $B_0\cap Q_0=\emptyset$.
Similarly we can show $B_0\cap R_0=\emptyset$,
which proves \eqref{nobdrylem35}.

Let $K\subset\mathbb{R}^{\mathbf{n}}$ be compact.
Using \eqref{nobdrylem34}, $R_0\supset (\mathbb{R}^{\mathbf{n}}\setminus\bar{D})$ and the rectifiability of $\partial D=\bar{D}\setminus D$ we obtain
\begin{align*}
\mathscr{H}^{\mathbf{n}-1}(K\setminus (Q_0\cup R_0))
&\leq
\mathscr{H}^{\mathbf{n}-1}((K\setminus D)\setminus R_0)
&\leq
\mathscr{H}^{\mathbf{n}-1}(\partial D\cap K)
<\infty.
\end{align*}
In view of Federer \cite[\S 4.5.11]{federer} and \cite[\S 2.10.6]{federer}
we can now use the general Gauss-Green theorem (see Federer \cite[\S 4.5.6]{federer}).
Combined with \eqref{nobdrylem34} and \eqref{nobdrylem35} this establishes
\begin{align*}
\int_{A_0}\mathrm{div}_{\mathbb{R}^{\mathbf{n}}}\phi
\,\mathrm{d}\mathscr{L}^{\mathbf{n}}
\leq\mathscr{H}^{\mathbf{n}-1}(D \cap B_0)
\leq\mathscr{H}^{\mathbf{n}-1}(D \setminus (Q_0\cup R_0))=0
\end{align*}
for all $\phi\in\mathcal{C}^{1}_{\mathrm{c}}(D,\mathbf{B}^{\mathbf{n}}(0,1))$.
Thus $A_0$ is stationary in $D$.
Then the constancy theorem (see Allard \cite[Thm.\ 4.6.(3)]{allard} or Simon \cite[Thm.\ 8.4.1]{simon}) yields $A_0=D$
which establishes the result.
Hence it remains to prove \eqref{nobdrylem33}.

We want to show $p[R_2\cap\mathrm{graph}f]\subset R_0$.
Consider $y\in R_2\cap\mathrm{graph}f$.
By \eqref{nobdrylema} and as $\mu$ is integer $\mathbf{n}$-rectifible we can estimate for $r\in (0,\infty)$
\begin{align*}
r^{-\mathbf{n}}\mathscr{L}^{\mathbf{n}}(A_0\cap\mathbf{B}^{\mathbf{n}}(\hat{y},r))
&=
r^{-\mathbf{n}}\int_{\mathrm{graph}f\cap (A_0\times\mathbb{R}^{\mathbf{k}})\cap\mathbf{C}(y,r)}|JF\circ p|^{-1}\,\mathrm{d}\mathscr{H}^{\mathbf{n}}
\\&\leq
r^{-\mathbf{n}}\mu(\mathbf{B}(y,(1+\mathrm{lip}(f))r))
\end{align*}
and as $y\in R_2$ this goes to $0$ for $r\searrow 0$.
Thus $\hat{y}\in R_0$.

It remains to show 
$p[Q_2]\subset Q_0$.
Suppose this is false,
then there exists a $y_0\in Q_2$, an $\epsilon\in (0,1)$
and a sequence $(r_m)_{m\in\mathbb{N}}$ with $r_m\searrow 0$ such that
\begin{align}
\label{nobdrylem51}
r_m^{-\mathbf{n}}\mathscr{L}^{\mathbf{n}}(\mathbf{B}^{\mathbf{n}}(\hat{y}_0,r_m)\setminus A_0)
>2\epsilon
\end{align}
for all $m\in\mathbb{N}$.
Consider the sequence $(\mu_m)_{m\in\mathbb{N}}$ given by $\mu_m=\mu_{y_0,r_m}$.
By \eqref{nobdrylema}, unit density and as $y_0\in U_1$ we have
\begin{align*}
\limsup_{m\to\infty}\mu_m(\mathbf{B}(0,R))
=\limsup_{m\to\infty}r_m^{-\mathbf{n}}\mu(\mathbf{B}(y_0,Rr_m))
\leq (1+C\;\mathrm{lip}f)^{\mathbf{n}}R^{\mathbf{n}},
\\
\limsup_{m\to\infty}\|\delta \mu_m\|(\mathbf{B}(0,R))
=\lim_{m\to\infty}r_m^{-\mathbf{n}-1}\|\delta \mu\|(\mathbf{B}(y_0,Rr_m))=0
\end{align*}
for every $R\in (0,\infty)$. 
By varifold compactness (see Allard \cite[Thm.\ 6.4]{allard} or Simon \cite[Thm.\ 8.5.5]{simon})
there exists a stationary integer $\mathbf{n}$-rectifiable Radon measure $\nu$ with $0\in\mathrm{spt}\nu$
and such that for a subsequence we have 
\begin{align}
\label{nobdrylem52}
\mathbf{V}(\mu_m)\rightharpoonup \mathbf{V}(\nu)
\;\;\text{as radon measures on}\;\mathbb{R}^{\mathbf{n}+\mathbf{k}}\times\mathbf{G}(\mathbf{n}+\mathbf{k},\mathbf{n})
\end{align}
Define $f_m\in\mathcal{C}^{0,1}(\mathbf{B}^{\mathbf{n}}(0,3),\mathbb{R}^{\mathbf{k}})$
by $f_m(\hat{x}):=r_m^{-1}\left(f(r_m\hat{x}+\hat{y}_0)-f(\hat{y}_0)\right)$.
By the Arzela-Ascoli theorem there exists a $g\in\mathcal{C}^{0,1}(\mathbf{B}^{\mathbf{n}}(0,3),\mathbb{R}^{\mathbf{k}})$
such that for a subsequence $\|f_m-g\|_{C^0}\to 0$.
We claim
\begin{align}
\label{nobdrylem55}
\mathrm{spt}\nu\cap\mathbf{C}(0,3)&\subset\mathrm{graph}(g)
\\
\label{nobdrylem56}
\mathrm{spt}\nu\cap\mathbf{C}(0,1)&=\mathrm{graph}(g)\cap\mathbf{C}(0,1).
\end{align}

Suppose there exists a $z\in\mathrm{spt}\nu\cap\mathbf{C}(0,3)\setminus\mathrm{graph}(g)$.
Then we find $\rho\in (0,1)$ 
with $\mathbf{B}(z,4\rho)\cap\mathrm{graph}(g)=\emptyset$ 
and $\nu(\mathbf{B}(z,\rho))>0$.
Thus for some large enough $m\in\mathbb{N}$ we have
$\mathbf{B}(z,3\rho)\cap\mathrm{graph}(f_m)=\emptyset$ 
and $\mu_m(\mathbf{B}(z,2\rho))>0$.
But by definition of $f_m$ and $\mu_m$ combined with \eqref{nobdrylema}
we also have $\mathrm{spt}\mu_m\subset\mathrm{graph}(f_m)$,
which yields a contradiction.
Thus \eqref{nobdrylem55} holds.

Now suppose there exists a $z\in\mathrm{graph}(g)\cap\mathbf{C}(0,1)\setminus\mathrm{spt}\nu$.
Define
\begin{align*}
\rho_0:=
\inf\left\{\rho\in (0,\infty):\;\mathbf{C}(z,\rho)\cap\mathrm{spt}\nu\neq\emptyset\right\}>0
\end{align*}
Note that $\rho_0< 1$ as $0\in\mathrm{spt}\nu$.
Then there exists a $z_0\in\partial\mathbf{C}(z,\rho_0)\cap\mathrm{spt}\nu$
such that $\mathbf{C}(z,\rho_0)\cap\mathrm{spt}\nu=\emptyset.$
Consider the sequence $(\nu_l)_{l\in\mathbb{N}}$ given by $\nu_l=\nu_{z_0,r_l}$.
As above there exists a stationary integer $\mathbf{n}$-rectifiable Radon measure 
$\eta$ with $0\in \mathrm{spt}\eta$
and such that for a subsequence we have 
\begin{align}
\label{nobdrylem62}
\mathbf{V}(\nu_l)\rightharpoonup\mathbf{V}(\eta)
\quad\text{as radon measures on}\;\mathbb{R}^{\mathbf{n}+\mathbf{k}}\times\mathbf{G}(\mathbf{n}+\mathbf{k},\mathbf{n}).
\end{align}
Similar as above we also see
\begin{align}
\label{nobdrylem63}
\mathrm{spt}\eta\cap\mathbf{C}(0,3)\subset\mathrm{graph}(h)
\quad &\text{for some}\;
h\in\mathcal{C}^{0,1}(\mathbf{B}^{\mathbf{n}}(0,3),\mathbb{R}^{\mathbf{k}})
\end{align}
As $\nu$ is stationary and integer $\mathbf{n}$-rectifiable we know
$\Theta^{\mathbf{n}}(\nu,z_0)\in [1,\infty)$ (see \cite[Cor.\ 4.3.3]{allard}).
This yields that $\eta$ satisfies (see \cite[Cor.\ 8.5.4]{allard})
\begin{align*}
\eta_{0,\lambda}=\eta
\quad&\text{for all}\;\lambda>0,
\end{align*}
in particular $x\in\mathbf{T}(\eta,x)$ for $\eta$-a.\ e.\ $x\in\mathbb{R}^{\mathbf{n}+\mathbf{k}}$.
By choice of $z_0$ and \eqref{nobdrylem62} we also have
$\mathrm{spt}\eta\subset\{x\in\mathbb{R}^{\mathbf{n}+\mathbf{k}}:
\hat{x}\cdot(\hat{z}-\hat{z}_0)\leq 0\}$,
thus Lemma \ref{halfspacelem} yields $\mathrm{spt}\eta\subset\{x\in\mathbb{R}^{\mathbf{n}+\mathbf{k}}:
\hat{x}\cdot(\hat{z}-\hat{z}_0)= 0\}$.
In view of \eqref{nobdrylem63} this implies $\mathscr{H}^{\mathbf{n}}(\mathrm{spt}\eta)=0$,
hence $\mathrm{spt}\eta=\emptyset$,
which contradicts $0\in \mathrm{spt}\eta$.
This proves \eqref{nobdrylem56}.

We continue to lead \eqref{nobdrylem51} to a contradiction.
Using \eqref{nobdrylema} and the unit density of $\mu$ we can estimate
\begin{align*}
&\mathscr{L}^{\mathbf{n}}(A_0\cap\mathbf{B}^{\mathbf{n}}(\hat{y}_0,r_m))
=
\int_{\mathrm{graph}(f)\cap (A_0\times\mathbb{R}^{\mathbf{k}}\cap\mathbf{C}(y_0,r_m))}|JF\circ p|^{-1}\,\mathrm{d}\mathscr{H}^{\mathbf{n}}
\\ \geq &
\int_{\mathbf{C}(y_0,r_m)}|\Lambda_{\mathbf{n}}\mathbf{T}(\mu,x)_{\natural}|^{-1}\,\mathrm{d}\mu(x)
=
r_m^{\mathbf{n}}\int_{\mathbf{C}(y_0,1)}|\Lambda_{\mathbf{n}}S_{\natural}|^{-1}\,\mathrm{d}\mathbf{V}(\mu_m)(x,S).
\end{align*}
Recall $\epsilon$ from \eqref{nobdrylem51}.
In view of \eqref{nobdrylem52} and \eqref{nobdrylem56} we obtain
\begin{align*}
r_m^{-\mathbf{n}}\mathscr{L}^{\mathbf{n}}(A_0\cap\mathbf{B}^{\mathbf{n}}(\hat{y}_0,r_m))+\epsilon
&\geq 
\int_{\mathbf{C}(0,1)}|\Lambda_{\mathbf{n}}S_{\natural}|^{-1}\,\mathrm{d}\mathbf{V}(\nu)(x,S)\\
&\geq 
\int_{\mathrm{graph}(g)\cap\mathbf{C}(0,1)}|JG\circ p|^{-1}\,\mathrm{d}\mathscr{H}^{\mathbf{n}}
=\omega_{\mathbf{n}}
\end{align*}
for $m$ large enough, where $G(\hat{z}):=(\hat{z},g(\hat{z}))$.
Thus we see
\begin{align*}
r_m^{-\mathbf{n}}\mathscr{L}^{\mathbf{n}}(\mathbf{B}^{\mathbf{n}}(\hat{y}_0,r_m)\setminus A_0)
=\omega_{\mathbf{n}}
-r_m^{-\mathbf{n}}\mathscr{L}^{\mathbf{n}}(A_0\cap\mathbf{B}^{\mathbf{n}}(\hat{y}_0,r_m))
\leq \epsilon,
\end{align*}
which contradicts \eqref{nobdrylem51}.
This completes the proof of \eqref{nobdrylem33}, which establishes the result.
\end{proof}

%
%
\begin{lem}
\label{halfspacelem}
Consider a stationary $\mathbf{n}$-rectifiable Radon measure $\eta$ in 
$\mathbb{R}^{\mathbf{n}+\mathbf{k}}$
with $x\in\mathbf{T}(\eta,x)$ for $\eta$-a.e.\ $x\in\mathbb{R}^{\mathbf{n}+\mathbf{k}}$.
Suppose
$\mathrm{spt}\eta\subset\{x\in\mathbb{R}^{\mathbf{n}+\mathbf{k}}:x\cdot w\leq 0\}$
for some $w\in\mathbb{R}^{\mathbf{n}+\mathbf{k}}$.
Then
$\mathrm{spt}\eta
\subset\{x\in\mathbb{R}^{\mathbf{n}+\mathbf{k}}:x\cdot w= 0\}$.
\end{lem}
%
%
\begin{proof}
Let $0<R_1<R_2<\infty$ and 
consider $\psi\in\mathcal{C}^{\infty}(\mathbb{R},[0,1])$ non-increasing
with $\psi(s)=1$ on $(-\infty,R_1/2]$, $\psi(s)=0$ on $[2R_2,\infty)$ and $\psi'(s)\leq (R_1-R_2)/2$ on $[R_1,R_2]$.
Define $\zeta\in\mathcal{C}^{\infty}_{\mathrm{c}}(\mathbb{R}^{\mathbf{n}+\mathbf{k}},[0,1])$
by $\zeta(x)=\psi(|x|)$.
As $\eta$ is stationary and $x\in\mathbf{T}(\eta,x)$ for $\eta$-a.e.\ $x\in\mathbb{R}^{\mathbf{n}+\mathbf{k}}$
we can calculate
\begin{align*}
0=\int_{\mathbb{R}^{\mathbf{n}+\mathbf{k}}}\mathrm{div}_{\eta}(\zeta w)\,\mathrm{d}\eta
&=\int_{\mathbb{R}^{\mathbf{n}+\mathbf{k}}}\mathbf{T}(\eta,x)_{\natural}(D\zeta)\cdot w\,\mathrm{d}\eta
=\int_{\mathbb{R}^{\mathbf{n}+\mathbf{k}}}\psi'(|x|)\frac{x\cdot w}{|x|}\,\mathrm{d}\eta
\end{align*}
By $\psi'\leq 0$ and our assumptions on $\mathrm{spt}\eta$ the integrant is always positive,
moreover $\psi'(s)< 0$ on $[R_1,R_2]$, thus 
$x\cdot w=0$ for $\eta$-a.e.\ $x\in\mathbf{B}(0,R_2)\setminus\mathbf{B}(0,R_1)$.
This establishes the result.
\end{proof}

\section{Improved height bound}
\label{height_bound}
For a Brakke flow initial local height bounds in a certain direction
yield weaker height bounds at later times in a decreased region.
Such results can be obtained from sphere comparisson,
see Brakke \cite[Thm.\ 3.7]{brakke} or Colding and Minicozzi \cite[Lem.\ 3]{coldingm2}.
Here we prove stronger height bounds for short times
using directly the Brakke flow equation.
%
%
\begin{prop}
\label{improvedheightbndprop}
For every $p\in \mathbb{N}$ there exists a $C_p\in (1,\infty)$ 
such that the following holds:

Let $\rho\in (0,\infty)$, $t_1\in\mathbb{R}$, $t_2\in (t_1,\infty)$,
$a\in\mathbb{R}^{\mathbf{n}+\mathbf{k}}$, $v\in\mathbb{R}^{\mathbf{n}+\mathbf{k}}$
and let $(\mu_t)_{t\in [t_1,t_2]}$ be a Brakke flow in $\mathbf{B}(a,2\rho)$.
Suppose
\begin{align*}
\mathrm{spt}\mu_{t_1}\cap\mathbf{B}(a,2\rho)
\subset\{x\in\mathbb{R}^{\mathbf{n}+\mathbf{k}}:(x-a)\cdot v\leq 0\mathbb\}.
\end{align*}

Then for all $t\in [t_1,t_2]$  we have
\begin{align*}
\mathrm{spt}\mu_{t}\cap\mathbf{B}(a,\rho)
\subset\{x\in\mathbb{R}^{\mathbf{n}+\mathbf{k}}:(x-a)\cdot v\leq C_p|v|(t-t_1)^p\rho^{1-2p}\mathbb\}.
\end{align*}
\end{prop}
%
%
\begin{proof}
We will actually prove the result for all $p\in\mathbb{N}\cup\{0\}$.
For $p=0$ the statement directly follows from
$\mathbf{B}(a,\rho)\subset\{(x-a)\cdot v\leq\rho|v|\}$ with $C_0=1$.
Suppose the statement is true for some $p\in \mathbb{N}\cup\{0\}$.
We want to show the statement holds for $p+1$.
We may assume $\rho=2$, $t_1=0$, $a=0$ 
and $v=\mathbf{e}_{\mathbf{n}+\mathbf{k}}$.
Set $N:=\mathbf{n}+\mathbf{k}-1$, $t_0:=\min\{2^{-7}\mathbf{n}^{-1},t_2\}$.
In case $t_0<t_2$ for $t\in (t_0,t_2]$ the conclusion directly follows from
$\mathbf{B}(0,2)\subset\{x\cdot\mathbf{e}_{\mathbf{n}+\mathbf{k}}\leq 2\}$ for $C_{p+1}$ large enough.
Hence we are only interested in $t\in [0,t_0]$.
We claim
\begin{align}
\label{improvedheightbndprop12}
\mathrm{spt}\mu_{t}\cap\mathbf{B}(0,2)
\subset\left(\mathbf{B}^N(\bar{x}_0,0)\times (-2,1/4]\right)
\end{align}
for all $t\in[0,t_0]$.
In order to see this consider $z_0=(\bar{x},h)\in\mathbf{B}(0,2)$ with $h> 1/4$.
By assumption we have $\mathrm{spt}\mu_0\cap\mathbf{B}(z_0,1/4)=\emptyset$.
Then Lemma \ref{barrierlem} with $R=1/8$ implies $z_0\notin\mathrm{spt}\mu_t$
for all $t\in [0,t_0]$.

Fix an arbitrary $\bar{x}_0\in\mathbf{B}^{N}(0,2)$.
By the induction assumption we see
\begin{align}
\label{improvedheightbndprop11}
\mathrm{spt}\mu_{t}\cap\mathbf{B}((\bar{x}_0,0),1)
\subset\left(\mathbf{B}^N(\bar{x}_0,0)\times (-1,C_{p}t^p)\right)
\end{align}
for all $t\in[0,t_0]$.
Consider the function
$\eta\in\mathcal{C}^{0,1}(\mathbb{R}\times\mathbb{R}^N\times\mathbb{R},\mathbb{R}^+)
\cup\mathcal{C}^{\infty}(\{\eta>0\})$
given by
\begin{align*}
\eta(t,\bar{x},h):=\left(h-8C_p|\bar{x}-\bar{x}_0|^2t^p-C_{p+1}t^{p+1}\right)_+.
\end{align*}
Treat $\eta$ as a function on $\mathbb{R}\times\mathbb{R}^{\mathbf{n}+\mathbf{k}}$.
Using $\mathrm{div}_{\mu_t}((\mathbb{R}^N\times\{0\})_{\natural})\leq\mathbf{n}$
and choosing~$C_{p+1}$ large enough we can estimate
\begin{align}
\label{improvedheightbndprop21}
(\partial_t-\mathrm{div}_{\mu_t}D)\eta^3(t,x)
\leq 3\eta^2(t,x)(16C_p\mathbf{n}t^p-(p+1)C_{p+1}t^p)
\leq 0
\end{align}
for all $(t,x)\in ([0,t_2]\times\mathrm{spt}\mu_t)\cap \{\eta>0\}$
at which the approximate tangent space exists.

Let $\chi\in\mathcal{C}^{\infty}(\mathbb{R}^{\mathbf{n}+\mathbf{k}},[0,1])$
be a cut-off function such that 
\begin{align*}
\mathbf{B}((\bar{x}_0,0),1/2)
\subset\{\chi=1\}
\subset\mathrm{spt}\chi
\subset\mathbf{B}((\bar{x}_0,0),1).
\end{align*}
In view of \eqref{improvedheightbndprop12}, \eqref{improvedheightbndprop11}
and by definition of $\eta$ we see
\begin{align}
\label{improvedheightbndprop22}
\mathrm{spt}\mu_t\cap\mathrm{spt}\eta(t,\cdot\,)\cap\mathrm{spt}D\chi=\emptyset
\end{align}
for all $t\in [0,t_0]$.
Consider the test function 
$\phi\in\mathcal{C}^{2}(\mathbb{R}\times\mathbb{R}^{\mathbf{n}+\mathbf{k}},\mathbb{R}^+)$
given by $\phi:=\eta^3\chi$.
Using \eqref{improvedheightbndprop21}, \eqref{improvedheightbndprop22}
and the Brakke flow equation \eqref{brakkeflowa} we obtain
\begin{align*}
\mu_s(\phi(s,\cdot\,))-\mu_0(\phi(0,\cdot\,))
\leq\int_0^s\int_{\mathbb{R}^{\mathbf{n}+\mathbf{k}}}\chi(\partial_t-\mathrm{div}_{\mu_t}D)\eta^3\;\mathrm{d}\mu_t\;\mathrm{d}t
\leq 0
\end{align*}
for all $s\in (0,t_0]$.
By assumption we have $\mu_0(\phi(0,\cdot))=0$,
hence $\phi(t,x)=0$ for all $t\in (0,t_0]$ and all $x\in\mathrm{spt}\mu_t$.
By definition of $\eta$ and $\chi$ this yields that $h\leq C_{p+1}t^p$
for all $t\in (0,t_0]$ and all 
$(\bar{x},h)\in\mathrm{spt}\mu_t\cap(\{\bar{x}_0\}\times[0,1/2))$.
As $\bar{x}_0\in\mathbf{B}^{N}(0,2)$ was arbitrary we just proved
\begin{align*}
\mathrm{spt}\mu_{t}\cap\left(\mathbf{B}^N(0,2)\times (-2,1/2)\right)
\subset\left(\mathbf{B}^N(0,2)\times (-2,C_{p+1}t^{p+1})\right)
\end{align*}
for all $t\in[0,t_0]$.
In view of \eqref{improvedheightbndprop12}
this completes the statement for $p+1$
and the result follows by induction.
\end{proof}

%
%
%
%
\begin{cor}
\label{heightboundcor}
For every $p\in \mathbb{N}$ there exists a $c_p\in (0,1)$ 
such that the following holds:

Let $R_1,r_0\in (0,\infty)$, $h_1\in (0,r_0/8]$, $R_2\in [r_0,\infty)$,
$x_0\in\mathbb{R}^{\mathbf{n}+\mathbf{k}}$,
$t_1\in\mathbb{R}$, $t_2\in (t_1,t_1+c_p r_0^2)$
and let $(\mu_t)_{t\in [t_1,t_2]}$ be a Brakke flow in $\mathbf{C}(x_0,R_1+r_0,R_2+r_0)$.
Suppose
\begin{align*}
\mathrm{spt}\mu_{t_1}\cap\mathbf{C}(x_0,R_1+r_0,R_2+r_0)
\subset\mathbf{C}(x_0,R_1+r_0,h_1).
\end{align*}

Then for all $t\in [t_1,t_2]$ and $h(t):=h_1+(t-t_1)^pr_0^{1-2p}$ we have
\begin{align*}
\mathrm{spt}\mu_{t}\cap\mathbf{C}(x_0,R_1,R_2)
\subset\mathbf{C}(x_0,R_1,h(t)).
\end{align*}
\end{cor}
%
%
\begin{proof}
We may assume $t_1=0$ and $x_0=0$.
Let $p\in\mathbb{N}$ be given and $C_{p+1}$
be the value according to Proposition \ref{improvedheightbndprop}.
First we want to show
\begin{align}
\label{heightbndcor11}
\mathrm{spt}\mu_{t}\cap\mathbf{C}(0,R_1,r_0/4)
\subset\mathbf{C}(x_0,R_1,h(t))
\end{align}
for all $t\in (0,t_2]$.
To see this consider arbitrary $t\in (0,t_2]$ and 
$x\in\mathrm{spt}\mu_t\cap\mathbf{C}(0,R_1,r_0/4)$.
We want to show $|\tilde{x}|\leq h(t)$.
Suppose $\tilde{x}\neq 0$.
Set $\tilde{v}:=|\tilde{x}|^{-1}\tilde{x}$, $v:=(0,\tilde{v})\in\mathbb{R}^{\mathbf{n}+\mathbf{k}}$, 
and $a:=(\hat{x},h_1\tilde{v})$. 
Note that 
$|x-a|=||\tilde{x}|-h_1|
\leq r_0/4$.

For arbitrary $y\in\mathbb{R}^{\mathbf{n}+\mathbf{k}}$
we have $(y-a)\cdot v\leq |\tilde{y}|-h_1$,
hence by our initial height bound we see
\begin{align*}
\mathrm{spt}\mu_{0}\cap\mathbf{B}(a,r_0/2)
\subset\{y\in\mathbb{R}^{\mathbf{n}+\mathbf{k}}:(y-a)\cdot v\leq 0\}.
\end{align*} 
Using Proposition \ref{improvedheightbndprop} with $\rho=r_0/4$ yields
\begin{align*}
|\tilde{x}|-h_1
=(x-a)\cdot v
\leq\left(r_0/4\right)^{-2p-1}C_{p+1}t^{p+1}
\leq t^pr_0^{1-2p},
\end{align*}
as $t\leq t_2\leq c_pr_0^2$ and for $c_p$ small enough.
As $t$ and $x$ were arbitrary,
this establishes \eqref{heightbndcor11}.

Now consider $z_0\in\mathbf{C}(0,R_1,R_2)$ with $\tilde{z}_0\geq r_0/4$.
By our initial height bound and $h_1\leq r_0/8$ we have
$\mathrm{spt}\mu_0\cap\mathbf{B}(z_0,r_0/8)=\emptyset$.
Thus Lemma \ref{barrierlem} with $R=r_0/16$ implies $z_0\notin\mathrm{spt}\mu_t$
for all $t\in [0,t_2]$, where we used $t_2\leq c_{p}r_0^2\leq (2\mathbf{n})^{-1} (r_0/16)^2$.
In view of \eqref{heightbndcor11}, this completes the proof.
\end{proof}

\section{Maintain graphical representability}
\label{stay_graph}
In this section we prove Theorem \ref{staygraphthm}.
The main idea of the proof
is to iterate Brakkes local regularity theorem (see Theorem \ref{Blocalregthm})
by choosing a time at which graphical representation is obtained
as the new starting time.
To do so we first show a version of Theorem \ref{Blocalregthm}
which only has assumptions at the initial and final time, 
see Proposition \ref{graphorholeprop}.

%
%
%
%
By Corollary \ref{heightboundcor} initial height bounds 
yield weaker height bounds later on.
Also by Huisken's monotonicity formula (see Theorem \ref{localmonthm})
initial bounds on area ratio imply bounds on area ratio in the future
(see Lemma \ref{laterbndlem}).
Moreover by the clearing out lemma (see Lemma \ref{clearoutlem}) non-vanishing at some time
yields a lower bound on measure a bit earlier.
Thus with Brakke's local regularity theorem (see Theorem \ref{Blocalregthm}) we obtain the Proposition below,
which is an improved version of a result found in the author's thesis \cite[Thm.\ 11.7]{lahiri}.
%
%
%
%
\begin{prop}
\label{graphorholeprop}
For all $\kappa,\alpha\in (0,1)$
there exists a $\sigma_1\in (0,1)$
such that the following holds:

Let $\beta\in (0,1]$, $\iota\in(0,\sigma_1\beta]$, $\varrho\in (0,\infty)$, 
$s_1\in\mathbb{R}$, $s_2\in (s_1+2\iota^2\varrho^2,s_1+4(\sigma_1\beta\iota)^{\alpha}\varrho^2]$,
$z_0\in\mathbb{R}^{\mathbf{n}+\mathbf{k}}$
and let $(\mu_t)_{t\in [s_1,s_2]}$ be a Brakke flow in $\mathbf{C}(z_0,2\varrho,2\varrho)$.
Suppose
\begin{align}
\label{graphorholelemb}
&\mathrm{spt}\mu_{s_2}\cap\mathbf{C}(z_0,\sigma_1\iota\varrho,\varrho)\neq\emptyset,
\\
\label{graphorholelema}
&\mathrm{spt}\mu_{s_1}\cap\mathbf{C}(z_0,2\varrho,2\varrho)
\subset\mathbf{C}(z_0,2\varrho,\sigma_1\beta\iota\varrho),
\\
\label{graphorholelemc}
&r^{-\mathbf{n}}\mu_{s_1}(\mathbf{B}(z_0,r))
\leq (2-\kappa)\omega_{\mathbf{n}}
\quad\text{for all}\quad
r\in (\sigma_1\iota\varrho,8\mathbf{n}\sqrt{s_2-s_1}].
\end{align}
Set $I:=(s_1+\iota^2\varrho^2,s_2-\iota^2\varrho^2)$.

Then there exists an 
$u\in\mathcal{C}^{\infty}\left(I\times\mathbf{B}^{\mathbf{n}}(\hat{z}_0,\sigma_1\iota\varrho),\mathbb{R}^{\mathbf{k}}\right)$
such that
\begin{align*}
\mu_{t}\mres\mathbf{C}(z_0,\sigma_1\iota\varrho,\varrho)=\mathscr{H}^{\mathbf{n}}\mres\mathrm{graph}(u(t,\cdot\,))
\;\;\;\text{for all}\;t\in I.
\end{align*}
Moreover $\sup|Du|\leq\beta$
and $F_t(\hat{x})=(\hat{x},u(t,\hat{x}))$ satisfies \eqref{smoothmcf}.
\end{prop}
%
%
\begin{proof}
We may assume $s_1=0$, $z_0=0$ and $\varrho=1$.
By Corollary \ref{heightboundcor} with $p=2/\alpha$
and assumption \eqref{graphorholelema}
we have
\begin{align}
\label{graphorholelem11}
\mathrm{spt}\mu_t\cap\mathbf{C}(0,1,1)\subset\mathbf{C}(0,1,2\sigma_1\beta\iota)
\end{align}
for all $t\in [0,s_2]$, where we used $s_2^{2/\alpha}\leq 4^{2/\alpha}(\sigma_1\beta\iota)^2\leq\sigma_1\beta\iota$
for $\sigma_1$ small enough.

Choose $\lambda_1\in (0,\kappa]$ such that
$\lambda_1\leq (2C_1)^{-1}$ and
$C_2(2^{\mathbf{n}}\lambda_1\omega_{\mathbf{n}})^{\frac{2}{\mathbf{n}+6}}\leq (16\mathbf{n})^{-1}$,
where $C_1$ and $C_2$ are the constants from Lemma \ref{laterbndlem}.
and Lemma \ref{clearoutlem} respectively.
Let $\Lambda\in (1,\infty)$ be from Theorem \ref{Blocalregthm}
chosen with respect to $\lambda=\lambda_1$.
Consider the radius
\begin{align*}
\rho_1:=4^{-1}\Lambda^{-\frac{1}{2}}\iota\in(2\sigma_1\iota,\sqrt{s_2}),
\end{align*}
where we chose $\sigma_1$ small enough.
Set $t_2:=s_2-(8\mathbf{n})^{-1}(\rho_1/2)^2$.
We want to show
\begin{align}
\label{graphorholelem21}
\rho_1^{-\mathbf{n}}\mu_{t_2}(\mathbf{B}(0,\rho_1/2))\geq\lambda_1\omega_{\mathbf{n}}.
\end{align}
Suppose this would be false,
then we can use Lemma \ref{clearoutlem} 
with $\eta=2^{\mathbf{n}}\omega_{\mathbf{n}}\lambda_1$ and $R=\rho_1/2$
to obtain $\mu_{s_2}(\mathbf{B}(0,\rho_1/4))=0$.
In view of \eqref{graphorholelem11}
this contradicts \eqref{graphorholelemb}.
Thus \eqref{graphorholelem21} has to be true.

By assumption \eqref{graphorholelemc}, Lemma \ref{laterbndlem} 
and $\lambda_1\leq\min\{\kappa,(2C_1\omega_{\mathbf{n}})^{-1}\}$ we obtain
\begin{align}
\label{graphorholelem31}
\rho_1^{-\mathbf{n}}\mu_{t}(\mathbf{B}(0,4\rho_1))
&\leq 2C_1\omega_{\mathbf{n}}
\leq \lambda_1^{-1},
\\
\label{graphorholelem41}
\rho_1^{-\mathbf{n}}\mu_{0}(\mathbf{B}(0,\rho_1))
&\leq (2-\kappa)\omega_{\mathbf{n}}
\leq (2-\lambda_1)\omega_{\mathbf{n}}.
\end{align}

Now choose $h_0$ according to Theorem \ref{Blocalregthm}
with respect to $\lambda=\lambda_1$ as above.
Set $h:=\min\{h_0,\Lambda^{-1}\beta\}$. 
Note that $2\sigma_1\beta\iota\leq h\rho_1$
for $\sigma_1$ small enough.
Thus \eqref{graphorholelem11}, \eqref{graphorholelem21},
\eqref{graphorholelem31} and \eqref{graphorholelem41}
let us apply Theorem \ref{Blocalregthm}
with $\rho=\rho_1$
which establishes the result.
\end{proof}

%
%
%
%
Looking at Proposition \ref{graphorholeprop} we see that 
at time $t_2-2\iota\varrho^2$ we satisfy a non-vanishing condition
in an increased cylinder.
This allows to iterate above Proposition
to obtain graphical representability inside a larger cylinder.
%
%
%
%
\begin{lem}
\label{improvedgohlem}
For all $\kappa\in (0,1)$ there exists a $\sigma_2\in (0,1)$
such that the following holds:

Let $\beta\in (0,1]$, $\lambda\in (0,\sigma_2\beta]$, 
$s_4\in (2\lambda\varrho_0^2,2\sqrt[32]{\lambda}\varrho_0^2]$,
$y_0\in\mathbb{R}^{\mathbf{n}+\mathbf{k}}$
and let $(\mu_t)_{t\in [0,s_4]}$ be a Brakke flow in $\mathbf{C}(y_0,5\varrho_0,2\varrho_0)$.
Suppose we have
\begin{align}
\label{improvedgohlemb}
&\mathrm{spt}\mu_{s_4}\cap\mathbf{C}(y_0,\sigma_2\lambda\varrho_0,\varrho_0)\neq\emptyset,
\\
\label{improvedgohlema}
&\mathrm{spt}\mu_{0}\cap\mathbf{C}(y_0,5\varrho_0,2\varrho_0)
\subset\mathbf{C}(y_0,5\varrho_0,\sigma_2\beta\lambda\varrho_0),
\\
\label{improvedgohlemc}
&r^{-\mathbf{n}}\mu_{0}(\mathbf{B}(y,r))
\leq (2-\kappa)\omega_{\mathbf{n}}
\end{align}
for all $y\in\mathbf{B}^{\mathbf{n}}(\hat{y}_0,3\varrho_0)\times\{\tilde{y}_0\}$ 
and all $r\in (\sigma_2\lambda\varrho_0,8\mathbf{n}\sqrt{s_4}]$.
Set $I:=(\lambda\varrho_0^2,s_4-\lambda\varrho_0^2)$.

Then there exists a
$v\in\mathcal{C}^{\infty}\left(I\times\mathbf{B}^{\mathbf{n}}(\hat{y}_0,2\varrho_0),\mathbb{R}^{\mathbf{k}}\right)$
such that
\begin{align*}
\mu_{t}\mres\mathbf{C}(y_0,2\varrho_0,\varrho_0)
=\mathscr{H}^{\mathbf{n}}\mres\mathrm{graph}(v(t,\cdot\,))
\;\;\;\text{for all}\;t\in I.
\end{align*}
Moreover $\sup|Dv|\leq\beta$
and $F_t(\hat{x})=(\hat{x},v(t,\hat{x}))$ satisfies \eqref{smoothmcf}.
\end{lem}
%
%
\begin{proof}
We may assume $y_0=0$, $\varrho_0=1$ and set $\alpha:=1/64$.
Let $\sigma_1$ be from Proposition \ref{graphorholeprop} with respect to $\kappa$.
For $m\in\mathbb{N}$ set 
\begin{align*}
R_m:=m\sigma_1\sqrt{\sigma_2}\lambda,
\quad
T_m:=s_4-2m\sigma_2\lambda^2,
\quad
J_m:=(\lambda^2,T_m).
\end{align*}
Note that $T_m= s_4-2\sigma_1^{-1}\sqrt{\sigma_2}R_m\lambda$,
in particular for $R_m\leq 3$ and $\sigma_2$ small enough
we have $T_m > s_4-\lambda > 4\sigma_2\lambda^2$.

Consider the following statement:
\\
$stat(m):\Leftrightarrow$
There exists a 
$v_m\in\mathcal{C}^{\infty}\left(J_m\times \mathbf{B}^{\mathbf{n}}(0,R_m),\mathbb{R}^{\mathbf{k}}\right)$
such that
\begin{align*}
\mu_{t}\mres\mathbf{C}(0,R_m,1)=\mathscr{H}^{\mathbf{n}}\mres\mathrm{graph}(v_m(t,\cdot\,))
\quad\text{for all}\;t\in J_m,
\\
\sup|Dv_m|\leq\Lambda_2\lambda
\;\text{and}\;
F_t(\hat{x})=(\hat{x},v_m(t,\hat{x}))\;\text{satisfies}\;\eqref{smoothmcf}.
\end{align*}
By Proposition \ref{graphorholeprop} 
with $s_2=s_4$, $\varrho=1$, $\iota=\sqrt{\sigma_2}\lambda$
we see that $stat(1)$ is true.
Now suppose $stat(m_0)$ holds for some $m_0\in\mathbb{N}$ with $R_{m_0}\leq 3$.
Using Proposition \ref{graphorholeprop} with $s_2=T_{m_0}-\sigma_2\lambda^2$, $\varrho=1$, $\iota=\sqrt{\sigma_2}\lambda$
and arbitrary 
$z_0\in\mathbf{B}^{\mathbf{n}}(0,R_{m_0}+\sigma_1\sqrt{\sigma_2}\lambda)\times\{0\}^{\mathbf{k}}$
yields that also $stat(m_0+1)$ is true.
Thus $stat(m_1)$ holds for some $m_1\in\mathbb{N}$ with $2\leq R_{m_1}\leq 3$,
which establishes the result.
\end{proof}

%
%
%
%
%
%
Now consider a Brakke flow which is initially graphical with small Lipschitz constant.
Then the conditions of Lemma \ref{improvedgohlem} are satisfied for arbitrarily small scaling.
Thus we can extend the interval of graphical representation up to the initial time.
%
%
%
%
\begin{lem}
\label{staygraphohlem}
There exists an $\sigma_3\in (0,1)$
such that the following holds:

Let $\beta\in (0,1]$, $\iota\in (0,\sigma_3\beta]$, $\rho_0\in (0,\infty)$,
$t_1\in\mathbb{R}$, $t_2\in (t_1+\iota\rho_0^2,t_1+\sqrt[32]{\iota}\rho_0^2]$,
$z_0\in\mathbb{R}^{\mathbf{n}+\mathbf{k}}$
and let $(\mu_t)_{t\in [t_1,t_2]}$ be a Brakke flow in $\mathbf{C}(z_0,5\rho_0,2\rho_0)$.
Assume $z_0\in\mathrm{spt}\mu_{t_1}$ and
\begin{align}
\label{staygraphohlema}
\mathrm{spt}\mu_{t_2}\cap\mathbf{C}(z_0,\sigma_3\iota\rho_0,\rho_0)\neq\emptyset.
\end{align}
Suppose there exists an $u_0\in\mathcal{C}^{0,1}\left(\mathbf{B}^{\mathbf{n}}(\hat{z}_0,5\rho_0),\mathbb{R}^{\mathbf{k}}\right)$
with $\mathrm{lip}(u_0)\leq\sigma_3\beta\iota$ and
\begin{align}
\label{staygraphohlemb}
\mu_{t_1}\mres\mathbf{C}(z_0,5\rho_0,2\rho_0)=\mathscr{H}^{\mathbf{n}}\mres\mathrm{graph}(u_0).
\end{align}
Set $I:=(t_1,t_2-\iota\rho_0^2)$.

Then there exists an 
$u\in\mathcal{C}^{\infty}\left(I\times\mathbf{B}^{\mathbf{n}}(\hat{z}_0,2\rho_0),\mathbb{R}^{\mathbf{k}}\right)$
such that
\begin{align*}
\mu_{t}\mres\mathbf{C}(z_0,2\rho_0,\rho_0)=\mathscr{H}^{\mathbf{n}}\mres\mathrm{graph}(u(t,\cdot\,))
\;\;\;\text{for all}\;t\in I.
\end{align*}
Moreover $\sup|Du|\leq\beta$ and 
$F_t(\hat{x})=(\hat{x},u(t,\hat{x}))$ satisfies \eqref{smoothmcf}.
\end{lem}
%
%
\begin{proof}
We may assume $t_1=0$, $z_0=0$ and $\rho_0=1$.

For $s\in (0,\iota/4]$
we consider the following statement:
\\
$stat(s):\Leftrightarrow$ There exists an 
$v_s\in\mathcal{C}^{\infty}\left(
(s,t_2-\iota/4)\times\mathbf{B}^{\mathbf{n}}(0,1),\mathbb{R}^{\mathbf{k}}\right)$
such that
\begin{align}
\label{staygraphohlem31}
\mu_{t}\mres\mathbf{C}(0,2,1)
=\mathscr{H}^{\mathbf{n}}\mres\mathrm{graph}(v_s(t,\cdot\,))
\;\;\text{for all}\;t\in (s,t_2-\iota/4),
\\
\label{staygraphohlem32}
\sup|Dv_s|\leq\beta
\;\text{and}\;
F_t(\hat{x})=(\hat{x},v_s(t,\hat{x}))\;\text{satisfies}\;\eqref{smoothmcf}
\end{align}

Suppose $stat(s_0)$ holds for some $s_0\in (0,\iota/4]$.
We want to show that in this case also $stat(\frac{s_0}{2})$ holds.
Let $\hat{y}\in\mathbf{B}^{\mathbf{n}}(0,2)$
be arbitrary, 
set $y:=(\hat{y},u_0(\hat{y}))$
and $\varrho_0:=\sqrt{\iota^{-1}s_0}\leq 1/2$.
Using assumption \eqref{staygraphohlemb} and $\mathrm{lip}(u_0)\leq\sigma_3\beta\iota$ yields
\begin{align}
\label{staygraphohlem41}
\mathrm{spt}\mu_{0}\cap\mathbf{C}(y,5\varrho_0,3/2)
\subset\mathbf{C}(y,5\varrho_0,5\sigma_3\beta\iota\varrho_0).
\end{align}
Then by Corollary \ref{heightboundcor}
with $R_1=r_0=2\varrho_0$ and $R_2=5/4$
we obtain
\begin{align}
\label{staygraphohlem42}
\mathrm{spt}\mu_{t}\cap\mathbf{C}((\hat{y},0),2\varrho_0,1)
\subset\mathbf{C}(y,2\varrho_0,\varrho_0)
\end{align}
for all $t\in [0,2s_0]$.
Here we estimated 
$5\sigma_3\beta\iota\varrho_0+\varrho_0^{-1}s_0\leq C\iota\varrho_0\leq\varrho_0$
and $|u_0(\hat{y})|\leq 1/4$.

Set $J_2:=(s_0/2,3s_0/2)$.
We want to use Lemma \ref{improvedgohlem} with
$\kappa=\frac{1}{2}$, $\lambda=\iota/2$, $s_4=2s_0$ and $y_0=y$.
Choosing $\sigma_3$ small enough we obtain the following:
$\iota\varrho_0^2/2\leq s_0/2$;
Statement \eqref{staygraphohlem41}
implies \eqref{improvedgohlema};
Using assumption \eqref{staygraphohlemb}
and $\mathrm{lip}(u_0)\leq\sigma_3\beta\iota$, 
we see that \eqref{improvedgohlemc} holds.
Moreover by \eqref{staygraphohlem42} and as $s_0<2s_0<t_2-\iota/4$ 
we can use assumption \eqref{staygraphohlem31} to show \eqref{improvedgohlemb}.
Then by Lemma \ref{improvedgohlem}
we obtain an
$v_{s,\hat{y}}\in\mathcal{C}^{\infty}\left(
J_2\times\mathbf{B}^{\mathbf{n}}(\hat{y},2\varrho_0),\mathbb{R}^{\mathbf{k}}\right)$
with
\begin{align*}
\mu_{t}\mres\mathbf{C}((\hat{y},0),2\varrho_0,1)
=\mu_{t}\mres\mathbf{C}(y,2\varrho_0,\varrho_0)
=\mathscr{H}^{\mathbf{n}}\mres\mathrm{graph}(v_{s,\hat{y}}(t,\cdot\,))
\end{align*}
for all $t\in J_2$.
Here we used \eqref{staygraphohlem42} to obtain the first equality.
Also $F(t,\hat{x})=(\hat{x},v_{s,\hat{y}}(t,\hat{x}))$ satisfies \eqref{smoothmcf}
and $\sup|Dv_{s,\hat{y}}|\leq\beta$.
As $\hat{y}\in\mathbf{B}^{\mathbf{n}}(0,2)$ was arbitrary
this shows $stat(\frac{s_0}{2})$ is true.

Similarly we can use Lemma \ref{improvedgohlem} 
with $y_0=0$, $s_4=t_2$, $\varrho_0=1$ and $\lambda=\iota/4$
to obtain that $stat(\iota/4)$ is true
for $\sigma_3$ small enough. 
Hence we can start an iteration
which yields that $stat(0)$ holds.
This establishes the result.
\end{proof}

%
%
%
%
%
%
Consider the situation of Lemma \ref{staygraphohlem}.
If $\hat{z}_0\times\mathbb{R}^{\mathbf{k}}$ intersects $\mu_{t_2}$ and $\iota$ is small enough
we have that $(\mu_t)_{t\in [t_2-\iota\rho_0^2,t_2]}$ satisfies the conditions of
Lemma \ref{staygraphohlem} on the smaller scale $\rho_0/2$
with $\sigma_3\beta\iota$ replaced by $\beta$.
Thus we can use $t_2-\iota$ as the new starting time.
This yields an iteration and by curvature bounds for graphical mean curvature flow,
we can assure that the gradient does not blow up.
This leads to the following:
%
%
%
%
\begin{lem}
\label{staygraphlem}
There exists a $\sigma_4\in (0,1)$ such that the following holds:

Let $\eta\in (0,\sigma_4]$, $\varrho_0\in (0,\infty)$, 
$s_1\in\mathbb{R}$, $s_2\in (s_1,s_1+\eta\varrho_0^2]$,
$y_0\in\mathbb{R}^{\mathbf{n}+\mathbf{k}}$
and let $(\mu_t)_{t\in [s_1,s_2]}$ be a Brakke flow in $\mathbf{C}(y_0,2\varrho_0,2\varrho_0)$.
Assume $y_0\in\mathrm{spt}\mu_{s_1}$ and
\begin{align}
\label{staygraphlema}
\mathrm{spt}\mu_{s_2}\cap\left(\{\hat{y}_0\}^{\mathbf{n}}\times\mathbf{B}^{\mathbf{k}}(\tilde{y}_0,\varrho_0)\right)\neq\emptyset.
\end{align}
Suppose there exists a $v_0\in\mathcal{C}^{0,1}\left(\mathbf{B}^{\mathbf{n}}(\hat{y}_0,2\varrho_0),\mathbb{R}^{\mathbf{k}}\right)$
with $\mathrm{lip}(v_0)\leq\eta^4$ and
\begin{align}
\label{staygraphlemb}
\mu_{s_1}\mres\mathbf{C}(y_0,2\varrho_0,2\varrho_0)=\mathscr{H}^{\mathbf{n}}\mres\mathrm{graph}(v_0).
\end{align}
Let $s\in(s_1,s_2)$ and $\varrho(s):=\eta^{-1/16}\sqrt{s_2-s}$.

Then there exists a 
$v_s\in\mathcal{C}^{\infty}\left((s_1,s)\times\mathbf{B}^{\mathbf{n}}(\hat{y}_0,\varrho(s)),\mathbb{R}^{\mathbf{k}}\right)$
such that
\begin{align*}
\mu_{t}\mres\mathbf{C}(y_0,\varrho(s),\varrho_0)=\mathscr{H}^{\mathbf{n}}\mres\mathrm{graph}(v_s(t,\cdot))
\;\;\;\text{for all}\;t\in (s_1,s).
\end{align*}
Moreover $\sup |Dv_s|\leq\eta$
and $F_t(\hat{x})=(\hat{x},v_s(t,\hat{x}))$ satisfies \eqref{smoothmcf}.
\end{lem}

%
%
\begin{proof}
We may assume $s_1=0$, $y_0=0$ and $\varrho_0=1$.
For a smooth $\mathbf{n}$-dimensional submanifold $M$ and
$\mu=\mathscr{H}^{\mathbf{n}}\mres M$ set
\begin{align*}
\|\mathrm{tilt}(\mu)\|_{U}&:=\sup_{x\in M\cap U}
\|(\mathbb{R}^{\mathbf{n}}\times\{0\}^{\mathbf{k}})_{\natural}-\mathbf{T}(M,x)\|
\\
\|\mathbf{A}(\mu)\|_{U}&:=\sup_{x\in M\cap U}|\mathbf{A}(M,x)|
\end{align*}
for $U\subset\mathbb{R}^{\mathbf{n}+\mathbf{k}}$ open.
Here $\mathbf{A}$ is the second fundamental form of $M$.

For $m\in \mathbb{N}$ set
\begin{align*}
\eta_0:=\sqrt[16]{\eta},
\;\;\delta_0:=\eta_0^{-1}\sqrt{s_2},
\;\;\delta_{m}:=\eta_0^{2m+16}\delta_0,
\\
\tau_0^{+}:=s_2-4\eta_0^{44}s_2,
\quad\tau_m^{+}:=s_2-\eta_0^{m+15}\delta_{m}^2,
\\
\tau_0:=s_2-8\eta_0^{44}s_2,
\quad\tau_m:=s_2-2\eta_0^{m+15}\delta_{m}^2.
\end{align*}
Note that for all $m\in\mathbb{N}\cup\{0\}$
\begin{align}
\label{staygraphlem12}
0<\tau_{m+1}-\tau_{m}^+
<\tau_{m+1}^+-\tau_{m}
\leq s_2-\tau_{m}
\leq C\eta_0^{m+10}\delta_{m+1}^2.
\end{align}

For $m\in \mathbb{N}\cup\{0\}$ we consider the following statement:
\\
$stat(m):\Leftrightarrow$
There exists a 
$v_m\in\mathcal{C}^{\infty}\left(
(0,\tau_m^{+})\times\mathbf{B}^{\mathbf{n}}(0,2\delta_{m}),\mathbb{R}^{\mathbf{k}}\right)$
such that
\begin{align}
\label{staygraphlem21}
&\mu_{t}\mres U_m
=\mathscr{H}^{\mathbf{n}}\mres\mathrm{graph}(v_m(t,\cdot\,))
\;\;\text{for all}\;t\in (0,s)
\\
\label{staygraphlem22}
&F_t(\hat{x})=(\hat{x},v_m(t,\hat{x}))\;\text{satisfies}\;\eqref{smoothmcf},
\\
\label{staygraphlem23}
&\sup_{t\in (0,\tau_m^+)}
\|\mathrm{tilt}(\mu_t)\|_{U_{m}}
\leq\sum_{i=0}^{m}2^{-i-4}\eta,
\end{align}
\begin{align}
\label{staygraphlem24}
\delta_{m+1}^{2}\sup_{t\in [\tau_0/2,\tau_m]}\|\mathbf{A}(\mu_{t})\|_{U_{m}}^2
&\leq\eta_0^{2m+33},
\end{align}
where $U_m:=\mathbf{C}(0,2\delta_{m},1)$.

First we observe the following:
Consider $t\in [0,s_2)$
then there exists an $i\in\mathbb{N}\cup\{0\}$
such that $\tau_i\geq t$ and
\begin{align}
\label{staygraphlem27}
\eta_0^{-1}\sqrt{s_2-t}\leq\delta_{i}.
\end{align}
If $t\in [0,\tau_0]$
this directly follows from the definition of $\delta_0$.
For $t\in (\tau_0,s_2)$
choose $i\in\mathbb{N}$ such that (with \eqref{staygraphlem12})
\begin{align*}
s_2-\tau_i
\leq s_2-t
\leq s_2-\tau_{i-1}
\leq\eta_0^{i+8}\delta_{i}^2.
\end{align*}
Also note that
by \eqref{staygraphlemb}, $\mathrm{lip}(v_0)\leq\eta^4$, $\eta_0^{16}=\eta\leq\sigma_4$, $s_2\leq\eta$
and Corollary \ref{heightboundcor} with $R_2=1$, $R_1=r_0=\eta_0^{-6}\delta_0=\eta_0^{-7}\sqrt{s_2}\leq\eta_0$ and $p=8$ we have
\begin{align}
\label{staygraphlem11}
\mathrm{spt}\mu_{t}\cap\mathbf{C}(0,\eta_0^{-6}\delta_0,1)
\subset\mathbf{C}(0,\eta_0^{-6}\delta_0,\eta^3\delta_0)
\end{align}
for all $t\in [0,s_2]$ where we chose $\sigma_4$ small enough.

Next notice that $stat(0)$ is true.
To see this use Lemma \ref{staygraphohlem}
with $t_1=0$, $t_2=s_2$, $\beta=\eta_0^{16}/16$, $\iota=\eta_0^{46}$ and $\rho_0=2\delta_0$.
Use $\eta_0^{16}=\eta\leq\sigma_4$
and choose $\sigma_4$ small enough
to obtain $\eta_0^{64}\leq\sigma_3\beta\iota$,
$s_2-\iota\rho_0^2\geq\tau_0^{+}$
and $\sqrt[32]{\iota}\rho_0^2\geq s_2$.
In particular use \eqref{staygraphlem11} to see that
$\mathrm{spt}\mu_{t}\cap\mathbf{C}(0,4\delta_0,2\delta_0)=\mathrm{spt}\mu_{t}\cap\mathbf{C}(0,4\delta_0,1)$.
This shows \eqref{staygraphlem21}-\eqref{staygraphlem23} for $m=0$ on the larger radius $4\delta_0$.
By Lemma \ref{wangcurveestlem} with $t_1=0$, $t_2=\tau_0^+$ and $\rho=2\delta_0$
we have
\begin{align*}
\sup_{t\in [\tau_0/2,\tau_0]}\|\mathbf{A}(\mu_{t})\|_{\mathbf{C}(0,2\delta_{m},1)}
\leq C\tau_0^{-1/2}
= Cs_2^{-1/2}
= C/(\eta_0\delta_0)
= C\eta_0^{17}/\delta_1,
\end{align*}
which completes $stat(0)$.

Now consider $m\in\mathbb{N}\cup\{0\}$
such that $stat(m)$ holds,
we want to show that then also $stat(m+1)$ holds.
By \eqref{staygraphlema} and \eqref{staygraphlem11} there exists a point
$a_{\omega}$ in $\mathrm{spt}\mu_{s_2}\cap(\{0\}^{\mathbf{n}}\times\mathbf{B}^{\mathbf{k}}(0,1/4))$.
Using \eqref{staygraphlem12} and Lemma \ref{barrierlem} with $z_0=a_{\omega}$ and $R=4\mathbf{n}\sqrt{s_2-\tau_m}$
we see
$\mathbf{B}(a_{\omega},C\eta_0^{5}\delta_{m+1})\cap\mathrm{spt}\mu_{\tau_m}\neq\emptyset$.
Now choose $a_m\in\mathrm{spt}\mu_{\tau_m}$ as a nearest point to $a_{\omega}$.
In particular $a_m\in\mathbf{B}(a_{\omega},C\eta_0^{5}\delta_{m+1})\subset\mathbf{C}(0,C\eta_0^{5}\delta_{m+1},1/2)$ and 
$a_{\omega}-a_m\in\mathbf{T}(\mu_{\tau_{m}},a_{m})^{\bot}$.
We claim
\begin{align}
\label{staygraphlem31}
\mathrm{spt}\mu_{t}\cap\mathbf{C}(0,\delta_{m},1)
\subset\mathbf{C}(a_{m},2\delta_{m},\delta_{m+1})
\end{align}
for all $t\in [\tau_{m},s_2]$.
For $m=0$ this directly follows from \eqref{staygraphlem11}, $|\hat{a}_0|\leq\delta_{0}$
and $\eta^3\delta_0=\eta_0^{30}\delta_1$. 
For $m\in\mathbb{N}$ note that by \eqref{staygraphlem21} and \eqref{staygraphlem23}
we can use Corollary \ref{heightboundcor}
with $t_1=\tau_{m}$, $t_2=s_2$, $R_1=r_0=\delta_{m}$, $R_2=1/4$ and $p=2$ to see
\begin{align*}
\mathrm{spt}\mu_{t}\cap\mathbf{C}(a_{m},2\delta_{m},1/4)
\subset\mathbf{C}(a_{m},2\delta_{m},C\eta\delta_{m}).
\end{align*}
for all $t\in [\tau_{m},s_2]$.
Then \eqref{staygraphlem31} follows from \eqref{staygraphlem11}, $|\hat{a}_m|\leq\delta_{m}$
and $C\eta\delta_{m}=C\eta_0^{14}\delta_{m+1}\leq\delta_{m+1}$ for $\sigma_4$ small enough. 

Fix an $S\in\mathbf{SO}(\mathbf{n}+\mathbf{k})$
with $S(\mathbb{R}^{\mathbf{n}}\times\{0\}^{\mathbf{k}})=\mathbf{T}(\mu_{\tau_{m}},a_{m})$
and set
\begin{align*}
\nu_t:=(S^{-1})_{\sharp}(\mu_t-a_{m})
\quad\text{for }t\in [0,s_2],
\end{align*}
in particular $(\nu_t)_{t\in [0,s_2]}$ is a Brakke flow in $\mathbf{C}(0,1,1)$.
Using \eqref{staygraphlem21} and \eqref{staygraphlem23}
for $\sigma_4$ small enough
there exists an
$f\in\mathcal{C}^{\infty}\left((0,\tau_{m}^+)\times\mathbf{B}^{\mathbf{n}}(0,\delta_{m}),\mathbb{R}^{\mathbf{k}}\right)$ 
with $\sup|Df|\leq\eta$ and
\begin{align}
\label{staygraphlem42}
\nu_{t}\mres\mathbf{C}(0,\delta_m,\delta_m)
=\mathscr{H}^{\mathbf{n}}\mres\mathrm{graph}(f(t,\cdot\,))
\;\;\text{for all}\;t\in (0,\tau_{m}^+).
\end{align}
By choice of $a_{\omega}$ and $a_m$ we also have
\begin{align}
\label{staygraphlem43}
S^{-1}(a_{\omega}-a_m)\in\big(\{0\}^{\mathbf{n}}\times\mathbf{B}^{\mathbf{k}}(0,C\eta_0^{5}\delta_{m+1})\big)\cap\mathrm{spt}\nu_{s_2}.
\end{align}
By defintion of $\nu$ its evident that $f(\tau_m,0)=0$ and $Df(\tau_m,0)=0$.
In view of \eqref{staygraphlem23} and \eqref{staygraphlem24} we can estimate
\begin{align}
\label{staygraphlem44}
\sup_{\mathbf{B}^{\mathbf{n}}(0,2^9\delta_{m+1})}|Df(\tau_m,\cdot\,)|\leq C\eta_0^{m+33/2}.
\end{align}
Here we used Corollary \cite[Cor.\ A2]{lahiri2}, $\delta_{m+1}\leq\eta_0^2\delta_m$, $\eta_0^{16}=\eta\leq\sigma_4$ and chose $\sigma_4$ small enough.
In view of $0\in\mathrm{spt}\nu$, \eqref{staygraphlem12}, \eqref{staygraphlem42}, \eqref{staygraphlem44}
we can use Corollary \ref{heightboundcor} to see
\begin{align}
\label{staygraphlem45}
\mathrm{spt}\nu_t\cap\mathbf{C}(0,2^{9}\delta_{m+1},2^{9}\delta_{m+1})
\subset\mathbf{C}(0,2^{9}\delta_{m+1},C\eta_0^{m+33/2}\delta_{m+1})
\end{align}
for all $t\in [\tau_{m},s_2]$.

We want to use Lemma \ref{staygraphohlem} with 
$t_1=\tau_m$, $t_2=s_2$, $\rho_0=2^6\delta_{m+1}$, 
$\beta=l_1$, $\iota=4^{-6}\eta_0^{m+16}$ and $z_0=0$,
where $l_1$ is the from Lemma \ref{wangcurveestlem}.
By $\eta_0^{16}=\eta\leq\sigma_4$ and for $\sigma_4$ small enough we see:
\eqref{staygraphlem43} implies  \eqref{staygraphohlema},
\eqref{staygraphlem42} implies \eqref{staygraphohlemb}
and \eqref{staygraphlem44} yields the desired Lipschitz bound.
Moreover with \eqref{staygraphlem12} we can estimate
$$\eta_0^{m+16}\delta_{m+1}^2
\leq s_2-\tau_{m+1}^+
\leq s_2-\tau_m
\leq\sqrt[32]{\eta_0^{m+16}}\delta_{m+1}^2.$$
Then we obtain 
a $g\in\mathcal{C}^{\infty}\left((\tau_{m},\tau_{m+1}^+)\times\mathbf{B}^{\mathbf{n}}(0,2^7\delta_{m+1}),\mathbb{R}^{\mathbf{k}}\right)$ 
with
\begin{align}
\label{staygraphlem51}
\nu_{t}\mres\mathbf{C}(0,2^7\delta_{m+1},2^7\delta_{m+1})
=\mathscr{H}^{\mathbf{n}}\mres\mathrm{graph}(g(t,\cdot\,))
\end{align}
for all $t\in (\tau_{m},\tau_{m+1}^+)$ and $\sup|Dg|\leq l_1$.
By \eqref{staygraphlem42}, \eqref{staygraphlem51} and \eqref{staygraphlem12}
we can use Lemma \ref{wangcurveestlem} 
with $t_1=\tau_{m+1}^+-\delta_{m+1}^2\geq\tau_0-\delta_1>0$, $t_2=\tau_{m+1}^+$ and $\rho=2^6\delta_{m+1}$
to obtain
\begin{align}
\label{staygraphlem52}
\sup_{t\in [\tau_{m},\tau_{m+1}^+)}\|\mathbf{A}(\nu_t)\|^2_{\mathbf{C}(0,2^6\delta_{m+1},2^7\delta_{m+1})}
&\leq C\delta_{m+1}^{-2}.
\end{align}
Here we used \eqref{staygraphlem12} to estimate
\begin{align*}
\tau_{m}-(\tau_{m+1}^+-\delta_{m+1}^2)
\geq-C\eta_0^{m+10}\delta_{m+1}^2+\delta_{m+1}^2
\geq\delta_{m+1}^2/2.
\end{align*}
By definition of $\nu$ we have
$\|\mathbf{A}(\nu_{t})\|^2_{\mathbf{B}(0,2^6\delta_{m+1})}
\leq\|\mathbf{A}(\mu_{t})\|^2_{\mathbf{C}(0,\delta_{m},1)}$
for all $t\in (0,\tau_m]$.
Then in view of \eqref{staygraphlem12}, \eqref{staygraphlem24}, \eqref{staygraphlem42}, \eqref{staygraphlem51} and \eqref{staygraphlem52}
we can use Lemma \ref{curveboundcor}
with $t_1=\tau_{m}$, $t_2=\tau_{m+1}^+$, $\varrho=2^5\delta_{m+1}$, $L=C$ and $p=4$ to estimate
\begin{align}
\label{staygraphlem53}
\delta_{m+1}^2\|\mathbf{A}(\nu_t)\|^2_{\mathbf{B}(0,2^5\delta_{m+1})}
&\leq C\big(\eta_0^{2m+33}+\eta_0^{4(m+10)}\big)\leq C\eta_0^{2m+33}
\end{align}
for all $t\in [\tau_{m},\tau_{m+1}^+)$.

In view of \eqref{staygraphlem45}, \eqref{staygraphlem51} and \eqref{staygraphlem53}
we can use Lemma \cite[Lem.\ A4]{lahiri2} to obtain the gradient bound
\begin{align}
\label{staygraphlem56}
\sup_{[\tau_{m},\tau_{m+1}^+)\times\mathbf{B}^{\mathbf{n}}(0,2^5\delta_{m+1})}|Dg|\leq C\eta_0^{m+33/2}.
\end{align}

By definition of $\nu$, \eqref{staygraphlem23}, \eqref{staygraphlem51} and \eqref{staygraphlem56}
there exists a smooth function
$h\in\mathcal{C}^{\infty}\left((\tau_{m},\tau_{m+1}^+)\times\mathbf{B}^{\mathbf{n}}(0,8\delta_{m+1}),\mathbb{R}^{\mathbf{k}}\right)$ 
with
\begin{align*}
\mu_{t}\mres\mathbf{C}(a_m,8\delta_{m+1},\delta_{m+1})
=\mathscr{H}^{\mathbf{n}}\mres\mathrm{graph}(h(t,\cdot\,))
\\
\|\mathrm{tilt}(\mu_t)\|_{\mathbf{C}(a_m,8\delta_{m+1},\delta_{m+1})}
\leq\sum_{i=0}^{m}2^{-i-4}\eta+C\eta_0^{m+33/2}
\leq\sum_{i=0}^{m+1}2^{-i-4}\eta,
\end{align*}
for all $t\in (\tau_{m},\tau_{m+1}^+)$,
where we used $\eta_0^{16}=\eta\leq\sigma_4$
and chose $\sigma_4$ small enough.
By \eqref{staygraphlem31} and $|a_m|\leq\delta_{m+1}$
this proves \eqref{staygraphlem21}-\eqref{staygraphlem23} for $m+1$.
Also
\begin{align*}
\|\mathbf{A}(\mu_{t})\|^2_{\mathbf{C}(0,4\delta_{m+1},1)}
\leq\|\mathbf{A}(\mu_{t})\|^2_{\mathbf{B}(a_m,2^5\delta_{m+1})}
=\|\mathbf{A}(\nu_{t})\|^2_{\mathbf{B}(0,2^5\delta_{m+1})}
\end{align*}
for all $t\in [\tau_m,\tau_{m+1}]$.
Hence \eqref{staygraphlem53} and $\delta_{m+1}=\eta_0^{-2}\delta_{m+2}$
imply \eqref{staygraphlem24} for $m+1$.
This proves $stat(m+1)$
and in view of \eqref{staygraphlem27} an induction establishes the result.
\end{proof}

%
%
%
%
Consider the setting of Theorem \ref{staygraphthm}.
We will use Lemma \ref{staygraphlem} to show that for all $t\in (t_1,t_2)$ we have that
$\mathrm{spt}\mu_t\cap\mathbf{C}(a,\rho,\rho)$
is contained in a Lipschitz graph and has density $<2$ almost everywehere.
Then by Theorem~\ref{nobdrythm} and the properties of a Brakke flow
we find a sequence $\tau_m\nearrow t_2$
such that $\mathrm{spt}\mu_{\tau_m}\mres\mathbf{C}(a,\rho,\rho)$ is graphical.
This allows us to apply Lemma \ref{staygraphlem} with arbitrary centre point in $\mathbf{B}^{\mathbf{n}}(\hat{a},\rho)\times\{0\}^{\mathbf{k}}$
and final time $\tau_m$ to conclude $\mathrm{spt}\mu_{t}\mres\mathbf{C}(a,\rho,\rho)$ is graphical for all $t\in [t_1,\tau_m)$.
By the convergence of $(\tau_m)$ this establishes Theorem \ref{staygraphthm}.

%
%
%
%
\begin{proof}[Proof of Theorem {\ref{staygraphthm}}]
We may assume $a=0$, $t_1=0$ and $\rho=2$.
Using \eqref{staygraphthmb} and Corollary \ref{heightboundcor}
we obtain
\begin{align}
\label{staygraphthm13}
\mathrm{spt}\mu_{t}\cap\mathbf{C}(0,3,3)\subset\mathbf{C}(0,3,7l_0)
\end{align}
for all $t\in [0,t_2]$.

Set $U:=\mathbf{B}^{\mathbf{n}}(0,2)\times\mathbb{R}^{\mathbf{k}}$.
By definition of a Brakke flow we find a sequence $(\tau_m)_{m\in\mathbb{N}}$ with $\tau_m\nearrow t_2$, $\tau_m\in (0,t_2]$
such that for all $m\in\mathbb{N}$ 
we have $\mu_m:=\mu_{\tau_m}\mres\mathbf{C}(0,2,2)$ is integer $\mathbf{n}$-rectifiable
and the generalised mean curvature vector $\mathbf{H}_{\mu_m}$ inside $U$ exists.
In particular $\|\partial \mu_m\|$ is absolutely continuous 
with respect to $\mu_m$.
Fix an arbitrary $m\in\mathbb{N}$.
We want to show 
\begin{align}
\label{staygraphthm21}
\mathrm{spt}\mu_m\cap U\subset\mathrm{graph}(f_m)
\end{align}
for some Lipschitz function $f_m:\mathbf{B}^{\mathbf{n}}(0,4)\to\mathbb{R}^{\mathbf{k}}$.

Let $x,y\in \mathrm{spt}\mu_m\cap U$ with $x\neq y$.
Set $y_0:=(\hat{y},u_0(\hat{y}))$.
We want to show $|\tilde{x}-\tilde{y}|\leq L|\hat{x}-\hat{y}|$
for some constant $L\in (1,\infty)$ which will depend on $l_0$.
By \eqref{staygraphthm13} we have $|\tilde{x}-\tilde{y}|\leq 7 l_0$.
Hence we may assume $|\hat{x}-\hat{y}|\leq l_0$.

First consider the case $\tau_m\leq 4|\hat{x}-\hat{y}|^2\leq 4l_0^2$
and let $z\in\mathrm{spt}\mu_m\cap U$.
Then $\mu_m(\mathbf{B}(z,2\sqrt{\mathbf{n}\tau_m}))>0$,
so by Lemma \ref{barrierlem} we have $\mu_0(\mathbf{B}(z,4\sqrt{\mathbf{n}\tau_m}))>0$.
Thus by \eqref{staygraphthmb} and $\mathrm{lip}(u_0)\leq l_0$ we have 
$|\tilde{z}-\tilde{y}_0|\leq l_0|\hat{z}-\hat{y}_0|+8\sqrt{\mathbf{n}\tau_m}$.
For $z=x,y$ this yields the wanted estimate.

Now consider the case $0<4|\hat{x}-\hat{y}|^2<\tau_m$.
Set $\epsilon:=|\hat{x}-\hat{y}|$.
By \eqref{staygraphthm13} we have 
$y\in\mathrm{spt}\mu_m\cap(\{\hat{y}_0\}\times\mathbf{B}^{\mathbf{k}}(0,1))$.
Set $s_m:=\tau_m-2\epsilon^2$.
Using Lemma \ref{staygraphlem} 
with $s_1=0$, $s_2=\tau_m$, $\varrho_0=1$, $\eta=\sqrt[4]{l_0}$
we obtain a $v_m\in\mathcal{C}^{\infty}(\mathbf{B}^{\mathbf{n}}(\hat{y},8\mathbf{n}\epsilon))$ 
with $\sup|Dv_m|\leq\sqrt[4]{l_0}\leq 1$ and
\begin{align}
\label{staygraphthm32}
\mathrm{spt}\mu_{s_m}\cap\mathbf{C}(y_0,8\mathbf{n}\epsilon,1)=\mathscr{H}^{\mathbf{n}}\mres\mathrm{graph}(v_m).
\end{align}
Consider $z\in\mathrm{spt}\mu_m\cap U$ with $|\hat{z}-\hat{y}|\leq\epsilon$.
Then $\mu_m(\mathbf{B}(z,2\sqrt{\mathbf{n}}\epsilon))>0$,
so by Lemma \ref{barrierlem} we have $\mu_{s_m}(\mathbf{B}(z,4\sqrt{\mathbf{n}}\epsilon))>0$.
In view of \eqref{staygraphthm13} 
we can use \eqref{staygraphthm32} to estimate
$|\tilde{z}-v_m(\hat{y})|\leq (1+8\mathbf{n})\epsilon$.
For $z=x,y$ this proves \eqref{staygraphthm21}.

Next we want to show that $\mu_m$ has unit density.
Let $y\in \mathrm{spt}\mu_m\cap U$ and $r\in (0,\sqrt{\tau_m})$ be given.
Set $s_r:=\tau_m-16\sqrt[32]{l_0}r^2$ and $y_0:=(\hat{y},u_0(\hat{y}))$.
Note that by \eqref{staygraphthm13} we have
$y\in\mathrm{spt}\mu_m\cap(\{\hat{y}_0\}\times\mathbf{B}^{\mathbf{k}}(0,1))$.
Using Lemma \ref{staygraphlem} with $s_1=0$, $s_2=\tau_m$, $\varrho_0=1$, $\eta=\sqrt[4]{l_0}$
we obtain a $v_r\in\mathcal{C}^{\infty}(\mathbf{B}^{\mathbf{n}}(\hat{y},4r))$ 
with $\sup|Dv_r|\leq\sqrt[4]{l_0}$ and
\begin{align}
\label{staygraphthm41}
\mathrm{spt}\mu_{s_r}\cap\mathbf{C}(y_0,4r,1)=\mathscr{H}^{\mathbf{n}}\mres\mathrm{graph}(v_r),
\end{align}
Consider a radial cut-off function 
$\zeta_r\in\mathcal{C}^{\infty}_{\mathrm{c}}\left(\mathbb{R}^{\mathbf{n}+\mathbf{k}},[0,1]\right)$
with $\sup|D^{2}\zeta_{r}|\leq Cr^{-2}$ and
\begin{align*}
\zeta_{r}(x)=
\begin{cases}
1 &\text{for}\;\; 0\leq\left|x-y\right|\leq r\\
0 &\text{for}\;\; (1+2^{-\mathbf{n}-2})r\leq\left|x-y\right|.
\end{cases}
\end{align*}
Using equation \eqref{brakkeflowa}, Remark \ref{brakkevarbound}
and Lemma \ref{barrierlem} we estimate
\begin{align*}
&\mu_{m}\left(\mathbf{B}(y,r)\right)-\mu_{s_r}\left(\mathbf{B}(y,(1+2^{-\mathbf{n}-2})r)\right)
\\&\leq
\mu_{m}(\zeta_r)-\mu_{s_r}(\zeta_r)
\leq
C\int_{s_r}^{\tau_m}\left(\sup|D^2\zeta_r|\,\mu_{t}(\{\zeta>0\})\right)\,\mathrm{d}t
\\&\leq
C\sqrt[32]{l_0}\sup\{\mu_{t}(\mathbf{B}(y,2r)),t\in [s_r,\tau_m]\}
\leq
C\sqrt[32]{l_0}\mu_{s_r}\left(\mathbf{B}(y,4r)\right).
\end{align*}
In view of \eqref{staygraphthm13} we have 
$|\tilde{y}-\tilde{y}_0|\leq 1/4<1-4r$,
hence we can use \eqref{staygraphthm41} and the above estimate to obtain
\begin{align*}
\mu_{m}\left(\mathbf{B}(y,r)\right)
\leq
(1+C\sqrt[4]{l_0})\left(C\sqrt[32]{l_0}(4r)^{\mathbf{n}}+((1+2^{-\mathbf{n}-2})r)^{\mathbf{n}}\right)
\leq \frac{3}{2}r^{\mathbf{n}},
\end{align*}
where we chose $l_0$ small enough.
As we already know $\mu_{m}$ is integer rectifiable, 
this shows that $\mu_{m}$ even has unit density in $U$.
Also, by \eqref{staygraphthma} and Lemma \ref{barrierlem} we have
$\mathrm{spt}\mu_m\cap U\neq\emptyset$.
Then Theorem \ref{nobdrythm}
yields that in \eqref{staygraphthm21} actually holds equality.
Hence
\begin{align*}
\mu_{\tau_m}\mres\mathbf{C}(0,2,2)=\mathscr{H}^{\mathbf{n}}\mres\mathrm{graph}f_m
\end{align*}
for all $m\in\mathbb{N}$,
for some Lipschitz function $f_m:\mathbf{B}^{\mathbf{n}}(0,2)\to\mathbb{R}^{\mathbf{k}}$.
In view of this and \eqref{staygraphthm13}
we can use Lemma \ref{staygraphlem} with $s_2=\tau_m$, $\varrho_0=1$
and arbitrary $y_0\in\mathbf{B}^{\mathbf{n}}(0,2)\times\{0\}^{\mathbf{k}}$
to obtain graphical representability inside $\mathbf{C}(0,2,2)$ for times in $(0,\tau_m)$.
As $\tau_m\nearrow t_2$ this actually holds on $(0,t_2)$.
Finally for the Lipschitz bound consider $(t,\hat{y})\in (0,t_2)\times\mathbf{B}^{\mathbf{n}}(0,2)$.
Lemma \ref{staygraphlem} with $y_0:=(\hat{y},u_0(\hat{y}))$, $\rho_0=1$, $s_2= t+\epsilon$ and $\eta=\sqrt[4]{l}+t+\epsilon$
yields $|Du(t,\hat{y})|\leq\sqrt[4]{l}+t+\epsilon$ for all $\epsilon\in (0,t_2-t)$.
Letting $\epsilon\searrow 0$ completes the result.
\end{proof}

\section{Brakke-type local regularity}
\label{local_regularity}
Here we proof Theorem \ref{locregthm}.
First note that under slightly stronger assumptions on the starting density ratios
the result directly follows from 
Lemma \ref{improvedgohlem} and Theorem \ref{staygraphthm},
see below:

%
%
%
%
\begin{lem}
\label{becomegraphlem}
There exists a constant $\sigma_5\in (0,1)$
and for every $\kappa\in (0,1)$ exists an $h_2\in (0,\sigma_5^2)$
such that the following holds:

Let $h\in (0,h_2]$, $\varrho\in (0,\infty)$, 
$s_1\in\mathbb{R}$, $s_2\in (s_1+\sqrt{h}\varrho^2,s_1+\sigma_5\varrho^2]$,
$x_0\in\mathbb{R}^{\mathbf{n}+\mathbf{k}}$
and let $(\mu_t)_{t\in [s_1,s_2]}$ be a Brakke flow in $\mathbf{C}(x_0,2^5\varrho,2^5\varrho)$.
Suppose $x_0\in\mathrm{spt}\mu_{s_2}$,
\begin{align}
\label{becomegraphthma}
&\mathrm{spt}\mu_{s_1}\cap\mathbf{C}(x_0,2^5\varrho,2^5\varrho)
\subset\mathbf{C}(x_0,2^5\varrho,h\varrho),
\\
\label{becomegraphthmc}
&r^{-\mathbf{n}}\mu_{s_1}(\mathbf{B}(y,r))
\leq (2-\kappa)\omega_{\mathbf{n}}
\end{align}
for all $y\in\mathbf{B}^{\mathbf{n}}(\hat{x}_0,2^{4}\varrho)\times\{\tilde{x}_0\}$ 
and all $r\in (h\varrho,\varrho)$.
Set $I:=(s_1+\sqrt{h}\varrho^2,s_2)$.

Then there exists a 
$v\in\mathcal{C}^{\infty}\left(I\times\mathbf{B}^{\mathbf{n}}(\hat{x}_0,\varrho),\mathbb{R}^{\mathbf{k}}\right)$
such that 
\begin{align*}
\mu_{t}\mres\mathbf{C}(x_0,\varrho,\varrho)=\mathscr{H}^{\mathbf{n}}\mres\mathrm{graph}(v(t,\cdot\,))
\quad\text{for all}\;t\in I.
\end{align*}
Moreover $\sup |Dv(t,\cdot\,)|\leq\sqrt[16]{h}+\varrho^{-2}(t-s_1)$ for all $t\in I$
and $F_t(\hat{x})=(\hat{x},v(t,\hat{x}))$ satisfies \eqref{smoothmcf}.
\end{lem}
\begin{proof}
We may assume $x_0=0$, $s_1=0$ and $\varrho=1$.
First note that by assumption \eqref{becomegraphthma}, 
$s_2\leq\sigma_5$, $h\leq h_2<\sigma_5$ 
and Corollary \ref{heightboundcor} with $R_1=R_2=r_0=16$ we have
\begin{align}
\label{becomegraphthm11}
&\mathrm{spt}\mu_{t}\cap\mathbf{C}(0,16,16)
\subset\mathbf{C}(0,16,2\sigma_5),
\end{align}
for all $t\in [0,s_2]$.

Let $\sigma_2$ be from Lemma \ref{improvedgohlem} with respect to $\kappa$
and set $s_4:=\sqrt{h}<s_2$, $J:=(\sqrt{h}/4,s_4-\sqrt{h}/4)$.
Lemma \ref{barrierlem} and $0\in\mathrm{spt}\mu_{s_2}$ yield
the existence of a $z_0\in\mathrm{spt}\mu_{s_4}\cap\mathbf{B}(0,1)$.
By Lemma \ref{improvedgohlem} 
with $\varrho_0=4$, $y_0=(\hat{z}_0,0)$, $\beta=2^4\sigma_2^{-1}\sqrt{h}$ and $\lambda=2^{-6}\sqrt{h}$
there exists a $v_1\in\mathcal{C}^{\infty}(J\times\mathbf{B}^{\mathbf{n}}(\hat{z}_0,8),\mathbb{R}^{\mathbf{k}})$
such that
\begin{align*}
\mu_t\mres\mathbf{C}((\hat{z}_0,0),8,4)=\mathscr{H}^{\mathbf{n}}\mres\mathrm{graph}(v_1(t,\cdot\,))
\;\;\;\text{for all}\;t\in J.
\end{align*}
Moreover 
$\sup |Dv_1|\leq 2^4\sigma_2^{-1}\sqrt{h}$.
Here we chose $h_2\leq\sigma_2^{-2}$. 
Then \eqref{becomegraphthm11}, $|\hat{z}_0|\leq 1$
and Theorem \ref{staygraphthm}
with $t_1=\sqrt{h}/2$, $a=(0,v_1(t_1,0))$, $l= 2^4\sigma_2^{-1}\sqrt{h}$ and $\rho=2$ yield the result.
Here we chose $h_2$ small depending on $\sigma_2$ and $l_0$.
\end{proof}

Now under the assumptions of Theorem \ref{locregthm}
we can find a time $s_1$ shortly after $t_1$ such that
$\mu_{s_1}\mres\mathbf{C}(a,\rho,\rho)$ 
has bounded mean-curvature-excess and still has small height.
By Lemma \ref{tiltboundlem} then also the tilt-excess has to be small.
Thus Brakke's cylindrical growth theorem (see Theorem \ref{cylgrowththm})
can be used to show that the density assumptions of Lemma \ref{becomegraphlem} hold, 
which then yields the conclusion of Theorem \ref{locregthm}.

%
%
%
%
\begin{proof}[proof of Theorem {\ref{locregthm}}]
We may assume $a=0$, $t_1=0$ and $\rho=1$.
First consider the case $\gamma> 0$.
Set $U:=\mathbf{B}(0,1)$
and $C(x,r):=\mathbf{C}(x,r)\cap U$ for $r\in (0,\infty)$, $x\in U$.
In view of assumption \eqref{locregthma}
and as $\mathbf{C}(0,\sqrt{2},\sqrt{2})\subset\mathbf{B}(0,2)$ 
we can use Corollary \ref{heightboundcor} with $r_0=\sqrt{2}-1$ and $p=4$
to obtain
\begin{align}
\label{locregthm11}
\mathrm{spt}\mu_{t}\cap\mathbf{C}(0,1,1)
\subset\mathbf{C}(0,1,2\gamma)
\end{align}
for all $t\in [0,\sqrt[4]{\gamma}]$
for $\gamma_0$ small enough.
Fix a $\sigma\in (0,2^{-5})$ 
such that $(1-8\sigma)^{-\mathbf{n}}\leq 1+\lambda/8$
and $(1+4\sigma)^{\mathbf{n}}\leq 1+\lambda/32$.
In particular in the following $\gamma_0$ may depend on $\sigma$.
By Lemma \ref{barrierlem} and assumption \eqref{locregthmb} we can estimate
\begin{align}
\label{locregthm21}
\mu_{t}\left(\mathbf{B}(0,1-\sigma)\right)
\leq C\sigma^{-\mathbf{n}-\mathbf{k}}\mu_{0}\left(\mathbf{B}(0,1)\right)
\leq C\sigma^{-\mathbf{n}-\mathbf{k}}
\end{align}
for all $t\in [0,\sqrt[4]{\gamma}]$.

Fix a cut-off function $\psi\in\mathcal{C}^{\infty}\left(\mathbb{R},[0,1]\right)$
with $|\psi''|\leq C\sigma^{-2}$ and
\begin{align*}
\psi(t)=
\begin{cases}
1 &\text{for}\;\; 0\leq|t| \leq 1-2\sigma \\
0 &\text{for}\;\; 1-\sigma\leq|t| .
\end{cases}
\end{align*}
Consider $\zeta\in\mathcal{C}^{\infty}_{\mathrm{c}}\left(\mathbf{B}(0,1),[0,1]\right)$
given by $\zeta(x)=\psi(|x|)$.
For $s\in (0,\sqrt[4]{\gamma}]$ equation \eqref{brakkeflowa} and Remark \ref{brakkevarbound} yield
\begin{align*}
D&:=\mu_{s}\left(\zeta\right)
+\frac{1}{2}\int_0^{s}\int_{\mathbb{R}^{\mathbf{n}
+\mathbf{k}}}|\mathbf{H}_{\mu_t}|^2\zeta\,\mathrm{d}\mu_t\,\mathrm{d}t
\\&\leq 
\mu_{0}\left(\zeta\right)
+\sup|D^2\zeta|\int_0^{s}\mu_{t}\left(\{\zeta>0\}\right)\,\mathrm{d}t.
\end{align*}
Hence by \eqref{locregthmb} and \eqref{locregthm21} we have
\begin{align*}
D\leq (2-\lambda)\omega_{\mathbf{n}}+Cs\sigma^{-\mathbf{n}-\mathbf{k}-2}
\leq (2-\lambda/2)\omega_{\mathbf{n}},
\end{align*}
where we used $s\leq\sqrt[4]{\gamma_0}$
and we chose $\gamma_0$ small enough.
By \eqref{locregthm11} we have
$\{\zeta=1\}\supset\mathbf{B}(0,1-2\sigma)\supset\mathrm{spt}\mu_s\cap C(0,1-4\sigma)$,
for $\gamma_0\leq\sigma$.
Thus
\begin{align}
\label{locregthm31}
\mu_{s}\left(C(0,1-4\sigma)\right)
+\frac{1}{2}\int_0^{s}\int_{C(0,1-4\sigma)}|\mathbf{H}_{\mu_t}|^2\;\mathrm{d}\mu_t\;\mathrm{d}t
\leq D\leq (2-\lambda/2)\omega_{\mathbf{n}}
\end{align}
for all $s\in (0,\sqrt[4]{\gamma}]$.
In particular we find an $s_1\in (0,\sqrt[4]{\gamma}]$ such that
$\mu_{s_1}\mres U$ is integer $\mathbf{n}$-rectifiable, 
has $\mathcal{L}^2$-integrable mean curvature vector
and
\begin{align}
\label{locregthm32}
\int_{C(0,1-4\sigma)}|\mathbf{H}_{\mu_{s_1}}|^2\;\mathrm{d}\mu_{s_1}
\leq 2(2-\lambda/2)\omega_{\mathbf{n}}\gamma^{-1/4}
\leq C\gamma^{-1/4}.
\end{align}
Consider $y\in\mathbf{B}(0,\sigma)$.
By \eqref{locregthm31} and choice of $\sigma$ we can estimate
\begin{align}
\label{locregthm41}
\mu_{s_1}(C(y,1-8\sigma))
\leq (2-\lambda/2)\omega_{\mathbf{n}}
\leq (2-\lambda/4)\omega_{\mathbf{n}} (1-8\sigma)^{\mathbf{n}}.
\end{align}

Let $f\in\mathcal{C}^{\infty}_{\mathrm{c}}\left(C(0,1-4\sigma),[0,1]\right)$
be such that $f(x)=\psi((1-4\sigma)^{-1}|\hat{x}|)$ for $x\in\mathrm{spt}\mu_{s_1}\cap U$.
In view of \eqref{locregthm11},\eqref{locregthm31} and \eqref{locregthm32} 
we can use Lemma \ref{tiltboundlem}
with $f=g=h$ to obtain
\begin{align}
\label{locregthm51}
\int_{C(0,1-6\sigma)}\left\|(\mathbb{R}^{\mathbf{n}}\times\{0\}^{\mathbf{k}})_{\natural}
-\mathbf{T}(\mu_{s_1},x)_{\natural}\right\|^2\;\mathrm{d}\mu_{s_1}(x)
\leq C \gamma^{7/8}.
\end{align}
Here we estimated $\sup|Df|^2\leq C\sigma^{-2}\leq C\gamma_0^{-1}\leq C\gamma^{-1}$.

Consider $y\in\mathbf{B}(0,\sigma)$
and $r_0\in (2^{-5}\gamma^{16\alpha_0}\sigma,\sigma)$.
Let $R_2=1-8\sigma$
and $R_1=(1+4\sigma)r_0$.
By \eqref{locregthm32} and \eqref{locregthm51}
the assumptions of Theorem \ref{cylgrowththm} are satisfied 
for $\alpha^2=C\gamma^{-1/4} (\gamma^{16\alpha_0}\sigma)^{-\mathbf{n}}$ 
and $\beta^2=C\gamma^{7/8} (\gamma^{16\alpha_0}\sigma)^{-\mathbf{n}}$.
Hence we can estimate
\begin{align*}
&\left|R_2^{-\mathbf{n}}\int_{U}\psi(R_2^{-1}|\hat{x}-\hat{y}|)\,\mathrm{d}\mu_{s_1}(x)
-R_1^{-\mathbf{n}}\int_{U}\psi(R_1^{-1}|\hat{x}-\hat{y}|)\,\mathrm{d}\mu_{s_1}(x)\right|
\\
&\leq
C\gamma^{-1/8} (\gamma^{16\alpha_0}\sigma)^{-\mathbf{n}}\gamma^{7/16}
\leq C\sigma^{-\mathbf{n}}\gamma_0^{3/16}
\leq \lambda\omega_{\mathbf{n}}/8,
\end{align*}
where we chose $\alpha_0$ and $\gamma_0$ small enough.
By estimate \eqref{locregthm41}, definition of $\psi$ and $\sigma$ this yields
\begin{align*}
((1+4\sigma)r_0)^{-\mathbf{n}}\mu_{s_1}\left(\mathbf{B}(y,r_0)\right)
\leq
(2-\lambda/8)\omega_{\mathbf{n}}
\leq (2-\lambda/16)\omega_{\mathbf{n}}(1+4\sigma)^{-\mathbf{n}}.
\end{align*}
Now we can use Lemma \ref{becomegraphlem}
with $s_2=t_2$, $\kappa=\lambda/16$, $\varrho=2^{-5}\sigma$ and $h=\gamma^{16\alpha_0}$
to obtain the result. 
For the case $\gamma=0$ use the above result with arbitrary small~$\gamma$.
\end{proof}

\section{White-type local regularity}
\label{white_regularity}
Here we want to prove Theorem \ref{whitelocregthm}.
First we observe that a Brakke flow for which all Gaussian density ratios are one,
has to be a plane.
This mainly follows from Huisken's monotonicity formula (see Theorem \ref{localmonthm}).

%
%
%
%
\begin{lem}
\label{globalgaussianlem}
Let $M\in (1,\infty)$, $t_1\in\mathbb{R}$, $t_2\in (t_1,\infty)$
and $(\mu_t)_{t\in [t_1,t_2]}$ be a Brakke flow in $\mathbb{R}^{\mathbf{n}+\mathbf{k}}$.
Suppose $\mathrm{spt}\mu_{t_2}\neq\emptyset$ 
and for all $(s,y)\in(t_1,t_2]\times\mathbb{R}^{\mathbf{n}+\mathbf{k}}$
\begin{align}
\label{globalgaussianlemb}
\sup_{R\in (0,\infty)} R^{-\mathbf{n}}\mu_{s}(\mathbf{B}(y,R))
&\leq M,
\\
\label{globalgaussianlema}
\sup_{t\in [t_1,s)}\int_{\mathbb{R}^{\mathbf{n}+\mathbf{k}}}\Phi_{(s,y)}
\,\mathrm{d}\mu_{t}
&\leq 1.
\end{align}

Then there exists a $T\in\mathbf{G}(\mathbf{n}+\mathbf{k},\mathbf{n})$
and an $a\in\mathbb{R}^{\mathbf{n}+\mathbf{k}}$ such that
$\mu_t=\mathscr{H}^{\mathbf{n}}\mres(T+a)$ for all $t\in (t_1,t_2)$.
\end{lem}
%
%
\begin{proof}
We may assume $t_1=-1$ and $t_2=0$.
For $t\in (-1,0)$ 
let $D(t)$ be the set of all $y\in\mathrm{spt}\mu_{t}$
such that $\Theta^{\mathbf{n}}(\mu_{t},y)\geq 1$
and $\mathbf{T}(\mu_{t},y)$ exists.
Fix $s\in (-1,0)$ and $y\in D(s)$.
For $\epsilon\in (0,1)$ there exist a radial symmetric cut-off function 
$\zeta\in\mathcal{C}^0_{\mathrm{c}}\left(\mathbb{R}^{\mathbf{n}+\mathbf{k}},[0,1]\right)$
such that
\begin{align}
\label{globalgaussianlem21}
\int_{\mathbb{R}^{\mathbf{n}}\times\{0\}^{\mathbf{k}}}\Phi_{(0,0)}(-1,x)\zeta(x)
\;\mathrm{d}\mathscr{H}^{\mathbf{n}}(x)
\geq
1-\epsilon,
\end{align}
By \eqref{globalgaussianlem21} and definition of the approximate tangent space 
we can estimate
\begin{align*}
(1-\epsilon)\Theta^{\mathbf{n}}(\mu_{s},y)
&\leq 
\Theta^{\mathbf{n}}(\mu_{s},y)\int_{\mathbf{T}(\mu_{s},y)}\Phi_{(0,0)}(-1,x)\zeta(x)\;\mathrm{d}\mathscr{H}^{\mathbf{n}}(x)
\\
&\leq
\lim_{\lambda\searrow 0}\lambda^{-\mathbf{n}}\int_{\mathbb{R}^{\mathbf{n}+\mathbf{k}}}\Phi_{(0,0)}(-1,\lambda^{-1}(x-y))\;\mathrm{d}\mu_s(x)
\\
&=
\lim_{\lambda\searrow 0}\int_{\mathbb{R}^{\mathbf{n}+\mathbf{k}}}\Phi_{(s+\lambda^2,y)}(s,x)\;\mathrm{d}\mu_s(x).
\end{align*}
Then with Huisken's monotonicity formula (see Theorem \ref{localmonthm}), 
and by continuity of the integral we obtain for $h_0\in (0,s+1)$ small enough
\begin{align*}
(1-\epsilon)\Theta^{\mathbf{n}}(\mu_{s},y)
\leq 
\lim_{\lambda\searrow 0}\int_{\mathbb{R}^{\mathbf{n}+\mathbf{k}}}\Phi_{(s+\lambda^2,y)}\;\mathrm{d}\mu_{s-h_0}
\leq
\lim_{h\searrow 0}\int_{\mathbb{R}^{\mathbf{n}+\mathbf{k}}}\Phi_{(s,y)}\;\mathrm{d}\mu_{s-h}
+\epsilon,
\end{align*}
Thus by \eqref{globalgaussianlema} and as $\epsilon$ was arbitrary we have
\begin{align}
\label{globalgaussianlem23}
1\leq\Theta^{\mathbf{n}}(\mu_{s},y)
\leq\lim_{h\searrow 0}\int_{\mathbb{R}^{\mathbf{n}+\mathbf{k}}}\Phi_{(s,y)}\;\mathrm{d}\mu_{s-h}
\leq 1
\end{align}
for all $y\in D(s)$ for all $s\in (-1,0)$.
Hence $\mu_{t}$ has unit density for a.e.\ $t\in (-1,0)$.

Fix an arbitrary $t_0\in (-1,0)$
such that $\mu_{t_0}$ has unit density.
Assumption $\mathrm{spt}\mu_{0}\neq\emptyset$ 
and Lemma \ref{barrierlem} imply $\mathrm{spt}\mu_{t_0}\neq\emptyset$,
so we can find $\mathbf{n}+1$ points $y_0,\ldots,y_{\mathbf{n}}$ in $D(t_0)$
such that $v_i:=y_i-y_0$,  $i=1,\ldots,\mathbf{n}$ are linearly independent.
Set $T:=\mathrm{span}(v_i)_{1\leq i\leq\mathbf{n}}$.
By estimates \eqref{globalgaussianlema}, \eqref{globalgaussianlem23}
and Theorem \ref{localmonthm} we obtain
\begin{align*}
\int_{\mathbb{R}^{\mathbf{n}+\mathbf{k}}}\Phi_{(t_0,y_i)}\;\mathrm{d}\mu_{t}=1
\end{align*}
for all $t\in [-1,t_0)$ for all $i\in\{0,\ldots,\mathbf{n}\}$.

Then Theorem \ref{localmonthm} yields the existence of a $J\subset (-1,t_0)$
such that for all $t\in J$ we have
$\mu_t$ has unit density,
the generalised mean curvature vector $\mathbf{H}_{\mu_t}$ exists
with $\int|\mathbf{H}_{\mu_t}|^2\mu_t< \infty$ and
\begin{align}
\label{globalgaussianlem11}
\mathbf{H}_{\mu_t}(x)+(2(t_0-t))^{-1}(\mathbf{T}(\mu_t,x)^{\bot})_{\natural}(x-y_i)=0
\end{align}
for $\mu_t$-a.e.\ $x\in\mathbb{R}^{\mathbf{n}+\mathbf{k}}$
and all $i=0,1,\ldots,\mathbf{n}$.
Moreover $\mathscr{L}^1((-1,t_0)\setminus J)=0$.

Let $t\in J$ and let $E_t$ be the set of points $x\in\mathrm{spt}\mu_t$
such that $\Theta^{\mathbf{n}}(\mu_{t},x)\geq 1$,
$\mathbf{T}(\mu_{t},x)$ exists 
and \eqref{globalgaussianlem11} holds for all $i\in\{0,\ldots,\mathbf{n}\}$.
We see $\mu_t(\mathbb{R}^{\mathbf{n}+\mathbf{k}}\setminus E_t)=0$.
Consider $x\in E_t$ then by \eqref{globalgaussianlem11}
we have
\begin{align*}
(\mathbf{T}(\mu_t,x)^{\bot})_{\natural}(y_0-y_i)
=(\mathbf{T}(\mu_t,x)^{\bot})_{\natural}(x-y_i)-(\mathbf{T}(\mu_t,x)^{\bot})_{\natural}(x-y_0)=0
\end{align*}
for all $i\in\{1,\ldots,\mathbf{n}\}$.
So $v_i=y_i-y_0\in\mathbf{T}(\mu_t,x)$ for all $i\in\{1,\ldots,\mathbf{n}\}$,
hence $\mathbf{T}(\mu_t,x)=\mathrm{span}(v_i)_{1\leq i\leq\mathbf{n}}=T$.
As this holds for all $x\in E_t$ for all $t\in J$,
we have $\mathbf{H}_{\mu_t}\equiv 0$ for a.e.\ $t\in (-1,t_0)$.
This follows from Brakke's general regularity theorem \cite[Thm.\ 6.12]{brakke}
(see also \cite[Thm.\ 3.2]{kasait}).
Or deduce this from Menne's 
characterization of the mean curvature vector \cite[Thm.\ 15.6]{menne5}.

Now for a.e.\ $t\in (-1,t_0)$ we have $t\in J$ and $\mathbf{H}_{\mu_t}\equiv 0$,
so equality \eqref{globalgaussianlem11} with $i=0$ yields $E_t\subset T+y_0$.
Thus $\mathrm{spt}\mu_t\subset T+y_0$.
Then by Theorem \ref{nobdrythm} we have $\mu_t=\mathscr{H}^{\mathbf{n}}\mres (T+y_0)$.
As this holds for a.e.\ $t$ in $(-1,t_0)$
and by the continuity properties of the Brakke flow due to Brakke \cite[Thm.\ 3.10]{brakke}
we obtain $\mu_t=\mathscr{H}^{\mathbf{n}}\mres (T+y_0)$ for all $t\in (-1,t_0)$.
Finally choose $t_0$ arbitrary close to $0$ to establish the result.
\end{proof}

Now suppose the Gaussian density ratios are locally bounded by $1+\delta$.
In view of the previous Lemma \ref{globalgaussianlem} an indirect blow-up argument
combined with Ilmanen's compactness theorem (see Theorem \ref{compactnessthm}),
yields a small neighbourhood in which we have small height
and density ratios close to one, see below.
In view of Theorem \ref{locregthm} this implies
Theorem \ref{whitelocregthm}.

%
%
%
%
\begin{lem}
\label{localgaussianlem}
For all $\epsilon,\sigma\in (0,1)$ there exists a $\delta\in (0,1)$
such that the following holds:

Let $\rho,\rho_0\in (0,\infty)$ and let $(\mu_t)_{t\in [-\rho^2,0]}$ be a Brakke flow in $\mathbf{B}(0,4\sqrt{\mathbf{n}}\rho_0)$.
Suppose $\rho\leq\rho_0$, $0\in\mathrm{spt}\mu_{0}$ 
and for all $(s,y)\in(-\rho^2,0]\times\mathbf{B}(0,\rho)$ we have
\begin{align}
\label{localgaussianlemc}
\int_{\mathbb{R}^{\mathbf{n}+\mathbf{k}}}\Phi_{(s,y)}\varphi_{(s,y),\rho_0}\, d\mu_{-\rho^2}
\leq 1+\delta.
\end{align}

Then there exists a $T\in\mathbf{G}(\mathbf{n}+\mathbf{k},\mathbf{n})$ 
such that we have
\begin{align*}
&\sup\left\{|(T^{\bot})_{\natural}(x)|, x\in\mathrm{spt}\mu_{-\sigma\delta^2\rho^2}\cap\mathbf{B}(0,2\delta\rho)\right\}
\leq\epsilon\delta\rho
\\
&\text{and}\quad
\mu_{-\sigma\delta^2\rho^2}(\mathbf{B}(0,\delta\rho))\leq\omega_{\mathbf{n}}(1+\epsilon)(\delta\rho)^{\mathbf{n}}.
\end{align*}
\end{lem}
\begin{proof}
We may assume $\rho=\delta^{-1}$.
Suppose the statement would be false.
Then there exist $\epsilon,\sigma\in (0,1)$ and for every $j\in\mathbb{N}$ 
we find an $\rho_j\in [j,\infty)$
and a Brakke flow $(\nu_t^j)_{t\in [-j^2,0]}$ in 
$B_j:=\mathbf{B}(0,4\sqrt{\mathbf{n}})\rho_j)$
such that $0\in\mathrm{spt}\nu^j_0$,
\begin{align}
\label{localgaussianlem24}
\sup_{t\in [-j^2,s)}
\int_{\mathbb{R}^{\mathbf{n}+\mathbf{k}}}
\Phi_{(s,y)}\varphi_{(s,y),\rho_j}\;\mathrm{d}\nu^j_{t}
=
\int_{\mathbb{R}^{\mathbf{n}+\mathbf{k}}}
\Phi_{(s,y)}\varphi_{(s,y),\rho_j}\;\mathrm{d}\nu^j_{-j^2}
\leq 1+\frac{1}{j}
\end{align}
for all $(s,y)\in (-j^2,0]\times\mathbf{B}(0,j)$
and one of the following holds
\begin{align}
\label{localgaussianlem25}
\inf_{T\in\mathbf{G}(\mathbf{n}+\mathbf{k},\mathbf{n})}
\sup\left\{|(T^{\bot})_{\natural}(x)|, x\in\mathrm{spt}\nu^j_{-\sigma}\cap \mathbf{B}(0,2)\right\}
>\epsilon,
\\
\label{localgaussianlem26}
\nu^j_{-\sigma}(\mathbf{B}(0,1))
> \omega_{\mathbf{n}}(1+\epsilon).
\end{align}
Note that the equality in \eqref{localgaussianlem24}
follows from Huisken's monotonicity formula (see Theorem \ref{localmonthm}).

To obtain a converging subsequence of the $(\nu^j_t)$ we need uniform bounds on the measure of compact sets.
We claim that for every $R\in (0,\infty)$ we can find a $D(R)\in (0,\infty)$ such that
\begin{align}
\label{localgaussianlem31}
\sup_{j\geq 2}\sup_{t\in [-2,0]}\nu^j_{t}(\mathbf{B}(0,R)\cap\mathbf{B}(0,j/2))
&\leq D(R).
\end{align}
Before we prove this, we first show
\begin{align}
\label{localgaussianlem32}
\sup_{j\geq 2}\sup_{t\in [-2,0]}\sup_{y\in\mathbf{B}(0,j)}\sup_{R\in (0,j/4]}\nu^j_{t}(\mathbf{B}(y,R))
&\leq C R^{\mathbf{n}}
\end{align}
To see \eqref{localgaussianlem32} consider $R\leq j/4$, $x\in \mathbf{B}(y,2R)$, $t\in [-2,0]$
and $c_2:=(2\mathbf{n})^{-1}$.
Then we can estimate $\Phi_{(t,y)}(t-c_2R^2,x)\geq (4\pi c_2 R^2)^{-\frac{\mathbf{n}}{2}}\exp(-1/c_2)$
as well as $\varphi_{(t,y),\rho_j}(t-c_2R^2,x)\geq(1-1/4)^3$.
Thus Lemma \ref{barrierlem} and assumption \eqref{localgaussianlem24} yield
\begin{align*}
\nu^j_{t}(\mathbf{B}(y,R))
&\leq
C\nu^j_{t-c_2R^2}(\mathbf{B}(y,2R))
\\&\leq
CR^{\mathbf{n}}\int_{\mathbb{R}^{\mathbf{n}+\mathbf{k}}}
\Phi_{(t,y)}\varphi_{(t,y),\rho_j}\;\mathrm{d}\nu^j_{t-c_2R^2}
\leq 
CR^{\mathbf{n}}.
\end{align*}
To prove \eqref{localgaussianlem31} let $R\in (0,\infty)$.
For $j\leq 4R$ we can use Lemma \ref{barrierlem} to estimate
\begin{align*}
\nu^j_{t}(\mathbf{B}(0,j/2))
\leq \nu^j_{t}(\mathbf{B}(0,4\sqrt{\mathbf{n}}j))
\leq \nu^j_{-2}(\mathbf{B}(0,8\sqrt{\mathbf{n}}j))=:D_j(R)
\end{align*}
for all $t\in [-2,0]$.
Combined with \eqref{localgaussianlem32} this proves \eqref{localgaussianlem31}.

Now we can use the 
compactness theorem by Ilmanen, 
(see Theorem \ref{compactnessthm})
with $U_j=\mathbf{B}(0,j/2)$,
to see that a subsequence of the $(\nu_t^j)$
converges to a Brakke flow $(\nu_t)_{t\in [-2,0]}$ in $\mathbb{R}^{\mathbf{n}+\mathbf{k}}$.
Note that without loss of generality we  will assume that the whole sequence converges.
In particular
\begin{align}
\label{localgaussianlem41}
\nu_{t}(\phi)=\lim_{j\to\infty}\nu^j_{t}(\phi)
\quad\text{for all}\;\phi\in\mathcal{C}_{\mathrm{c}}^0(\mathbf{B}(0,j_0/2))
\end{align}
for all $t\in[-2,0]$ and all $j_0\in\mathbb{N}$.
Combining this with \eqref{localgaussianlem32} yields
\begin{align}
\label{localgaussianlem42}
\sup_{t\in[-2,0]}\sup_{y\in\mathbb{R}^{\mathbf{n}+\mathbf{k}}}\sup_{R\in (0,\infty)}\nu_{t}(\mathbf{B}(y,R))
\leq C R^{\mathbf{n}}.
\end{align}

Next we want to show
\begin{align}
\label{localgaussianlem51}
\int_{\mathbb{R}^{\mathbf{n}+\mathbf{k}}}\Phi_{(s,y)}\;\mathrm{d}\nu_{t}\leq 1
\end{align}
for all $(s,y)\in (-2,0]\times\mathbb{R}^{\mathbf{n}+\mathbf{k}}$ and all $t\in [-2,s)$.
To see this fix $s,y,$ and $t$ like that.
First we see that by \eqref{localgaussianlem42} we have
\begin{align}
\label{localgaussianlem52}
\int_{\mathbb{R}^{\mathbf{n}+\mathbf{k}}}\Phi_{(s,y)}\;\mathrm{d}\nu_{t}<\infty.
\end{align}
In order to prove \eqref{localgaussianlem52} consider
$f_l:\mathbb{R}^{\mathbf{n}+\mathbf{k}}\to\mathbb{R}^+$ 
given by $f_l(x):=\Phi_{(s,y)}(t,x)$ for $|x-y|< l$ and $f_l\equiv 0$ outside $\mathbf{B}(y,l)$.
Obviously we have $f_{l+1}\geq f_l$.
Use \eqref{localgaussianlem42}
to estimate $\nu_t(\mathbf{B}(y,2l))\leq 2C_1 (2l)^{\mathbf{n}}$ for all $l\in\mathbb{N}$.
Then for $l\geq l_0$ we observe
\begin{align*}
&\int_{\mathbb{R}^{\mathbf{n}+\mathbf{k}}}f_{l+1}\;\mathrm{d}\nu_{t}
-\int_{\mathbb{R}^{\mathbf{n}+\mathbf{k}}}f_l\;\mathrm{d}\nu_{t}
\leq
\int_{\mathbf{B}(y,l+1)\setminus\mathbf{B}(y,l)}\Phi_{(s,y)}\;\mathrm{d}\nu_{t}
\\
&\leq C(s-t)^{-\mathbf{n}/2}\exp(-l^2/(4(s-t)))\nu_t(\mathbf{B}(y,2l))
\leq l^{-3-\mathbf{n}}\nu_t(\mathbf{B}(y,2l))
\leq l^{-2},
\end{align*}
where we chose $l_0$ large enough depending on $s-t$.
Thus $\lim_{l\to \infty}\int f_ld\nu_t<\infty$
and the monotone convergence theorem implies \eqref{localgaussianlem52}.

We continue to prove \eqref{localgaussianlem51}.
Let $\gamma\in (0,1)$ be arbitrary.
Note that $y,s,t$ are still fixed.
Using \eqref{localgaussianlem52} and \eqref{localgaussianlem41} 
we find $j_1,j_2,j_3\in\mathbb{N}$, $j_3> j_2> j_1$ 
such that $\rho_{j_2}\geq \rho_{j_1}+1$,
\begin{align*}
\int_{\mathbb{R}^{\mathbf{n}+\mathbf{k}}\setminus\mathbf{B}(y,\rho_{j_1})}
\Phi_{(s,y)}\;\mathrm{d}\nu_{t}\leq\gamma,
\\
\int_{\mathbf{B}(y,\rho_{j_1})}\Phi_{(s,y)}\;\mathrm{d}\nu_{t}
-\int_{\mathbf{B}(y,\rho_{j_2})}\Phi_{(s,y)}\;\mathrm{d}\nu^j_{t}\leq\gamma,
\\
1\leq\inf_{x\in\mathbf{B}(y,\rho_{j_2})}\varphi_{(s,y),\rho_j}(t,x) + \gamma\rho_{j_2}^{-\mathbf{n}}
\end{align*}
for all $j\geq j_3$.
Combining these estimates with \eqref{localgaussianlem32} 
we obtain
\begin{align*}
\int_{\mathbb{R}^{\mathbf{n}+\mathbf{k}}}
\Phi_{(s,y)}\;\mathrm{d}\nu_{t}
\leq
\int_{\mathbb{R}^{\mathbf{n}+\mathbf{k}}}
\Phi_{(s,y)}\varphi_{(s,y),\rho_j}\;\mathrm{d}\nu^j_{t}
+ (2+C(s-t)^{-\mathbf{n}/2})\gamma
\end{align*}
for all $j\geq j_3$.
By \eqref{localgaussianlem24}
and as $s,t,y,\gamma$ were arbitrary 
this establishes \eqref{localgaussianlem51}.

In view of \eqref{localgaussianlem42} and \eqref{localgaussianlem51}
we can use Lemma \ref{globalgaussianlem} 
to obtain a subspace $T\in\mathbf{G}(\mathbf{n}+\mathbf{k},\mathbf{n})$
such that
\begin{align}
\label{localgaussianlem61}
\nu_t=\mathscr{H}^{\mathbf{n}}\mres T
\end{align}
for all $t\in (-2,0)$.
Note that by \eqref{localgaussianlem41} 
and as $0\in\mathrm{spt}\nu^j_0$ for all $j\in\mathbb{N}$
we see $a=0$ is a porpper choice in Lemma \ref{globalgaussianlem}.
This should now contradict that \eqref{localgaussianlem25} or \eqref{localgaussianlem26}
hold for infinitely many $j$.

First suppose  that for infinitely many $j$ inequality \eqref{localgaussianlem25} holds,
i.e.\ there exists a $z_j\in\mathrm{spt}\nu^j_{-\sigma}\cap\mathbf{B}(0,1)$
such that $(T^{\bot})_{\natural}(z_j)>\epsilon$.
Consider $C_1$ and $\alpha_1$ from the clearing out lemma (see Lemma \ref{clearoutlem}).
Choose $\tau,\eta_1\in (0,1)$
such that $4\mathbf{n}\tau<(\epsilon/4)^2$
and $C_1\eta_1^{2\alpha_1}(\epsilon/4)^2\leq\tau$.
Then Lemma \ref{clearoutlem} with $R=\epsilon/4$
yields that $\nu^j_{-\sigma-\tau}(\mathbf{B}(z_j,\epsilon/4))>\eta_1$
for infinitely many $j$.
A subsequence of the $z_j$ converges to some $z_0\in \mathbf{B}(0,2)$
with $(T^{\bot})_{\natural}(z_0)\geq\epsilon$.
Consider a cut-off function 
$\zeta_1\in\mathcal{C}^{\infty}_{\mathrm{c}}(\mathbf{B}(z_0,\epsilon/2),[0,1])$
with $\{\zeta_1=1\}\supset\mathbf{B}(z_0,\epsilon/3)$.
Then 
\begin{align*}
\nu^j_{-\sigma-\tau}(\zeta_1)
\geq\nu^j_{-\sigma-\tau}(\mathbf{B}(z_j,\epsilon/4))
>\eta_1>0=\nu_{-\sigma-\tau}(\zeta_1)
\end{align*}
for infinitely many $j$,
where we used \eqref{localgaussianlem61} for the last equality.
In view of \eqref{localgaussianlem41} this yields a contradiction.

Now suppose that for infinitely many $j$ inequality \eqref{localgaussianlem26} holds.
Consider 
$\zeta_2\in\mathcal{C}^{\infty}_{\mathrm{c}}(\mathbf{B}(0,\sqrt[\mathbf{n}]{1+\epsilon/2})),[0,1])$
with $\{\zeta_2=1\}\supset\mathbf{B}(0,1)$.
In view of \eqref{localgaussianlem61} we can estimate
\begin{align*}
\nu_{-\sigma}(\zeta_2)
\leq\omega_{\mathbf{n}}(1+\epsilon/2)
<\omega_{\mathbf{n}}(1+\epsilon)
\leq\nu^{j}_{-\sigma}(\mathbf{B}(0,1))
\leq\nu^{j}_{-\sigma}(\zeta_2)
\end{align*}
for infinitely many $j$.
Again, this contradicts \eqref{localgaussianlem41},
which establishes the result.
\end{proof}
%
%
%
%
\begin{proof}[Proof of Theorem {\ref{whitelocregthm}}]
We may assume $t_0=0$, $a=0$ and $\rho=1$.
Let $\alpha_0$ and $\gamma_0$ be from Theorem \ref{locregthm}
with respect to $\lambda=1/2$.
Choose $\epsilon\in (0,\gamma_0]$ 
such that $2\epsilon^{\alpha_0}\leq\alpha_0\beta$
and set $\sigma:=2\epsilon^{\alpha_0}$.
Let $\delta$ be chosen with respect to $\epsilon$ and $\sigma$ according to Lemma \ref{localgaussianlem}
and choose $\eta\leq\delta$, $t_1:=-2\epsilon^{\alpha_0}\delta^2$.
Then Lemma \ref{localgaussianlem} with yields the existence of a 
$T\in\mathbf{G}(\mathbf{n}+\mathbf{k},\mathbf{n})$ such that
\begin{align*}
\sup\left\{|(T^{\bot})_{\natural}(x)|, x\in\mathrm{spt}\mu_{t_1}\cap\mathbf{B}(0,2\delta)\right\}
\leq\epsilon\delta.
\\
\mu_{t_1}(\mathbf{B}(0,\delta))
\leq\omega_{\mathbf{n}}(1+\epsilon)\delta^{\mathbf{n}}.
\end{align*}
Then Theorem \ref{locregthm} with $\rho=\delta$ and $\gamma=\epsilon$
yields the desired graphical representation for $\eta=\epsilon^{\alpha_0}\delta$.
\end{proof}

\begin{appendix}

\section{Curvature estimates}
\label{curvature_estimates}
%
%
%
%
\begin{sett}
\label{normalmcfsett}
Consider open and connected sets $W\subset\mathbb{R}^{\mathbf{n}}$, $J\subset\mathbb{R}$.
Consider $\Psi\in\mathcal{C}^{\infty}(J\times W,\mathbb{R}^{\mathbf{n}+\mathbf{k}})$
such that $\Psi_t:=\Psi(t,\cdot\,)$ are empbeddings and set $M_t:=\Psi_t[W]$ for all $t\in J$.
Fix $t\in J$ and $\hat{y}\in W$.
We set
\begin{align*}
&b_i(t,\hat{y}):=D_i\Psi(t,\hat{y}),
\;\;\;
g_{ij}(t,\hat{y}):=b_i(t,\hat{y})\cdot b_j(t,\hat{y}),
\\
&\mathbf{T}(M_t,\Psi(t,\hat{y})):=\mathrm{span}(b_i(t,\hat{y}))_{1\leq i\leq\mathbf{n}}
\\
&A_{ij}(t,\hat{y})
:=\big(\mathbf{T}(M_t,\Psi(t,\hat{y}))^{\bot}\big)_{\natural}D_iD_j\Psi(t,\hat{y})
\in\mathbb{R}^{\mathbf{n}+\mathbf{k}}
\end{align*}
for $i,j\in\{1,\ldots,\mathbf{n}\}$.
In the following we sum over repeated indices.
Let $(g^{ij}(t,\hat{y}))$ be the inverse to $(g_{ij}(t,\hat{y}))$.
and let $(\mathrm{d}b^{i}(t,\hat{y}))_{1\leq i\leq\mathbf{n}}$ 
be the dual basis to $(b_i(t,\hat{y}))_{1\leq i\leq\mathbf{n}}$.
We set
\begin{align*}
\mathbf{A}(M_t,\Psi_t(\hat{y}))&:=A(t,\hat{y})
:=A_{ij}(t,\hat{y})\,\mathrm{d}b^{i}(t,\hat{y})\otimes\mathrm{d}b^{j}(t,\hat{y})
\\
\mathbf{H}(M_t,\Psi_t(\hat{y}))&
:=g^{ij}(t,\hat{y})A_{ij}(t,\hat{y}).
\end{align*}
Consider $T=T_{i_1\cdots i_p}\mathrm{d}b^{i_1}\otimes\cdots\otimes\mathrm{d}b^{i_p}$ 
and $S=S_{j_1\cdots j_p}\mathrm{d}b^{j_1}\otimes\cdots\otimes\mathrm{d}b^{j_p}$
with $T_{i_1\cdots i_p},S_{j_1\cdots j_p}\in\mathbb{R}^{\mathbf{n}+\mathbf{k}}$.
We set
\begin{align*}
&\langle T,S\rangle :=g^{i_1j_1}\cdots g^{i_pj_p}T_{i_1\cdots i_p}\cdot S_{j_1\cdots j_p},
\quad
|T|^2:=\langle T,T\rangle.
\end{align*}

Moreover we assume $\Psi$ moves by smooth mean curvature flow in normal direction, i.e.\
\begin{align}
\label{normalmcf}
\partial_t\Psi(t,\hat{y})=\mathbf{H}(M_t,\Psi_t(\hat{y}))
\end{align}
for all $(t,\hat{y})\in J\times W$.
\end{sett}

%
%
%
%
\begin{lem}
\label{tangentialdifflem}
Consider open and connected sets $W_0\subset\subset W\subset\subset\Omega\subset\mathbb{R}^{\mathbf{n}}$, $I\subset\mathbb{R}$.
Let $F\in\mathcal{C}^{\infty}(I\times\Omega,\mathbb{R}^{\mathbf{n}+\mathbf{k}})$
be such that $F_t:=F(t,\cdot\,)$ are embeddings and set $M_t:=F_t[\Omega]$ for all $t\in I$.
Suppose $F$ satisfies \eqref{smoothmcf} and fix $s_0\in I$.

Then there exists an open $J\subset I$ with $s_0\in J$
and a $\phi\in\mathcal{C}^{\infty}(J\times W,\mathbb{R}^{\mathbf{n}})$
such that $\phi_t:=\phi(t,\cdot\,)$ is injective and $W_0\subset\phi_t[W]$ for all $t\in J$.
Fuhthermore $\phi$ satisfies $\phi_{s_0}(\hat{y})=\hat{y}$ for all $\hat{y}\in W$ and
\begin{align}
\label{tangentialdifflema}
DF(t,\phi(t,\hat{y}))\partial_t\phi(t,\hat{y})
=-(\mathbf{T}(M_t,F(t,\phi(t,\hat{y}))))_{\natural}\partial_tF(t,\phi(t,\hat{y}))
\end{align}
for all $(t,\hat{y})\in J\times W$.
In particular $\Psi(t,\hat{y}):=F(t,\phi(t,\hat{y}))$ satisfies \eqref{normalmcf}
and we are in Setting \ref{normalmcfsett}.
\end{lem}
\begin{proof}
Actually \eqref{tangentialdifflema} equals
\begin{align}
\label{tangentialdifflem11}
\partial_t\phi(t,\hat{y})
=G(t,\phi(t,\hat{y}))
\end{align}
for some $G\in \mathcal{C}^{\infty}(I\times\Omega,\mathbb{R}^{\mathbf{n}})$
depending on $F$.
The function $G$ can be obtained in the following way:
Let $p_i:=D_iF$, $B_{ij}:=p_i\cdot p_j$ and $(B^{ij})$ be the inverse of $(B_{ij})$.
Set $q_i:=B^{ij}p_j$. Now we obtain \eqref{tangentialdifflem11} from
\begin{align*}
Q:=\left(q_1\cdots q_{\mathbf{n}}\right),
\quad
DF=\left(p_1\cdots p_{\mathbf{n}}\right),
\quad
Q^{T}v=(v\cdot p_i)\hat{\mathbf{e}}_i,
\quad
G:=-Q^{T}\partial_tF,
\end{align*}
for all $v\in \mathbb{R}^{\mathbf{n}+\mathbf{k}}$.
Then the Picard-Lindel\"off-theorem yields the result.
\end{proof}

%
%
Next we need a bound on the evolution of the absolute value of the curvature tensor and its covariant derivatives.
In the case of one co-dimension this was done by Huisken \cite[Thm.\ 3.4]{huisken1} (see also \cite[Sect.\ 13]{hamilton}).
Here we state the generalization by Andrews and Baker.
%
%
\begin{prop}[{\cite[Prop.\ 9]{andrewsb}}]
\label{curvevoprop}
For every $m\in\mathbb{N}\cup\{0\}$ there exists an $C_m\in [1,\infty)$
such that in Setting \ref{normalmcfsett} the following holds:
\begin{align*}
\begin{split}
\partial_t\left|\nabla^{m} A\right|^2
\leq\Delta\left| A\right|^2-2\left|\nabla^{m+1} A\right|^2
&+C_m\sum_{a+b+c=m}|\nabla^m A||\nabla^a A||\nabla^b A||\nabla^{c}A|
\end{split}
\end{align*}
everywhere in $J\times W$.
Here $\nabla$ and $\Delta$ are the connection and Laplace-Beltrami operator 
associated to $(M_t,(g_{ij}))$.
\end{prop}

%
%
%
%
Now we can prove a higher regularity theorem similar to Ecker's \cite[Prop.\ 3.22]{eckerb} (see also \cite{eckerh1}).
%
%
\begin{prop}
\label{higherregprop}
For all $m\in\mathbb{N}$ and $K\in [1,\infty)$
there exists a $C_{m,K} \in (1,\infty)$
such that the following holds:

Consider an open and connected set $\Omega\subset\mathbb{R}^{\mathbf{n}}$.
Let $\rho\in (0,\infty)$, $t_0\in\mathbb{R}$, $a\in\mathbb{R}^{\mathbf{n}+\mathbf{k}}$
and let $F\in\mathcal{C}^{2}((t_0-4\rho^2,t_0)\times\Omega,\mathbb{R}^{\mathbf{n}+\mathbf{k}})$
be such that for all $t\in (t_0-4\rho^2,t_2)$ we have
$F_t:=F(t,\cdot\,)$ are embeddings, $M_t:=F_t[\Omega]$ and $\partial M_t\cap\mathbf{B}(a,2\rho)=\emptyset$.
Suppose $F$ satisfies \eqref{smoothmcf} and
\begin{align*}
\kappa^2:=\rho^2\sup_{t\in (t_0-4\rho^2,t_0)}\sup_{x\in M_t\cap\mathbf{B}(a,2\rho)}|\mathbf{A}(M_t,x)|^2\leq K
\end{align*}

Then we have
\begin{align*}
\rho^{2m+2}\sup_{t\in [t_0-\rho^2,t_0)}\sup_{x\in M_t\cap\mathbf{B}(a,\rho)}|\nabla^m\mathbf{A}(M_t,x)|^2
\leq C_{m,K}\kappa^2
\end{align*}
\end{prop}
%
%
\begin{proof}
Obviously the statement holds for $m=0$.
Now suppose the statement holds for all $l=0,\ldots, m$
for some $m\in\mathbb{N}\cup\{0\}$
and show this implies it also holds for $m+1$.
We may assume $t_0=0$, $a=0$ and $\rho=3$.
Fix $t_2\in (-9,0)$ and let $\Lambda\in [1,\infty)$ be chosen below,.
Consider the family of functions
$\alpha_t:M_t\to\mathbb{R}^+$, $t\in (-12,0)$ given by
\begin{align*}
\alpha_t(x)&:=|\nabla^{m+1}\mathbf{A}(M_t,x)|^2(\Lambda\kappa^2+|\nabla^m\mathbf{A}(M_t,x)|^2).
\end{align*}
Also consider the cut-off functions $\eta_t\in\mathcal{C}^2_{\mathrm{c}}(\mathbf{B}(0,5),\mathbb{R}^+)$, $t\in (-12,0)$ given by
\begin{align*}
\eta_t(x)&:=\left((t+10)_+(16-|x|^2)_+\right)^3.
\end{align*}
For $t\in (-12,0)$ set
\begin{align*}
m(t):=\sup_{M_t}\alpha_t\eta_t,
\quad
m_0:=\sup_{t\in [-10,t_2]}m(t),
\quad
s_0:=\inf\{t\in[t_1,t_2],\;m(t)=m_0\}.
\end{align*}
There exists an $x_0\in M_{s_0}$ such that $\alpha_{s_0}(x_0)\eta_{s_0}(x_0)=m_0$
and $\hat{y}_0\in\Omega$ with $F_{s_0}(\hat{y}_0)=x_0$.
Consider $s_3=\min\{s_0+1/2,s_0/2\}$ in particular $s_0\in (s_3-1,s_3)$.
By induction assumption we can use the Proposition for $m$
with $a$, $t_0$, $\rho$ replaced by $x_0$, $s_3$, $1$
to obtain
\begin{align}
\label{higherregprop11}
\max_{l=1,\ldots,m}|\nabla^{l}\mathbf{A}(M_{s_0},x_0)|^2
\leq
\Gamma_m\kappa^2
\end{align}
for some $\Gamma_m\in (K,\infty)$ depending only on $m$ and $K$.
In the following quantities that only depend on 
$m$, $K$ and $\Gamma_m$ will be denoted by $C_m$.

In view of Lemma \ref{tangentialdifflem}
there exists an $\epsilon\in (0,-s_0)$, $J=(s_0-\epsilon,s_0+\epsilon)$, an open $W\subset\Omega$
and a $\phi\in\mathcal{C}^{\infty}\left(J\times W,\Omega\right)$,
such that $\phi_t=\phi(t,\cdot\,)$ is a diffeomorphism onto its image and
$F^{-1}[M_t\cap\mathbf{B}(x_0,\epsilon)]\subset\phi_t[W]$ for all $t\in J$.
Moreover choose $\phi$ such that $\Psi_t(\hat{y}):=F_t(\phi_t(\hat{y}))$ satisfies \eqref{normalmcf}.
In particular $\hat{y}_0\in W$ and we are in Setting \ref{normalmcfsett}.
Consider $\hat{\alpha}\in\mathcal{C}^{\infty}\left(J\times W,\mathbb{R}\right)$ given by
\begin{align*}
\hat{\alpha}(t,\hat{x}):=\hat{\alpha}_t(\hat{x})
:=\alpha_t(\Psi_t(\hat{x}))
=|\nabla^{m+1}A(t,\hat{x})|^2(\Lambda\kappa^2+|\nabla^m A(t,\hat{x})|^2)
\end{align*}
in the $\Psi$ coordinates.
Using Proposition \ref{curvevoprop} and inequality \eqref{higherregprop11}
we can estimate at $(s_0,\hat{y}_0)$
\begin{align*}
(\partial_t-\Delta)\hat{\alpha}
= &
(\Lambda\kappa^2+|\nabla^m A|^2)(\partial_t-\Delta)|\nabla^{m+1}A|^2
+|\nabla^{m+1}A|^2(\partial_t-\Delta)|\nabla^m A|^2
\\&
-2\nabla|\nabla^{m+1}A|^2\cdot\nabla|\nabla^{m}A|^2
\\
\leq &
\Lambda\kappa^2\left(-2|\nabla^{m+2}A|^2
+C_m|\nabla^{m+1}A|(|\nabla^{m+1}A|\kappa^2+\kappa^3)\right)
\\ &
-2|\nabla^{m+1}A|^4
+C_m|\nabla^{m+1}A|^2\kappa^4
-2\nabla|\nabla^{m+1}A|^2\cdot\nabla|\nabla^{m}A|^2.
\end{align*}
By Young's and Kato's inequality we can estimate
\begin{align*}
C_m\Lambda|\nabla^{m+1}A|\kappa^5
&\leq C_m\Lambda|\nabla^{m+1}A|^2\kappa^4+\kappa^6,
\\
C_m\Lambda|\nabla^{m+1}A|^2\kappa^4
&\leq|\nabla^{m+1}A|^4/2+C_m\Lambda^2\kappa^8,
\\
2\left|\nabla|\nabla^{m+1}A|^2\cdot\nabla|\nabla^{m}A|^2\right|
&\leq 8|\nabla^{m}A||\nabla^{m+1}A|^2|\nabla^{m+2}A|
\\&\leq
\Lambda\kappa^2|\nabla^{m+2}A|^2+2\Lambda^{-1}|\nabla^{m+1}A|^4.
\end{align*}
Fixing $\Lambda=4$ this leads to
\begin{align*}
(\partial_t&-\Delta)\hat{\alpha}
\leq-|\nabla^{m+1}A|^4+C_m\kappa^4
\leq -L^{-1}\kappa^{-4}\hat{\alpha}^2+L\kappa^4
\end{align*}
at $(s_0,\hat{y}_0)$ for some $L\in [1,\infty)$ depending only on $m$, $K$ and $\Gamma_m$.
Consider $\hat{\eta}(t,\hat{x}):=\eta_t(\Psi(t,\hat{x}))$
then at $(s_0,\hat{y}_0)$ we can estimate
\begin{align*}
(\partial_t-\Delta)(\hat{\alpha}\hat{\eta})
&\leq\hat{\eta}(L\kappa^4-L^{-1}\kappa^{-4}\hat{\alpha}^2)+C\hat{\alpha}-2\nabla\hat{\alpha}\cdot\nabla\hat{\eta}
\\
&\leq\hat{\eta}(L\kappa^4-L^{-1}\kappa^{-4}\hat{\alpha}^2)+C\hat{\alpha}-2\hat{\eta}^{-1}\nabla\hat{\eta}\cdot\nabla(\hat{\alpha}\hat{\eta}).
\end{align*}
By choice of $s_0$ and $\hat{y}_0$ we have 
$\Delta(\hat{\alpha}\hat{\eta})\leq 0$, $\nabla(\hat{\alpha}\hat{\eta})=0$ 
and $\partial_t(\hat{\alpha}\hat{\eta})\geq 0$.
Hence
\begin{align*}
(L\kappa^4)^{-1}\hat{\alpha}^2\hat{\eta}
\leq L\kappa^4\hat{\eta}+C\hat{\alpha}
\end{align*}
at $(s_0,\hat{y}_0)$.
Multiplying with $\hat{\eta}$ and using Young's inequality we obtain
\begin{align*}
(L\kappa^4)^{-1}(\hat{\alpha}\hat{\eta})^2
\leq L\kappa^4\hat{\eta}^2+C\hat{\alpha}\hat{\eta}.
\leq C_m\kappa^4+(2L\kappa^4)^{-1}(\hat{\alpha}\hat{\eta})^2
\end{align*}
at $(s_0,\hat{y}_0)$.
Thus $m_0=\hat{\alpha}(s_0,\hat{y}_0)\hat{\eta}(s_0,\hat{y}_0)\leq C_m\kappa^4$ 
and we arrive at
\begin{align*}
|\nabla^{m+1}\mathbf{A}(M_t,x)|^2(\Lambda\kappa^2+|\nabla^m\mathbf{A}(M_t,x)|^2)\eta_t(x)
\leq C_m\kappa^4
\end{align*}
for all $t\in [-10,t_2]$ and all $x\in M_t$.
For $t\in [-9,t_2]$ and $x\in M_t\cap\mathbf{B}(0,3)$ we have $\eta_t(x)\geq 1$,
so we can conclude the desired estimate on $[-9,t_2]$.
As $t_2\in [-9,0)$ was arbitrary this establishes the statement for $m+1$.
\end{proof}

%
%
%
%
The above curvature estimates will be used in the form of the next Lemma.
It basiciely says that if you have a smooth mean curvature flow that satifies a strong curvature bound up to some time $t_1$
and the curvature on $[t_1,t_2]$ is a-priori bounded by some constant, then the increase in curvature is controlled by $(t_2-t_1)^p$.
Thus for short times you can actually almost maintain the strong curvature bound.
\begin{lem}
\label{curveboundcor}
For $p\in\mathbb{N}$ and $L\in [1,\infty)$
there exists a $C_{p,L} \in (1,\infty)$
such that the following holds:

Consider an open and connected set $\Omega$.
Let $\varrho\in (0,\infty)$, $z_0\in\mathbb{R}^{\mathbf{n}+\mathbf{k}}$,
$t_2\in\mathbb{R}$, $t_1\in [t_2 -\varrho^2/(4pL),t_2)$
and let $F\in\mathcal{C}^{2}((t_2-\varrho^2,t_2)\times\Omega,\mathbb{R}^{\mathbf{n}+\mathbf{k}})$
be such that for all $t\in (t_2-\varrho^2,t_2)$ we have
$F_t:=F(t,\cdot\,)$ are embeddings, $M_t:=F_t[\Omega]$ and $\partial M_t\cap\mathbf{B}(z_0,2\varrho)=\emptyset$.
Suppose $F$ satisfies \eqref{smoothmcf}.
For $s\in  [t_1,t_2]$, $r\in (0,2\varrho]$ set
\begin{align*}
\kappa(s,r)^2:=&\varrho^2\sup_{t\in (t_2-\varrho^2,s)}\sup_{x\in M_t\cap\mathbf{B}(z_0,r)}|\mathbf{A}(M_t,x)|^2
\end{align*}
and assume $\kappa(t_2,2\varrho)\leq L$.

Then we have
\begin{align*}
\kappa(t_2,\varrho)^2\leq C_{p,L}\big(\kappa(t_1,2\varrho)^2+\varrho^{-2p}(t_2-t_1)^p\big).
\end{align*}
\end{lem}
%
%
\begin{proof}
We may assume $z_0=0$, $t_1=0$ and $\varrho=1$.
Consider $q=0,1\ldots,p$ and set $\varrho_q:=2-q/p$.
We will actually show
\begin{align}
\label{curveboundcor11}
\kappa(t_2,\varrho_q)^2\leq\Gamma_q(2\kappa(0,2)^2+t_2^q).
\end{align}
for some $\Gamma_q\in [1,\infty)$ depending only on $L$, $p$ and $q$.
For $q=0$ this holds with $\Gamma_0=L$.
Now suppose \eqref{curveboundcor11} holds for some $q\in\{0,1\ldots,p-1\}$.
For $t\in [0,t_2]$ set $\varrho_q(t):=\varrho_q-1/(2p)-tL$.
We want to show
\begin{align}
\label{curveboundcor12}
\kappa(t,\varrho_q(t))^2\leq\Gamma_{q+1}((t+1)\kappa(0,2)^2+t\; t_2^q).
\end{align}
for some $\Gamma_{q+1}\in [1,\infty)$ depending only on $L$, $p$, $q$ and $\Gamma_{q}$.
Let $I_0$ be the interval containing all $s\in [0,t_2]$ such that  
\eqref{curveboundcor12} holds.
Note that by $t_2\leq 1/(4pL)$ we have $\varrho_q(t_2)\geq\varrho_{q+1}$,
thus if $t_2\in I_0$ we can conclude the result by induction.
Obviously $0\in I_0$.
Also, for an increasing sequence $(s_m)$ in $I_0$ with $s_m\nearrow s_0$ 
the smoothness of $(M_t)$ yields that $s_0\in I_0$.
To establish $I_0=[0,t_2]$ (and thus the result) it remains to show that for each $s_0\in I_0$
there exists an $\epsilon\in (0,t_2-s_0)$ such that $(s_0,s_0+\epsilon)\subset I_0$.
Hence, consider $s_0\in I_0$.

In view of Lemma \ref{tangentialdifflem}
there exists an $\epsilon\in (0,t_2-s_0)$, $J=(s_0-\epsilon,s_0+\epsilon)$, an open $W\subset\Omega$
and $\phi\in\mathcal{C}^{\infty}\left(J\times W,\Omega\right)$,
such that $\phi_t=\phi(t,\cdot\,)$ is a diffeomorphism onto its image and
$F^{-1}[M_t\cap\mathbf{B}(0,2)]\subset\phi_t[W]$ for all $t\in J$.
Moreover choose $\phi$ such that $\Psi_t(\hat{y}):=F_t(\phi_t(\hat{y}))$ satisfies \eqref{normalmcf},
in particular we are in Setting \ref{normalmcfsett}.

Consider an $s_2\in (s_0,s_0+\epsilon)$ and $x_0\in M_{s_2}\cap\mathbf{B}(0,\rho_q(s_2))$,
hence there exists an $\hat{y}_0\in W$ such that $x_0=\Psi_{s_2}(\hat{y}_0)$.
By \eqref{normalmcf} and $\kappa(t_2,2)\leq L$ we have
\begin{align*}
|\Psi_{s}(\hat{y}_0)|
\leq|x_0|+\int_{s_0}^{s}|D_t\Psi_{t}(\hat{y}_0)|\,\mathrm{d}t
<\rho_q(s_2)+(s_2-s_0)L
\leq\rho_q(0)
\end{align*}
for all $s\in [s_0,s_2]$.
By Proposition \ref{curvevoprop} we have
\begin{align}
\label{curveboundcor21}
\left|\frac{\mathrm{d}}{\mathrm{d}s}|\mathbf{A}(M_{s},\Psi_{s}(\hat{y}_0))|^2\right|
\leq C_1\sum_{i=0}^2\sup_{t\in [0,s_2]}\sup_{x\in M_t\cap\mathbf{B}(0,\rho_q(0))}|\nabla^i\mathbf{A}(M_{t},x)|^2
\end{align}
for $s\in (s_0,s_2)$ and some constant $C_1\in [1,\infty)$.
We assume \eqref{curveboundcor11} holds for $q$,
so by $\rho_q(0)\leq\rho_q-1/(2p)$ and Proposition \ref{higherregprop} with $\rho=1/(4p)$, $K=\Gamma_{q}(2L^2+1)$ 
and arbitrary $(t_0,a)\in [0,s_2]\times\mathbf{B}(0,\rho_q(0))$ we have
\begin{align}
\label{curveboundcor22}
\sum_{i=0}^2\sup_{t\in [0,s_2]}\sup_{x\in M_t\cap\mathbf{B}(0,\rho_q(0))}|\nabla^i\mathbf{A}(M_{t},x)|^2
\leq\Lambda_q(\kappa(0,2)^2+t_2^q).
\end{align}
for some $\Lambda_q\in [1,\infty)$ 
depending only on $L$, $p$, $q$ and $\Gamma_q$.

As $s_0\in I_0$ and $|\Psi_{s_0}(\hat{y}_0)|<\rho_q(0)\leq\rho_q(s_0)$ inequality \eqref{curveboundcor12}
combined with \eqref{curveboundcor21} and \eqref{curveboundcor22} yield
\begin{align*}
&|\mathbf{A}(M_{s_2},x_0)|^2
\leq
|\mathbf{A}(M_{s_0},\Psi_{s_0}(\hat{y}_0))|^2+(s_2-s_0)\sup_{t\in [s_0,s]}\left|\frac{\mathrm{d}}{\mathrm{d}t}|\mathbf{A}(M_{t},\Psi_{t}(\hat{y}_0))|^2\right|
\\&\leq
\Gamma_{q+1}((s_0+1)\kappa(0,2)^2+s_0\;t_2^q)+C_1\Lambda_q(s_2-s_0)(\kappa(0,2)^2+t_2^q)
\\&\leq
\Gamma_{q+1}((s_2+1)\kappa(0,2)^2+s_2\;t_2^q),
\end{align*}
where we chose $\Gamma_{q+1}=C_1\Lambda_q$.
As $s_2\in (s_0,s_0+\epsilon)$ and $x_0\in M_{s_2}\cap\mathbf{B}(0,\rho_q(s_2))$
were arbitrary we proved $(s_0,s_0+\epsilon)\subset I_0$.
\end{proof}

%
%
%
%
For a solution of mean curvature flow, that is graphical with small gradient we obtain an a-priori curvature bound.
This follows directly from the work of Wang \cite{wang} and Ecker and Huisken \cite{eckerh1}.
%
%
\begin{lem}
\label{wangcurveestlem}
There exist $C \in (1,\infty)$ and $l_1\in (0,1)$ 
with the following property:

Let $\rho\in (0,\infty)$, $t_1\in\mathbb{R}$, $t_2\in (t_1,t_1+l_1\rho^2)$,
$w\in\mathcal{C}^{\infty}((t_1,t_2)\times\mathbf{B}^{\mathbf{n}}(0,2\rho),\mathbb{R}^{\mathbf{k}})$.
Suppose $F_t(\hat{x}):=(t,w(t,\hat{x}))$ satisfies \eqref{smoothmcf}
and $\sup|Dw|\leq l_1$.

Then we have
\begin{align*}
\sup_{\hat{x}\in\mathbf{B}^{\mathbf{n}}(0,\rho)}|\mathbf{A}(M_t,(\hat{x},w(t,\hat{x}))|^2
&\leq C (t-t_1)^{-1}
\end{align*}
for all $t\in (t_1,t_2)$,
where $M_t=\mathrm{graph}(w(t,\cdot\,))$.
\end{lem}
%
%
\begin{proof}
Consider $[s_1,s_2]\subset (t_1,t_2)$.
We may assume $s_1=0$ and $\rho=1$.
For $t\in [0,s_2]$ set $M_t=\mathrm{graph}(w(t,\cdot\,))$
and 
$\mathcal{M}:=\{(t,x)\in [0,s_2]\times\mathbb{R}^{\mathbf{n}+\mathbf{k}}:\;x\in M_t\}$.
On $\mathcal{M}$ define
\begin{align*}
*\Omega(t,x):=\left(\mathrm{det}(\mathbf{Id}_{\mathbb{R}^{\mathbf{n}}}+(Dw(t,\hat{x})^TDw(t,\hat{x})\right)^{-1}\in (0,1],
\end{align*}
where $\mathbf{Id}_{\mathbb{R}^{\mathbf{n}}}$ is the identity on $\mathbb{R}^{\mathbf{n}}$.
The gradient bound on $w$ yields
\begin{align}
\label{wangcurveestlem11}
\inf_{\mathcal{M}}*\Omega\geq\left(1+Cl_1\right)^{-1}\geq 19/20,
\end{align}
where we chose $l_1$ small enough.
For $(t,x)\in\mathcal{M}$ consider 
\begin{align*}
f(t,x):=*\Omega(t,x)-9/10,
\quad
v(t,x):=1/f(t,x),
\\
\kappa:=\frac{1}{2}\inf_{t\in [0,s_2]}\inf_{x\in M_t\cap\mathbf{C}(0,\sqrt{2})}(v(t,x))^{-2},
\quad
\varphi(t,x):=(v(t,x))^2/(1-\kappa (v(t,x))^2),
\\
g(t,x):=|\mathbf{A}(M_t,x)|^2\varphi(t,x)
\quad\text{and}\quad
\eta(x):=(2-|\hat{x}|^2)^2.
\end{align*}
For $t\in [0,s_2]$ set $g_t:=g(t,\cdot\,)$ and
\begin{align*}
m(t):=\sup_{M_t}tg_t\eta,
\quad
m_0:=\sup_{t\in [0,s_2]}m(t),
\quad
s_0:=\inf\{t\in [0,s_2],\;m(t)=m_0\}.
\end{align*}
Consider $x_0\in M_{s_0}\cap\mathbf{C}(0,\sqrt{2})$
such that $m_0=s_0g(s_0,x_0)\eta(x_0)$.

In view of Lemma \ref{tangentialdifflem}
there exists an $\epsilon>0$, $J=(s_0-\epsilon,s_0+\epsilon)\subset (t_1,t_2)$, $W:=\mathbf{B}^{\mathbf{n}}(0,7/4)$
and $\phi\in\mathcal{C}^{\infty}\left(J\times W,\Omega\right)$,
such that $\phi_t=\phi(t,\cdot\,)$ is a diffeomorphism onto its image and
$\mathbf{B}^{\mathbf{n}}(0,\sqrt{2})\subset\phi_t[W]$ for all $t\in J$.
Moreover choose $\phi$ such that $\Psi_t(\hat{y}):=F_t(\phi_t(\hat{y}))$ satisfies \eqref{normalmcf},
in particular we are in Setting \ref{normalmcfsett}.
Let $\hat{y}_0\in W$ be such that $\Psi_{s_0}(\hat{y}_0)=x_0$.

By \eqref{wangcurveestlem11} and \cite[Lem.\ 3.1]{wang} we can estimate
\begin{align}
\label{wangcurveestlem21}
\left(\partial_t-\Delta\right)\hat{f}
=\left(\partial_t-\Delta\right)*\Omega
\geq|A|^2/2
\geq 5\hat{f}|A|^2
\end{align}
everywhere in $J\times W$, where $\hat{f}(t,\hat{x}):=f(\Psi_t(\hat{x}))$.
Then by the proof of \cite[Cor.\ 4.1]{wang} we obtain
\begin{align}
\label{wangcurveestlem22}
\left(\partial_t-\Delta\right)|A|
\leq 5|A|^3
\end{align}
everywhere in $J\times W$.
Estimates \eqref{wangcurveestlem21}, \eqref{wangcurveestlem22} and
the proof of \cite[Lem.\ 4.1]{wang} with $h=A$, $c_1=c_2=5$ 
and $\phi$ replaced by $\varphi$ yield
\begin{align*}
\left(\partial_t-\Delta\right)\hat{g}
\leq 
-10\kappa\hat{g}^2-2\kappa(1-\kappa\hat{v}^2)^{-2}|\nabla\hat{v}|^2\hat{g}
-2\hat{\varphi}\hat{v}^{-3}\nabla\hat{v}\cdot\nabla\hat{g}
\end{align*}
everywhere in $J\times W$, 
where $\hat{g}(t,\hat{x}):=g(t,\Psi_t(\hat{x}))$, $\hat{v}(t,\hat{x}):=v(t,\Psi_t(\hat{x}))$ 
and $\hat{\varphi}(t,\hat{x}):=\varphi(t,\Psi_t(\hat{x}))$.
Then following the proof of \cite[Thm.\ 3.1]{eckerh1} 
with $r(t,x)=|x|^2-|w(t,\hat{x})|$
we see
\begin{align*}
\left(\partial_t-\Delta\right)\iota\hat{g}\hat{\eta}
\leq & 
-10\kappa\hat{g}^2\hat{\eta}\iota
-2\left(\hat{\varphi}\hat{v}^{-3}\nabla\hat{v}+\hat{\eta}^{-1}\nabla\hat{\eta}\right)
\cdot\nabla(\iota\hat{g}\hat{\eta})
\\
&+C\left((1+(\kappa\hat{v}^2)^{-1})\right)\iota\hat{g}
+\hat{g}\hat{\eta}
\end{align*}
everywhere in $J\times W$, where $\hat{\eta}(t,\hat{x}):=\eta(\Psi_t(\hat{x}))$ and $\iota(t,\hat{x}):=t$.
In particular at $(s_0,x_0)$ we have
\begin{align*}
10\kappa g^2 \eta s_0
\leq
C\left((1+(\kappa v^2)^{-1})\right)s_0g
+g\eta.
\end{align*}
Recall the definition of $\kappa$ and that by \eqref{wangcurveestlem11}
we have $\kappa\geq c$. Also note that $\eta\leq 2$ and $s_0\leq s_2\leq l_1\leq 1$.
Thus we obtain
\begin{align*}
s_2\sup_{M_{s_2}\cap\mathbf{C}(0,\sqrt{2})}g_{s_2}\eta
\leq s_0g(s_0,x_0)\eta(x_0)
\leq C.
\end{align*}
As $\inf_{\mathbf{C}(0,1)}\eta\geq 1$ and $\inf_{\mathcal{M}}\varphi>1$ we arrive at
\begin{align*}
\sup_{\hat{x}\in\mathbf{B}^{\mathbf{n}}(0,1)}|\mathbf{A}(M_{s_2},(\hat{x},w(s_2,\hat{x}))|^2
&\leq C (s_2-s_1)^{-1}
\end{align*}
for arbitrary $[s_1,s_2]\subset (t_1,t_2)$.
Now for any $t\in (t_1,t_2)$ choose $s_2=t$ and $s_1=(t+t_1)/2$
to conclude the result. 
\end{proof}

\section{Further remarks}
\label{appendix}
%
%
%
%
%
Based on Huisken's monotonicity formula \cite{huisken2}
one can obtain bounds on area ratio at later times
from initial area ratio bounds.
%
%
\begin{lem}
\label{laterbndlem}
There exists a $C\in (1,\infty)$
such that the following holds:

Let $r\in (0,\infty)$, $s_1\in\mathbb{R}$, $s_2\in (s_1,\infty)$,
$y_0\in\mathbb{R}^{\mathbf{n}+\mathbf{k}}$
and for $t\in [s_1,s_2]$ set $\varrho(t):=\sqrt{8\mathbf{n}}\max\{r,\sqrt{t-s_1}\}$.
Let $(\mu_t)_{t\in [s_1,s_2]}$ be a Brakke flow in $\mathbf{B}(y_0,\varrho(s_2))$.

Then for all $t\in [s_1,s_2]$ we have
\begin{align*}
r^{-\mathbf{n}}\mu_{t}(\mathbf{B}(y_0,r))
\leq C(\varrho(t))^{-\mathbf{n}}\mu_{s_1}(\mathbf{B}(y_0,\varrho(t))).
\end{align*}
\end{lem}
%
%
\begin{proof}
We may assume $y_0=0$, $s_1=0$.
Fix $t\in [0,s_2]$.
First consider the case $t\leq r^2$.
Then by Lemma \ref{barrierlem} with $R=\sqrt{2\mathbf{n}}r$ we have
\begin{align*}
\mu_{t}(\mathbf{B}(y_0,r))
\leq\mu_{t}(\mathbf{B}(y_0,\sqrt{2\mathbf{n}}r))
\leq 8\mu_{0}(\mathbf{B}(y_0,2\sqrt{2\mathbf{n}}r))
\end{align*}
and as $\varrho(t)=\sqrt{8\mathbf{n}}r$ this yields the desired estimate.

Now suppose $r^2<t$,
in particular $\varrho(t)=\sqrt{8\mathbf{n}t}$.
Set $s_0:=t+r^2$.
Consider $\Phi_{}=\Phi_{(s_0,0)}$
and $\varphi_{}=\varphi_{(s_0,0),\sqrt{4\mathbf{n}t}}$
from Definition \ref{sphtestfct}.
By $s_0< 2t$
we obtain
\begin{align*}
\mathrm{spt}\varphi_{}(0,\cdot\,)\subset\subset\mathbf{B}(0,\varrho(t)),
\;\;\;
\sup_{\mathbb{R}^{\mathbf{n}+\mathbf{k}}}\varphi_{}(0,\cdot\,)\leq C,
\;\;\;
\inf_{\mathbf{B}(0,r)}\varphi_{}(t,\cdot\,)\geq 1.
\end{align*}
Thus by Huiskin's monotonicity formula (see Theorem \ref{localmonthm})
\begin{align*}
\int_{\mathbf{B}(0,r)}\Phi_{}\,\mathrm{d}\mu_{t}
\leq 
\int_{\mathbb{R}^{\mathbf{n}+\mathbf{k}}}\Phi_{}\varphi_{}\,\mathrm{d}\mu_{t}
\leq 
\int_{\mathbb{R}^{\mathbf{n}+\mathbf{k}}}\Phi_{}\varphi_{}\,\mathrm{d}\mu_{0}
\leq 
C\int_{\mathbf{B}(0,\varrho(t))}\Phi_{}\,\mathrm{d}\mu_0.
\end{align*}
We have $s_0-t=r^2$,
hence
$\inf_{\mathbf{B}(0,r)}\Phi_{}(t,\cdot\,)\geq cr^{-\mathbf{n}}$.
So by the above inequality
\begin{align}
\label{laterbndlem21}
r^{-\mathbf{n}}\mu_{t}(\mathbf{B}(0,r))
\leq C\int_{\mathbf{B}(0,\varrho(t))}\Phi_{}(0,x)\,\mathrm{d}\mu_0(x).
\end{align}
By definition of $\Phi$ 
and $s_0=t+r^2\geq t=(8\mathbf{n})^{-1}\varrho(t)^2$ we have
$\Phi_{}(0,x)\leq C(\varrho(t))^{-\mathbf{n}}$ for all $x\in\mathbb{R}^{\mathbf{n}+\mathbf{k}}$.
Thus \eqref{laterbndlem21}
establishes the result.
\end{proof}
%
%
%
%
Next we show that under the assumption of certain measure bounds
if \eqref{whitelocregthma} is satisfied at one point
a slightly weaker estimate is satisifed in a small neighbourhood.
%
%
\begin{lem}
\label{ptngaussdenslem}
There exists a $C\in (1,\infty)$
such that the following holds:

Let $\varrho,M,\kappa\in (0,\infty)$,
$\Lambda\in [1,\infty)$,
$\delta\in (0,1/(C\Lambda)]$,
$s_0\in\mathbb{R}$,
$y_0\in\mathbb{R}^{\mathbf{n}+\mathbf{k}}$
and let $\mu$ be a Radon measure on $\mathbb{R}^{\mathbf{n}+\mathbf{k}}$.
Suppose we have
\begin{align}
\label{ptngaussdenslema}
\mu\left(\mathbf{B}(y_0,C\Lambda\varrho)\right)
&\leq M\varrho^{\mathbf{n}},
\\
\label{ptngaussdenslemb}
\int_{\mathbb{R}^{\mathbf{n}+\mathbf{k}}}\Phi_{(s_0,y_0)}\varphi_{(s_0,y_0),\Lambda\varrho}(s_0-\varrho^2,x)\,\mathrm{d}\mu(x)
&\leq 1+\kappa.
\end{align}

Then
for all $(s,y)\in(s_0-\delta^2\varrho^2,s_0]\times\mathbf{B}(y_0,\delta\varrho)$
we have
\begin{align*}
\int_{\mathbb{R}^{\mathbf{n}+\mathbf{k}}}\Phi_{(s,y)}\varphi_{(s,y),\Lambda\varrho}
(s_0-\varrho^2,x)\,\mathrm{d}\mu(x)
\leq 1+\kappa+CM\Lambda\delta.
\end{align*}
\end{lem}
%
%
\begin{proof}
We may assume $s_0=0$, $y_0=0$ and $\varrho=1$.
Fix $(s,y)\in(-\delta^2,0]\times\mathbf{B}(0,\delta)$.
Note that 
$\mathrm{spt}(\varphi_{(s,y),\Lambda}(-1,\cdot))
\subset\mathbf{B}(0,(2\mathbf{n}+1)\Lambda)$.
Let $x\in\mathbf{B}(0,(2\mathbf{n}+1)\Lambda)$.
Direct calculations yield
\begin{align*}
1\leq (s+1)^{-\mathbf{n}/2}
\leq (s+1)^{-\mathbf{n}}
\leq 1-Cs\leq 1+C\delta
\\
\exp\left(\left|\frac{|x|^2}{4}-\frac{|x-y|^2}{4(s+1)}\right|\right)
\leq\exp\left(C(\Lambda|y|+\Lambda^2|s|)\right)
\leq 1+C\Lambda\delta,
\end{align*}
where we used $\delta\leq (C\Lambda)^{-1}$.
Thus we have
\begin{align*}
|\Phi_{(0,0)}(-1,x)-\Phi_{(s,y)}(-1,x)|
\leq C\Lambda\delta
\\
|\varphi_{(0,0),\Lambda}(-1,x)-\varphi_{(s,y),\Lambda}(-1,x)|
\leq C\delta
\end{align*}
Combined with \eqref{ptngaussdenslema} and \eqref{ptngaussdenslemb}
this yields the result.
\end{proof}

%
%
%
%
\begin{rem}
\label{compactnessproofrem}
Here we want to derive Theorem \ref{compactnessthm} from Ilmanen's work \cite{ilmanen1}.
In case $U_i\equiv U$ the result directly follows from the proof of \cite[Thm.\ 7.1]{ilmanen1}.
Now consider the general case.
We can find a subsequence $\lambda_1:\mathbb{N}\to\mathbb{N}$
and a Brakke flow $(\nu^1_t)_{t\in [t_1,t_2]}$ in $U_1$ such that
$\lim_{j\to\infty}\mu^{\lambda_1(j)}_{t}(\phi)=\nu^1_{t}(\phi)$
for all $\phi\in\mathcal{C}_{\mathrm{c}}^0(U_1)$, for all $t\in [t_1,t_2]$.
Inductively for all $l\in\mathbb{N}$, $l\geq 2$ we can find a subsequence
$\lambda_l:\mathbb{N}\to\lambda_{l-1}[\mathbb{N}]$
and a Brakke flow $(\nu^l_t)_{t\in [t_1,t_2]}$ in $U_l$ such that $\lambda_l(1)\geq l$ and
\begin{align*}
\lim_{j\to\infty}\mu^{\lambda_l(j)}_{t}(\phi)=\nu^l_{t}(\phi)
\quad\text{for all}\;\phi\in\mathcal{C}_{\mathrm{c}}^0(U_l)
\end{align*}
for all $t\in [t_1,t_2]$.
In particular we have
$\nu^{l_2}_t\mres U_{l_0}=\nu^{l_1}_t\mres U_{l_0}$
for all $l_0\leq l_1\leq l_2$ and all $t\in [t_1,t_2]$.
Then $\mu_t(\phi):=\lim_{l\to\infty}\nu^l_t(\phi |_{U_l})$
is well defined for all $\phi\in\mathcal{C}_{\mathrm{c}}^0(U)$
and gives the desired Brakke flow on $U$.
With $\sigma(j)=\lambda_j(j)$ this establishes the result.
\end{rem}
\end{appendix}

\vspace*{1cm}
\noindent
Ananda Lahiri\\
Max Planck Institute for Gravitational Physics 
(Albert Einstein Institute)\\
Am M\"uhlenberg 1, 14476 Potsdam, Germany\\
Ananda.Lahiri@aei.mpg.de
\end{document}